\documentclass[12pt]{amsart}
\usepackage{amsmath}
\usepackage{amsfonts}
\usepackage{amssymb}
\usepackage[all,cmtip]{xy}           %xypic macro for latex2.09

\usepackage{bbm}
\usepackage{bbding}
\usepackage{txfonts}
\usepackage[shortlabels]{enumitem}
\usepackage{ifpdf}
\ifpdf
\usepackage[colorlinks,final,backref=page,hyperindex]{hyperref}
\else
\usepackage[colorlinks,final,backref=page,hyperindex,hypertex]{hyperref}
\fi
\usepackage{mathrsfs}
\usepackage{tikz-cd}
\usepackage[active]{srcltx}
\usepackage{tikz}
\usepackage{amscd}
\usepackage{bm}

\usetikzlibrary{decorations.pathreplacing}
\usetikzlibrary{calc}
\usetikzlibrary{arrows,shapes,chains}

\makeatletter

\newtheorem{thm}{Theorem}[section]
\newtheorem{prop}[thm]{Proposition}
\newtheorem{lem}[thm]{Lemma}
\newtheorem{cor}[thm]{Corollary}
\newtheorem{prop-def}[thm]{Proposition-Definition}

\theoremstyle{definition}
\newtheorem{defn}[thm]{Definition}

\newtheorem{remark}[thm]{Remark}
\newtheorem{exam}[thm]{Example}

\newcommand{\nc}{\newcommand}

 \nc{\mbibitem}[1]{\bibitem{#1}} % Use this to show number
 \nc{\mrm}[1]{{\rm #1}}

\nc{\ac}{\mathrm{\textup{!`}}}
\nc{\bs}{\bar{S}}
\nc{\pl}{\cdot}
 \nc{\la}{\longrightarrow}
\nc{\ot}{\otimes}
 \nc{\rar}{\rightarrow}

\nc{\PLA}{{\mathrm{RBS}}}
\nc{\bfk}{{\bf k}}
\nc{\Alg}{\mathsf{Alg}}
\nc{\C}{{\mathrm{C}}}
\nc{\RBA}{\mathsf{RBS}}
\nc{\RBO}{\mathsf{RBSO}}
\nc{\End}{\mrm{End}}
\nc{\Ext}{\mrm{Ext}}
\nc{\Fil}{\mrm{Fil}}
\nc{\Fr}{\mrm{Fr}}
\nc{\Frob}{\mrm{Frob}}
\nc{\Gal}{\mrm{Gal}}
\nc{\GL}{\mrm{GL}}
\nc{\Hom}{\mrm{Hom}}
\nc{\Hoch}{\mrm{Hoch}}
\nc{\hsr}{\mrm{H}}
\nc{\hpol}{\mrm{HP}}
\nc{\im}{\mrm{im}}
\nc{\Id}{\mrm{Id}}
\nc{\id}{\mrm{Id}}
\nc{\h}{\mrm{H}}
\nc{\Alt}{\mrm{Alt}}
\nc{\Irr}{\mrm{Irr}}
\nc{\incl}{\mrm{incl}}
\nc{\length}{\mrm{length}}
\nc{\NLSW}{\mrm{NLSW}}
\nc{\Lie}{\mrm{Lie}}
\nc{\mchar}{\rm char}
\nc{\mpart}{\mrm{part}}
\nc{\ql}{{\QQ_\ell}}
\nc{\qp}{{\QQ_p}}
\nc{\rank}{\mrm{rank}}
\nc{\rcot}{\mrm{cot}}
\nc{\rdef}{\mrm{def}}
\nc{\rdiv}{{\rm div}}
\nc{\rmH}{ {\mathrm{H}}}
\nc{\rtf}{{\rm tf}}
\nc{\rtor}{{\rm tor}}
\nc{\res}{\mrm{res}}
\nc{\Sh}{{\mathrm{Sh}}}
\nc{\sh}{\mathrm{\overline{Sh}}}
\nc{\SL}{\mrm{SL}}
\nc{\Spec}{\mrm{Spec}}
\nc{\sgn}{{\mathrm{sgn}}}
\nc{\tor}{\mrm{tor}}
\nc{\Tr}{\mrm{Tr}}
\nc{\tr}{\mrm{tr}}
\nc{\wt}{\mrm{wt}}
\nc{\op}{\mrm{op}}
\nc{\s}{\mrm{S}}
\nc{\ra}{\rightarrow}
\nc{\rH}{\mathrm{H}}
\nc{\RBS}{\mathfrak{RBS}}

\nc{\RBSinfty}{{\mathfrak{RBS}_\infty}}
\nc{\Int}{\mathbf{Int}}
\nc{\Lea}{\mathbf{Leaves}}
\nc{\Pare}{\mathbf{Parent}}
\def\lbb{\{\kern -.19em\{}
\def\rbb{\}\kern -.19em\}}
\nc{\BA}{{\mathbb A}}   \nc{\CC}{{\mathbb C}}
\nc{\DD}{{\mathbb D}}   \nc{\EE}{{\mathbb E}}
\nc{\FF}{{\mathbb F}}   \nc{\GG}{{\mathbb G}}
\nc{\HH}{ \mathrm{HH}}   \nc{\LL}{{\mathbb L}}
\nc{\NN}{{\mathbb N}}   \nc{\PP}{{\mathbb P}}
\nc{\QQ}{{\mathbb Q}}   \nc{\RR}{{\mathbb R}}
\nc{\TT}{{\mathbb T}}   \nc{\VV}{{\mathbb V}}
\nc{\ZZ}{{\mathbb Z}}   \nc{\TP}{\widetilde{P}}
\nc{\m}{{\mathbbm m}}

\nc{\cala}{{\mathcal A}}    \nc{\calc}{{\mathcal C}}
\nc{\cald}{\mathcal{D}}     \nc{\cale}{{\mathcal E}}
\nc{\calf}{{\mathcal F}}    \nc{\calg}{{\mathcal G}}
\nc{\calh}{{\mathcal H}}    \nc{\cali}{{\mathcal I}}
\nc{\call}{{\mathcal L}}    \nc{\calm}{{\mathcal M}}
\nc{\caln}{{\mathcal N}}    \nc{\calo}{{\mathcal O}}
\nc{\calp}{{\mathcal P}}    \nc{\calr}{{\mathcal R}}
\nc{\cals}{{\mathcal S}}    \nc{\calt}{{\Omega}}
\nc{\calv}{{\mathcal V}}    \nc{\calw}{{\mathcal W}}
\nc{\calx}{{\mathcal X}}

\nc{\fraka}{{\mathfrak a}}
\nc{\frakb}{\mathfrak{b}}
\nc{\frakg}{{\frak g}}
\nc{\frakl}{{\frak l}}
\nc{\fraks}{{\frak s}}
\nc{\frakt}{{\mathfrak{T}}}
\nc{\frakm}{{\frak m}}
\nc{\frakM}{{\frak M}}
\nc{\frakp}{{\frak p}}
\nc{\frakW}{{\frak W}}
\nc{\frakX}{{\frak X}}
\nc{\frakS}{{\frak S}}
\nc{\frakA}{{\frak A}}
\nc{\frakC}{{\frak{C}}}
\nc{\frakx}{{\frakx}}
\nc{\red}{\color{red}}
\nc{\blue}{\color{blue}}

\nc{\mlabel}[1]{\label{#1}}  % Use this to suppress names
\nc{\mcite}[1]{\cite{#1}}  % Use this to suppress names
\nc{\mref}[1]{\ref{#1}}  % Use this to suppress names
\nc{\meqref}[1]{\eqref{#1}}  % Use this to suppress names

\nc{\kai}[1]{\textcolor{blue}{\underline{Kai:}#1 }}

\begin{document}

\title[Homotopy  Rota-Baxter systems]{Deformation theory  and Koszul duality for  Rota-Baxter systems}

\author{Yufei Qin, Kai Wang and Guodong Zhou}

\address{Yufei Qin,
  School of Mathematical Sciences,   Key Laboratory of Mathematics and Engineering Applications (Ministry of Education),   Shanghai Key laboratory of PMMP,
  East China Normal University,
 Shanghai 200241,
   China}
   \email{290673049@qq.com}

\address{Kai Wang, School of Mathematical Sciences, University of Science and Technology of China, Hefei, Anhui Provience 230026, China}

	\email{wangkai17@ustc.edu.cn}

\address{Guodong Zhou,
  School of Mathematical Sciences,  Key Laboratory of Mathematics and Engineering Applications (Ministry of Education),   Shanghai Key laboratory of PMMP,
  East China Normal University,
 Shanghai 200241,
   China}
   
\email{gdzhou@math.ecnu.edu.cn}

\date{\today}

\begin{abstract}
This paper  investigates  Rota-Baxter systems in the sense of Brzezi\'nski   from the perspective of operad theory. The minimal model of the Rota-Baxter system operad is constructed,   equivalently  a concrete construction of its Koszul dual homotopy cooperad is given. The concept of homotopy Rota-Baxter systems and the $L_\infty$-algebra that governs   deformations of a Rota-Baxter system are derived from the Koszul dual homotopy cooperad.  The notion of infinity-Yang-Baxter pairs is introduced, which is a higher-order generalization of the traditional Yang-Baxter pairs.  It is shown that  a homotopy Rota-Baxter system structure on the endomorphism algebra of a graded space is equivalent to an  associative infinity-Yang-Baxter pair on this graded algebra, thereby generalizing the classical correspondence between Yang-Baxter pairs and Rota-Baxter systems.

\end{abstract}

\subjclass[2020]{
16E40   %(Co)homology of rings and associative algebras (e.g., Hochschild, cyclic, dihedral, etc.)
16S80   %Deformations of associative rings [See also 13D10, 14D15]
%12H05   %Differential algebra
%12H10   %difference algebra
%16S70    %Extensions of associative rings by ideals
%16W25    %Derivations, actions of Lie algebras
17B38   %Yang-Baxter equations and Rota-Baxter operators
18M65  %Non-symmetric operads, multicategories, generalized multicategories
18M70  %Algebraic operads, cooperads, and Koszul duality
}

\keywords{homotopy Rota-Baxter system; Koszul dual homotopy cooperad; minimal model; operad; Rota-Baxter system; Yang-Baxter pair; infinity-Yang-Baxter pair}

\maketitle

\tableofcontents

\allowdisplaybreaks

\section*{Introduction}
%\subsection{Rota-Baxter systems and associative Yang-Baxter pairs}\

The concept of the “classical Yang-Baxter equation” (CYBE) originated from the pioneering work of Faddeev and his collaborators on inverse scattering theory \cite{FT79, FT87}. Later, Belavin and Drinfeld systematically studied the solutions of the CYBE. This equation also arises as a special case of the Schouten bracket in differential geometry \cite{Sch40} and is interpreted as the “classical limit” of the quantum Yang-Baxter equation \cite{Bel81}. Since the 1980s, the CYBE has played a central role in various fields of mathematics and mathematical physics, including symplectic geometry, Poisson geometry, integrable systems, quantum groups, and quantum field theory (see \cite{CP94} and references therein). Initially formulated in tensor form, the CYBE has also been explored in operator form, following significant contributions from Semenov-Tian-Shansky \cite{STS83}, Kupershmidt \cite{Kup99}, Aguiar \cite{Agu00b}, Bai \cite{Bai07}, and others. This operator formulation is closely connected to the concept of the Rota-Baxter operator \cite{Bax60, Rot69, Rot95}.

Rota-Baxter algebras, originally introduced as Baxter algebras, first appeared in Baxter’s study of probability theory \cite{Bax60}. The subject was further developed by Rota \cite{Rot69} and Cartier \cite{Car72}, among others, leading to the term “Rota-Baxter algebras”.  The field experienced renewed interest through the works of Guo et al. \cite{Agu01, GK00a, GK00b}, resulting in numerous applications and connections to various mathematical disciplines. Rota-Baxter algebras have been utilized in combinatorics \cite{Rot95}, renormalization in quantum field theory \cite{CK00}, multiple zeta values in number theory \cite{GZ08}, operad theory \cite{Agu01, BBGN13}, Hopf algebras \cite{CK00}, and the Yang-Baxter equations \cite{Bai07}, among other areas. In 1983, Semenov-Tian-Shansky \cite{STS83} demonstrated that a solution of the CYBE in a Lie algebra corresponds to a Rota-Baxter operator of weight zero on the same Lie algebra. Kupershmidt \cite{Kup99} later established that a skew-symmetric solution of the CYBE provides a relative Rota-Baxter operator. In 2000, Aguiar introduced the associative Yang-Baxter equation \cite{Agu00b} and proved that a solution to this equation equips an associative algebra with a Rota-Baxter operator \cite{Agu00a}. Bai \cite{Bai07} systematically studied operator-form solutions of the Yang-Baxter equations, while Gubarev \cite{Gub21} and Zhang, Gao, and Zheng \cite{ZGZ18p} independently demonstrated that solutions to the associative Yang-Baxter equation in matrix algebras correspond to Rota-Baxter algebra structures on those algebras.

The concepts of Rota-Baxter systems and associative Yang-Baxter pairs, which serve as binary-system analogs of Rota-Baxter operators and associative Yang-Baxter equations, respectively, were introduced by Brzeziński \cite{Brz16a} in his effort to understand the Jackson $q$-integral as a Rota-Baxter operator. These notions broaden and deepen the connections among three algebraic frameworks: Rota-Baxter algebras, dendriform algebras \cite{Lod01}, and infinitesimal bialgebras \cite{Agu00b}. 
 {Furthermore, Brzezi\'nski \cite{Brz16a}  proved that an associative Yang-Baxter pair can give a Rota-Baxter system structure within an associative algebra.}  Subsequently, Das \cite{Das20p} introduced generalized Rota-Baxter systems on $A_\infty$-algebras and $A_\infty$-bimodules.   Meanwhile,   Liu,   Wang and   Yin \cite{LQWY22p} constructed the   cohomology theory of Rota-Baxter systems which governs the simultaneous deformations of operators and algebra structures in Rota-Baxter systems. The Rota-Baxter system has proven to be highly interrelated with a variety of algebraic structures.  Qiu and Chen established linear bases for free Rota-Baxter systems, with which  they found that there is a left counital Hopf algebra structure on a free Rota-Baxter  system \cite{QC18}.    Li and Wang  \cite{LW23}   introduced a variation of Rota-Baxter systems on Lie algebras and they studied the integration of such structures on Lie groups.  Afterwards, they proposed the notion of such variant Rota-Baxter systems on  Hopf algebras \cite{LW23p}.  Ma and Li \cite{ML21p} introduced the notion of associative (BiHom-)Yang-Baxter pairs with weight, which serve as the BiHom-associative algebraic counterpart to the associative Yang-Baxter pair; see also \cite{GLLSW22}. For other developments, see
\cite{Brz16b, Brz22, MMS21, ZZG22p}.% { Furthermore, Zhang, Zhang, and Gao \cite{ZZG} introduced the notion of $\Omega$-family Rota-Baxter systems .}

The deformation theory of mathematical structures has evolved through the foundational work of numerous mathematicians, including Gerstenhaber, Nijenhuis, Richardson, Deligne, Schlessinger, Stasheff, Goldman, Millson, and others. Central to this theory is the principle that the deformation of a mathematical object can be described by a differential graded (dg) Lie algebra or, more generally, an $L_\infty$-algebra associated with the object. The complex underlying this algebra is known as the deformation complex. Lurie \cite{Lur} and Pridham \cite{Pri10} have rigorously formalized this principle into a theorem in characteristic zero, articulated within the framework of infinity categories. Consequently, a fundamental problem in the field is the explicit construction of the dg Lie algebra or $L_\infty$-algebra that governs the deformation theory of a given mathematical object.

Another important aspect of studying algebraic structures is exploring their homotopy versions, similar to how $A_\infty$-algebras relate to traditional associative algebras. The goal is often to find a minimal model for the operad that governs the algebraic structure. When the operad is Koszul, Koszul duality for operads \cite{GK94, GJ94}, developed by Ginzburg and Kapranov, provides a homotopy version of the algebraic structure through the cobar construction of the Koszul dual cooperad, serving as a minimal model. However, constructing minimal models becomes much harder when the operad is not Koszul, and only a few examples exist \cite{CGWZ24, WZ22}.
For example, G\'alvez-Carrillo, Tonks, and Vallette \cite{GCTV12} created a cofibrant resolution of the Batalin-Vilkovisky operad using inhomogeneous Koszul duality theory, but their resolution is not minimal. Later, Drummond-Cole and Vallette \cite{DCV13} found a minimal model that is a deformation retract of the previous cofibrant resolution. Dotsenko and Khoroshkin \cite{DK13} also constructed cofibrant resolutions for shuffle monomial operads using the inclusion-exclusion principle and suggested a method to find cofibrant resolutions for operads defined by a Gr\"obner basis \cite{DK10} by deforming those of the corresponding monomial operads.
 	
 	The purpose of this paper is to study the homotopy theory and deformation theory of Rota-Baxter systems from the perspective of operad theory. In recent years, there have been many works on the deformation theory of Rota-Baxter algebras. Tang, Bai, Guo, and Sheng \cite{TBGS19} developed the deformation and cohomology theory of \(\mathcal{O}\)-operators (also called relative Rota-Baxter operators) on Lie algebras. Das \cite{Das20} developed a similar theory for Rota-Baxter algebras of weight zero. Lazarev, Sheng, and Tang \cite{LST21} succeeded in establishing the deformation and cohomology theory of relative Rota-Baxter Lie algebras of weight zero, finding applications to triangular Lie bialgebras. They determined the \(L_\infty\)-algebra that controls deformations of a relative Rota-Baxter Lie algebra and introduced the notion of a homotopy relative Rota-Baxter Lie algebra. Later, Das and Misha also determined the \(L_\infty\)-structures controlling the deformation of Rota-Baxter operators of weight zero on associative algebras \cite{DM20}. Wang and Zhou \cite{WZ22} constructed the minimal model of the operad governing Rota-Baxter algebras of arbitrary weight, from which they provided the controlling \(L_\infty\)-algebra and cochain complex for the deformations of Rota-Baxter algebras of arbitrary weight.
 	The deformation theory of Rota-Baxter systems has rarely been studied. Das and Guo \cite{DG23} constructed a graded Lie algebra  controlling the deformations of the operator part in Rota-Baxter systems on Leibniz algebras.
 In this paper, we will construct the minimal model and Koszul dual homotopy cooperad of the operad for Rota-Baxter systems explicitly. From this minimal model, the homotopy version of Rota-Baxter systems, the controlling \(L_\infty\)-algebra, and the deformation cochain complex of Rota-Baxter systems can be naturally derived. Moreover, we will introduce the notion of  infinity-Yang-Baxter pairs, which can be considered as a higher version of the classical Yang-Baxter pairs, and we will study the relationship between  infinity-Yang-Baxter pairs and homotopy Rota-Baxter systems.

The paper is organized as follows: In Section~\ref{Sec: Rota-Baxter systems and associative Yang-Baxter pairs}, we begin by fixing some notations in Subsection~\ref{Subsect: Rota-Baxter systems}. We then recall the definitions of Rota-Baxter systems and associative Yang-Baxter pairs. In particular, we prove that solutions of associative Yang-Baxter pairs on $ M_n(\bfk) $ are equivalent to Rota-Baxter system operators on $ M_n(\bfk) $; see Theorem~\ref{YBP and RBS}. In Section~\ref{Section: Koszul dual cooperad}, we construct a homotopy cooperad, which serves as the Koszul dual of the operad for Rota-Baxter systems. Building on this, in Section~\ref{Section: minimal model RBSA}, we show that the cobar construction of the homotopy cooperad from Section~\ref{Section: Koszul dual cooperad} provides the minimal model for Rota-Baxter systems. For clarity, some technical details regarding the minimal model construction are presented in Appendix~\ref{Appendix: the minimal model for monomimal operad}. In Section~\ref{Section: Homotopy Rota-Baxter systems and infinity-Yang-Baxter pairs}, we introduce the notion of homotopy Rota-Baxter systems based on the minimal model of Rota-Baxter systems and further develop the concept of infinity-Yang-Baxter pairs. We also prove that a homotopy Rota-Baxter system structure on the endomorphism algebra of a graded space is equivalent to an associative infinity-Yang-Baxter pair on this graded algebra. Finally, in Section~\ref{Section: Linfinty algebras}, we explicitly describe the controlling  $L_\infty$-algebra and the deformation cochain complex for deformations of Rota-Baxter systems.
 	
Throughout this paper, let $\bfk$ be a field of characteristic $0$.  Except specially stated,  vector spaces are  $\bfk$-vector spaces and  all    tensor products and Hom-spaces are taken over $\bfk$.

\bigskip

\section{Rota-Baxter systems   and associative Yang-Baxter pairs}\label{Sec: Rota-Baxter systems   and associative Yang-Baxter pairs}

\subsection{Notations}\label{Subsec:Preliminaries}\

%\subsection{Notations}
A   (homologically) graded     space  is a $\mathbb{Z}$-indexed family of vector spaces $V=\{V_n\}_{n\in \ZZ}$.    Elements of $\cup_{n\in \ZZ} V_n$ are   homogeneous with
degree, written as   $|v|=n$,  for  $v\in V_n$.

Given two graded   spaces $V$ and $W$, a linear map $f: V\to W$ such that  $f(V_n)\subseteq  W_{n+r}$ for all $n$ is
called a graded map  of degree $r$ and denote  $|f|=r$.
Write  $$\mathrm{Hom}(V, W)_r=\prod_{p\in \ZZ} \Hom_\bfk(V_p, W_{p+r})$$  for  the space of graded maps of degree $r$.
The graded space  $\mathrm{Hom}(V, W)$ is defined to be $\{\mathrm{Hom}(V, W)_r\}_{r\in \ZZ}$.

The  tensor product  $V\otimes W$ of  two graded   spaces $V$ and $W$
is given by
\[(V\otimes W)_n =\bigoplus_{p+q=n} V_p \otimes W_q.\]
We will use Sweedler's notation to represent elements in the tensor products of graded spaces. Let \( V^1 \otimes \cdots \otimes V^n \) be the tensor product of graded spaces \( V^1, \ldots, V^n \). An element in this tensor product can be written as \( r = \sum_{i_1, \ldots, i_n} r^{[1]}_{i_1} \otimes \cdots \otimes r^{[n]}_{i_n} \) with \( r_{i_k}^{[k]} \in V^k \). For simplicity, we will omit the subscripts \( i_k \) and instead write:
	\[ r = \sum r^{[1]} \otimes \cdots \otimes r^{[n]} .\]

Let $(A,\cdot)$ be a graded  algebra.  For each $n\geqslant 1$, we define a map $\mathscr{F}$  from $A^{\otimes n+1}$ to $ \Hom(A^{\otimes n},A)$ as
\[ \mathscr{F} (a_1\otimes\cdots\otimes a_{n+1})(x_1,\ldots,x_n):= (-1)^{ \sum\limits_{k=1}^{n}\sum\limits_{j=k+1}^{n+1}|x_k||a_j|}a_1\cdot x_1\cdots x_n\cdot a_{n+1},\]
where $x_1,\ldots,x_n,a_1,\ldots,a_{n+1}\in A$.

\begin{lem}\label{[] iso}
	Let $V$ be a finite dimensional graded space.	For each $k\geqslant 1$, $\mathscr{F}:\End(V)^{\otimes k+1}\rightarrow \Hom(\End(V)^{\otimes n},\End(V)) $ is an isomorphism.
\end{lem}

\begin{proof}
	Let $\{e_{p} ^{~q}\}_{1\leqslant p, q\leqslant n}$ be the canonical basis of the matrix algebra $\End(V)$, satisfying
	\[e_{i} ^{~j}e_{k} ^{~l}=\delta_k^je_{i} ^{~l},\text{ for all } 1\leq i,j,k,l\leq n,\]
	where $ \delta^j_k$ is the Kronecker symbol.
	%, for some $ 1\leq \tilde{i},\tilde{p}_1,\ldots \tilde{p}_n,\tilde{q}_1,\ldots,\tilde{q}_n,\tilde{l}\leq n$.
	
%	We firstly show that for any nontrivial element $r_{n+1}\in \End(V)^{\otimes n+1}$, $[r_{n+1}]$ is an nontrivial operator.  We can write $r_{n+1}\in \End(V)^{\otimes n+1}$ as
%	$$r_{n+1}=\sum_{1\leq p_1,\ldots,p_n,p_{n+1},q_1,\ldots,q_0,q_n\leq n} r_{p_1,\ldots,p_n,p_{n+1}}^{q_0,q_1,\ldots,q_n}e_{q_0} ^{~p_1}\ot\cdots\ot e_{q_n} ^{~p_{n+1}},$$
%	with coefficients $r_{p_1,\ldots,p_n,p_{n+1}}^{q_0,q_1,\ldots,q_n}\in \textbf{k}$.
%	Since $r_{k+1}\neq 0$,  there exists a term $e_{\tilde{q_0}} ^{~\tilde{p}_1}\ot\cdots\ot e_{\tilde{q}_n} ^{~\tilde{p_{n+1}}}$ in $r_{k+1}$ with nontrivial $r_{\tilde{p}_1,\ldots,\tilde{p_{n+1}}}^{\tilde{q_0},\ldots,\tilde{q}_n}$. Thus,   $$[r_{k+1}](e_{\tilde{p}_1} ^{~\tilde{q}_1}\ot\cdots\ot e_{\tilde{p}_n} ^{~\tilde{q}_n})=\sum_{1\leq q_0,p_{n+1}\leq n}(-1)^{ \sum_{k=1}^{n}\sum_{j=k}^{n}|e_{\tilde{p}_k} ^{~\tilde{q}_k}||e_{q_j} ^{~p_{j+1}}|} r_{\tilde{p}_1,\ldots,\tilde{p}_n,p_{n+1}}^{q_0,\tilde{q}_1,\ldots,\tilde{q}_n}e_{q_0} ^{~p_{n+1}} .$$
%	Since $r_{\tilde{p}_1,\ldots,\tilde{p_{n+1}}}^{\tilde{q_0},\ldots,\tilde{q}_n}\neq 0$,  $[r_{k+1}](e_{\tilde{p}_1} ^{~\tilde{q}_1}\ot\cdots\ot e_{\tilde{p}_n} ^{~\tilde{q}_n})\neq 0$. Thus, $[r_{k+1}]$ is an nontrivial operator, and $[~] $ is an injection.
	
	Let  $T:\End(V)^{\otimes k}\rightarrow \End(V)$ be arbitrary operator. In terms of the canonical basis, $T$ can be represented as:	for all $ 1\leq p_1,\ldots,p_n,q_1,\ldots,q_n,q_{n+1}\leq n$,
	$$ T(e_{p_1} ^{~q_1}\ot\cdots\ot e_{p_n} ^{~q_n})=\sum\limits_{1\leq q_0,p_{n+1}\leq n}(-1)^{ \sum\limits_{k=1}^{n}\sum_{j=k}^{n}|e_{p_k} ^{~q_k}||e_{q_j} ^{~p_{j+1}}|} r_{p_1,\ldots,p_n,p_{n+1}}^{q_0,q_1,\ldots,q_n}e_{q_0} ^{~p_{n+1}},$$
	with coefficients $r_{p_1,\ldots,p_n,l}^{q_0,q_1,\ldots,q_n}\in \textbf{k}$. For $r_{n+1}=\sum r_{p_1,\ldots,p_n,l}^{q_0,q_1,\ldots,q_n}e_{q_0} ^{~p_1}\ot\cdots\ot e_{q_n} ^{~p_{n+1}}$, we can calculate that  $T=\mathscr{F}(r_{n+1}) $. Thus, $\mathscr{F}$ is a surjection. Since $V$ is finite dimensional, $\mathscr{F}$ is an isomorphism.
\end{proof}

%\begin{Exo} Recall that a graded vector space $V$ is locally finite if $\mathrm{dim}(V_n)<\infty, \forall n\in \ZZ$.

The  suspension  of a graded space $V$ is the graded space  $sV $ with $(sV)_n=V_{n-1}$ for any $n\in \ZZ$.
Write $sv\in (sV)_n$ for $v\in V_{n-1}$.
The map $s: V\to sV, v\mapsto sv$ is a graded map of degree $1$.
One can also define the  desuspension  of $V$ which is denoted by $s^{-1}V$ with  $(s^{-1}V)_n=V_{n+1}$. Write $s^{-1}v\in (s^{-1}V)_n$ for $v\in V_{n+1}$.  and
$s^{-1}: V\to s^{-1}V, v\mapsto s^{-1}v$ is a graded map of degree $-1$.

We will employ Koszul sign rule to determine signs,   that is, when   exchanging the positions of two graded objects in an expression, we have  to multiply the expression by a power of $-1$ whose  exponent is  the product of their degrees.

%For $v_1, \dots, v_n\in V$, write
%$v_{1, n}:=v_1\ot  \dots \ot v_n\in V^{\ot n}$ and also $sv_{1, n}=sv_1\ot  \dots \ot sv_n\in (sV)^{\ot n}$.
%There is the obvious map $s: V\to sV, v\mapsto sv$ which is of degree $1$ and has the inverse $s^{-1}: sV\to V, sv\mapsto v$.

Let  $n\geq 1$. Let $\mathbb{S}_n$ denote the symmetric group in $n$ variables.
For $0\leq i_1, \dots, i_r\leq n$ with $i_1+\cdots+i_r=n$,  $\Sh(i_1, i_2,\dots,i_r)$ is the   set of $(i_1,\dots, i_r)$-shuffles, i.e., those permutation $\sigma\in \mathbb{S}_n$ such that
$$\sigma(1)<\sigma(2)<\dots<\sigma(i_1),  \ \sigma(i_1+1)< \dots<\sigma(i_1+i_2),\ \dots,\
\sigma(i_1+\cdots+i_{r-1}+1)< \cdots<\sigma(n).$$

Let $V=\oplus_{n\in \mathbb{Z}} V_n$ be a graded vector space. Recall that   the graded symmetric algebra $S(V)$ of $V$ is defined to be the quotient of the tensor algebra $T(V)$ by   the graded two-sided ideal $I$   generated by
$x\ot y -(-1)^{|x||y|}y\ot x$ for all homogeneous elements $x, y\in V$. For $x_1\ot\cdots\ot x_n\in V^{\ot n}\subseteq T(V)$, write $ x_1\odot x_2\odot\dots\odot x_n$ its image in $S(V)$.
For homogeneous elements $x_1,\dots,x_n \in V$ and $\sigma\in \mathbb{S}_n$, the Koszul sign $\varepsilon(\sigma;  x_1,\dots, x_n)$ is defined by
$$ x_1\odot x_2\odot\dots\odot x_n=\varepsilon(\sigma;  x_1,\dots,x_n)x_{\sigma(1)}\odot x_{\sigma(2)}\odot\dots\odot x_{\sigma(n)}\in S(V);$$
we also define
$$ \chi(\sigma;  x_1,\dots,x_n)= \sgn(\sigma)\ \varepsilon(\sigma;  x_1,\dots,x_n),$$
where $\sgn(\sigma)$ is the sign of the permutation $\sigma$.

 \subsection{Rota-Baxter systems vs associative Yang-Baxter pairs}\label{Subsect: Rota-Baxter systems}\ 
 
In this subsection, we first recall the definitions of Rota-Baxter systems and associative Yang-Baxter pairs. We will then show that these two notions, Rota-Baxter systems and associative Yang-Baxter pairs, are equivalent in the context of matrix algebras.

\begin{defn} \label{definition-RBS} \cite[Definition 2.1]{Brz16a}	A triple $((A,\mu),R,S)$ consisting of an associative  algebra $(A,\mu)$  and two $\bfk$-linear operators $ R, S : A\to A$ is called a Rota-Baxter system if, for all $a,b\in A$, the following equations hold:
		\begin{gather}
			\mu\big(R(a)\ot R(b)\big)=R\big(\mu(R(a)\ot b)+\mu(a\ot S(b))\big), \label{RBS1}\\
			\mu\big(S(a)\ot S(b)\big)=S\big(\mu(R(a)\ot b)+\mu(a\ot S(b))\big). \label{RBS2}
		\end{gather}
		And we call the pair of operators $(R, S)$ a coupled Rota-Baxter operator on the associative algebra $(A,\mu)$.
%		\item[(2)]\label{definition-RRBS} {\rm(see \cite[Definition 2.1]{Das} )} A relative (generalized) Rota-Baxter system on  $A$ with respect to
%		the bimodule $(M,\rhd, \lhd)$ consists of a pair $(R, S) $ of linear operators $ R, S : M\to A$, for all $u,v\in A$,
%		\begin{gather}
%			R(u)R(v)=R(R(u)\rhd v+u\lhd S(v)), \\
%			S(u)S(v)=S(R(u)\rhd v+u \lhd S(v)).
%		\end{gather}
%		
\end{defn}

\begin{defn}\label{definition-AYBEP} \cite[Definition 3.1]{Brz16a}
	Let $A$ be a unitary associative algebra. An associative Yang-Baxter pair is a 	pair of elements $$r=\sum  r^{[1]}\otimes r^{[2]},s=\sum  s^{[1]}\otimes s^{[2]}\in A\otimes A$$ that satisfy the following equations
	\begin{align}
		& r^{13} r^{12}-r^{12} r^{23}+s^{23} r^{13}=0, \label{Eq: YBP1}\\
		& s^{13} r^{12}-s^{12} s^{23}+s^{23} s^{13}=0,\label{Eq: YBP2}
	\end{align}
	where $r^{12}=r \otimes 1,$ $r^{13}=\sum   r^{[1]}\otimes1 \otimes r^{[2]}$, $r^{23}=1 \otimes r$, etc.
\end{defn}

\begin{prop}\label{prop-from YBEP to RBS} \cite[Proposition 3.4]{Brz16a}
	 Let $(r, s)$ be an associative Yang-Baxter pair in a unitary algebra $(A,\cdot)$. Define two operators $ \mathscr{F}(r) , \mathscr{F}(s) : A \rightarrow A$ as
	 $$ \mathscr{F}( r) (a)=\sum r^{[1]} a r^{[2]}, \quad  \mathscr{F}(s) (a)=\sum s^{[1]} a s^{[2]} .
	 $$
	 	 Then $((A,\cdot), \mathscr{F}(r), \mathscr{F}(s))$ is a Rota-Baxter system.
\end{prop}
\begin{proof}For $a,b\in A$, we have
	 \begin{eqnarray*}
	 &&	\Big(  \mu\circ\big(\mathscr{F}(r)\ot \mathscr{F}(r)  \big)-\mathscr{F}(r)\circ\mu\circ\big(\mathscr{F}(r)\ot  \id\big)-\mathscr{F}(r)\circ\mu\circ\big(\id\ot \mathscr{F}(s)\big)\Big) (a\otimes b)\\
	 =	&&\sum r^{[1]}\cdot a\cdot r^{[2]}\cdot {r'}^{[1]}\cdot b\cdot {r'}^{[2]}-\sum r^{[1]}\cdot  {r'}^{[1]}\cdot a\cdot  {r'}^{[2]}\cdot b\cdot r^{[2]}-\sum r^{[1]}\cdot a\cdot {s'}^{[1]}\cdot b\cdot {s'}^{[2]}\cdot r^{[2]}\\
	 =&& \mathscr{F}(r^{13} r^{12}-r^{12} r^{23}+s^{23} r^{13})(a\otimes b).
	 \end{eqnarray*}
 Similarly, we can also show that $$\ \mu\circ\big(\mathscr{F}(s)\ot \mathscr{F}(s)  \big)-\mathscr{F}(s)\circ\mu\circ\big(\mathscr{F}(r)\ot  \id \big)-\mathscr{F}(s)\circ\mu\circ\big(\id\ot \mathscr{F}(s)\big)=\mathscr{F}(s^{13} r^{12}-s^{12} s^{23}+s^{23} s^{13}).$$
\end{proof}

Similar to    \cite[Theorem 3.4]{Gub21} and   \cite[Theorem 4.12]{ZGZ18p}, we have the following result.

\begin{thm}\label{YBP and RBS}
	The map  $\varphi: (r,s )\rightarrow (\mathscr{F}(r), \mathscr{F}(s))$ is a bijection between the set of associative Yang-Baxter pairs on $M_n(\bfk)$ and the set of coupled Rota-Baxter operators   on $M_n(\bfk)$.
\end{thm}

{
	\begin{proof}
		
		 Let $(R,S)$ be a pair of operators on $M_n(\bfk)$. In terms of the canonical basis,  $R$ and $S$  can be represented as:
		$$R(e_{p} ^{~q})=\sum_{il} r_{pl}^{iq} e_{i}^{~ l} \text{ and }S(e_{p} ^{~q})=\sum_{i l} s_{pl}^{iq} e_{i}^ {~l} ,\text{ for all } 1\leq p,q\leq n, $$
		where coefficients $r_{pl}^{iq}, s_{pl}^{iq}\in \bfk$.

		Set $r=\sum\limits_{i, j, k, l} r_{jl}^{ik} e_{i}^{~j} \otimes e_{k}^{~l}$ and $s=\sum\limits_{i, j, k, l} s_{jl}^{ik} e_{i}^{~j} \otimes e_{k}^{~l}$. Then we have $R=\mathscr{F}(r)$ and $S=\mathscr{F}(s)$. By Lemma~\ref{[] iso} and Proposition~\ref{prop-from YBEP to RBS}, $(M_n(\bfk),\mathscr{F}(r),\mathscr{F}(s))$ is Rota-Baxter system if and only if $(r,s)$ is    an associative Yang-Baxter pair in the unitary algebra $M_n(\bfk)$.

	\end{proof}

\bigskip

\section{The Koszul dual homotopy cooperad}\label{Section: Koszul dual cooperad}

In this section,   we will  construct a homotopy cooperad, and we will prove that the cobar construction of this homotopy cooperad is exactly the minimal model of the operad for Rota-Baxter system in Section~\ref{Section: minimal model RBSA}.  Thus this homotopy cooperad can be considered as the Koszul dual of the operad for Rota-Baxter system.

For basic knowledge about homotopy (co)operads,  we refer the reader to   \cite{Mar96, MV09a, MV09b, DP16}
or \cite[Section 3.1]{CGWZ24}.

%Recall that the  operad for Rota-Baxter   algebras of weight $\lambda$, denoted by $\RB$,  is generated by a unary operator $T$ and a binary operatore $\mu$ with the operadic relation generated by $$ \mu\circ_1\mu-\mu\circ_2\mu\quad \mathrm{and}\quad (\mu\circ_1T)\circ_2T-(T\circ_1\mu)\circ_1T-(T\circ_1\mu)\circ_2T-\lambda T\circ_1\mu.$$

%\smallskip

%Now, let's construct the homotopy cooperad ${\RB^\ac}$ explicitly and we will prove that its cobar construction is exactly the minimal model for the operad $\RB$.

Define a graded collection $\mathscr{S}(\RBS^\ac)$ by   $$\mathscr{S}(\RBS^\ac)(n)=\bfk u_{n}\oplus \bfk v_{n}\oplus \bfk w_{n}, \ \mathrm{with}\   |u_{n}|=0, |v_{n}|=|w_{n}|=1, $$ for $n\geqslant 1$. Now, we put a coaugmented homotopy cooperad structure on   $\mathscr{S}(\RBS^\ac)$. Firstly, consider  trees of arity $n\geqslant 1$ in the following list:
\begin{itemize}[keyvals]
		\item[(I)] Planar trees of weight two:     for each $1\leqslant j\leqslant n$ and   $1\leqslant i\leqslant n-j+1$, there exists such a tree, which can be visualized as
	
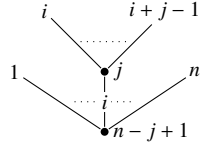
\begin{eqnarray*}
		\begin{tikzpicture}[scale=0.8,descr/.style={fill=white}]
			\tikzstyle{every node}=[thick,minimum size=3pt, inner sep=1pt]

			\node(r) at (0,-0.5)[minimum size=0pt,circle]{};
			\node(v0) at (0,0)[fill=black, circle,label=right:{\tiny $ n-j+1$}]{};
			\node(v1-1) at (-1.5,1){\tiny$ 1 $};
			\node(v1-2) at(0,1)[fill=black,circle,label=right:{\tiny $\tiny j$}]{};
			\node(v1-3) at(1.5,1){\tiny$ n $};
			\node(v2-1)at (-1,2){\tiny$ i $};
			\node(v2-2) at(1,2){\tiny$ i+j-1 $};
			\draw(v0)--(v1-1);
			\draw(v0)--(v1-3);
			\draw(v1-2)--(v2-1);
			\draw(v1-2)--(v2-2);
			\draw[dotted](-0.4,1.5)--(0.4,1.5);
			\draw[dotted](-0.5,0.5)--(-0.1,0.5);
			\draw[dotted](0.1,0.5)--(0.5,0.5);
			\path[-,font=\scriptsize]
			(v0) edge node[descr]{{\tiny$i$}} (v1-2);
		\end{tikzpicture}
	\end{eqnarray*}
	
\item[(II)] Trees of weights $k+1\geqslant 3$ and of ``height" two: in these trees,  the first vertex in the planar order has arity $k$, the vertex connected to the $t$-th leaf of the first vertex has arity $r_t$ for each $1\leqslant t\leqslant k$ (so $2\leqslant k\leqslant n$ and $r_1+\cdots +r_k=n$); these trees have the following picture:
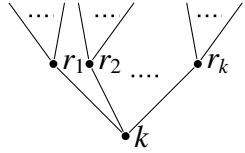
\begin{eqnarray*}
	\begin{tikzpicture}[scale=0.8,descr/.style={fill=white}]
		\tikzstyle{every node}=[thick,minimum size=3pt, inner sep=1pt]
		\node(v0) at (0,-1.5)[circle, fill=black,label=right:$k$]{};
		\node(v1) at(-1.2,-0.3)[circle, fill=black,label=right:$r_1$]{};
		\node(v1-1) at(-2,0.8){};
		\node(v1-2) at (-1,0.8){};
		\node(v3) at (1.2,-0.3)[circle, fill=black,label=right:$r_k$]{};
		\node(v3-1) at (1,0.8){};
		\node(v3-2) at (2,0.8){};
		\node(v2-1) at (-0.6,-0.3) [circle,fill=black,label=right:$r_2$]{};
		\node(v2-1-1) at (-0.8,0.8){};
		\node(v2-1-2) at (0.2, 0.8){};
		%\draw [dotted,line width=1pt] (-0.7,-0.5)--(-0.2,-0.5);
		\draw [dotted,line width=1pt] (0.1,-0.5)--(0.6,-0.5);
		\draw [dotted,line width=1pt] (-0.5,0.5)--(-0.1,0.5);
		\draw [dotted,line width=1pt] (-1.6,0.5)--(-1.2, 0.5);
		\draw [dotted, line width=1pt] (1.2,0.5)--(1.6,0.5);
		\draw        (v0)--(v1);
		\draw         (v0)--(v3);
		\draw         (v1)--(v1-1);
		\draw          (v1)--(v1-2);
		\draw        (v2-1)--(v2-1-1);
		\draw        (v2-1)--(v2-1-2);
		\draw        (v3)--(v3-1);
		\draw        (v3)--(v3-2);
		\draw        (v0)--(v2-1);
		%\path[-,font=\scriptsize]
		%(v0) edge node[descr]{{\tiny$t$}} (v2-1);
	\end{tikzpicture}
\end{eqnarray*}
where $k\geqslant 2$, $r_t\geqslant 1$ for all $k\geqslant t\geqslant 1$.
\item[(III)]{  Trees of weights $q+1\geqslant 3$ and of ``height" three and there exists a unique vertex in the first two levels: in these trees, there are numbers $2\leqslant p\leqslant n$, $1\leqslant i\leqslant r_1$ and $r_j\geqslant 1$ for all $1\leqslant j\leqslant p$ such that the first vertex has arity $r_1$, the second vertex has arity $p$ and is connected to the $i$-th leaf of the first vertex and the other vertices connected to the $1$-th (resp. $2$-th, $\ldots$ , $(j-1)$-th, $(j+1)$-th, $\ldots$ , $p$-th) leaf of the second vertex and with arity $r_2$ (resp. $r_3, \dots, r_p$), so $r_1+\cdots+r_p=n$; these trees are drawn in this way:
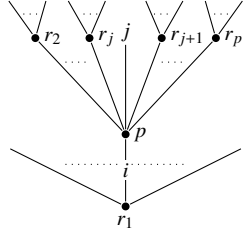
\begin{eqnarray*}
	\begin{tikzpicture}[scale=0.8,descr/.style={fill=white}]
		\tikzstyle{every node}=[thick,minimum size=3pt, inner sep=1pt]
		\node(v0) at (0,0)[circle,fill=black,label=below:{\tiny $r_1$}]{};
		\node(v1-1) at (-2,1){};
		\node(v1-2) at(0,1.2)[circle,fill=black,label=right:{\tiny $p$}]{};
		\node(v1-3) at(2,1){};
		\node(v2-1) at(-1.5,2.8)[circle,fill=black,label=right:{\tiny $r_2$}]{};
		\node(v2-2) at (-0.6, 2.8)[circle,fill=black,label=right:{\tiny $r_{j}$}]{};
		\node(v2-3) at (0,2.9){\tiny $j$}{};
		\node(v2-4) at(1.5,2.8)[circle,fill=black,label=right:{\tiny $r_p$}]{};
		\node(v2-5) at(0.6,2.8)[circle,fill=black,label=right:{\tiny $r_{j+1}$}]{};
		\node(v3-5) at (-2,3.5){};
		\node(v3-6) at (-1.3,3.5){};
		\node(v3-1) at (-1,3.5){};
		\node(v3-2) at (-0.3,3.5){};
		\node(v3-3) at (1.3,3.5){};
		\node(v3-4) at(2,3.5){};
		\node(v3-7) at (1,3.5){};
		\node(v3-8) at (0.3,3.5){};
		\draw(v0)--(v1-1);
		\draw(v0)--(v1-3);
		\path[-,font=\scriptsize]
		(v0) edge node[descr]{{\tiny$i$}} (v1-2);
		\draw(v1-2)--(v2-1);
		\draw(v1-2)--(v2-3);
		\draw(v1-2)--(v2-2);
		\draw(v1-2)--(v2-4);
		\draw(v1-2)--(v2-5);
		\draw(v2-1)--(v3-5);
		\draw(v2-1)--(v3-6);
		\draw(v2-2)--(v3-1);
		\draw(v2-2)--(v3-2);
		\draw(v2-4)--(v3-3);
		\draw(v2-4)--(v3-4);
		\draw(v2-5)--(v3-7);
		\draw(v2-5)--(v3-8);
		\draw[dotted](-1,0.7)--(-0.1,0.7);
		\draw[dotted](0.1,0.7)--(1,0.7);
		\draw[dotted](0.6,2.4)--(1,2.4);
		\draw[dotted](-1,2.4)--(-0.6,2.4);
		\draw[dotted](-1.7,3.2)--(-1.4,3.2);
		\draw[dotted](-0.7,3.2)--(-0.4,3.2);
		\draw[dotted](1.7,3.2)--(1.4,3.2);
		\draw[dotted](0.7,3.2)--(0.4,3.2);
	\end{tikzpicture}
\end{eqnarray*}
}
%where $2\leqslant q\leqslant p$, $1\leqslant k_1<\dots<k_{q-1}\leqslant p$, $1\leqslant i\leqslant r_1$ and $r_j\geqslant 1$ for all $1\leqslant j\leqslant q$.
\end{itemize}

Now, we define a family of operations $\{\Delta_T: \mathscr{S}(\RBS^\ac)\rightarrow \mathscr{S}({\RBS^\ac})^{\ot T}\}_{T\in \frakt}$ as follows:
\begin{itemize}
	\item[(1)] For a tree $T$ of type $\mathrm{(I)}$ with $1\leqslant j\leqslant n, 1\leqslant i\leqslant n-j+1$,   define $\Delta_T(u_n)=u_{n-j+1}\ot u_j$, which can be drawn as
	\begin{eqnarray*}
		\begin{tikzpicture}[scale=0.8,descr/.style={fill=white}]
			\tikzstyle{every node}=[thick,minimum size=5pt, inner sep=1pt]
			\node(r) at (0,-0.5)[minimum size=0pt,rectangle]{};
			\node(v-2) at(-1.5,0.5)[minimum size=0pt, label=left:{$\Delta_T(u_n)=$}]{};
			\node(v0) at (0,0)[draw,rectangle]{{\small $u_{n-j+1}$}};
			\node(v1-1) at (-1.5,1){};
			\node(v1-2) at(0,1.3)[draw,rectangle]{\small$u_j$};
			\node(v1-3) at(1.5,1){};
			\node(v2-1)at (-1,2.3){};
			\node(v2-2) at(1,2.3){};
			\draw(v0)--(v1-1);
			\draw(v0)--(v1-3);
			\draw(v1-2)--(v2-1);
			\draw(v1-2)--(v2-2);
			\draw[dotted](-0.4,1.8)--(0.4,1.8);
			\draw[dotted](-0.5,0.5)--(-0.1,0.5);
			\draw[dotted](0.1,0.5)--(0.5,0.5);
			\path[-,font=\scriptsize]
			(v0) edge node[descr]{{\tiny$i$}} (v1-2);
		\end{tikzpicture}
	\end{eqnarray*}
	
	Define
	$$\Delta_T(v_n)=\left\{\begin{array}{ll} v_n\ot u_1, & j=1,\\
		u_1\ot v_n, & j=n, \end{array}\right.$$
	which can be pictured as
	\begin{eqnarray*}
		\begin{tikzpicture}[scale=0.8,descr/.style={fill=white}]
			\tikzstyle{every node}=[thick,minimum size=5pt, inner sep=1pt]
			\node(r) at (0,-0.5)[minimum size=0pt,rectangle]{};
			\node(v-1) at(-1.5,0.5)[minimum size=0pt, label=left:{when $j=1$, $\Delta_T(v_n)=$}]{};
			\node(v0) at (0,0)[draw,rectangle]{\small$v_n$};
			\node(v1-1) at (-1.3,1){};
			\node(v1-2) at(0,1)[draw,rectangle]{\small$u_1$};
			\node(v1-3) at(1.3,1){};
			\node(v2-1)at (0,1.8){};
			\draw(v0)--(v1-1);
			\draw(v0)--(v1-3);
			\draw(v1-2)--(v2-1);
			\draw[dotted](-0.5,0.5)--(-0.1,0.5);
			\draw[dotted](0.1,0.5)--(0.5,0.5);
			\path[-,font=\scriptsize]
			(v0) edge node[descr]{{\tiny$i$}} (v1-2);
		\end{tikzpicture}
	\end{eqnarray*}

	\begin{eqnarray*}
		\begin{tikzpicture}[scale=0.8,descr/.style={fill=white}]
			\tikzstyle{every node}=[thick,minimum size=5pt, inner sep=1pt]
			\node(r) at (0,-0.5)[minimum size=0pt,rectangle]{};
			\node(va) at(-2.5,0.5)[minimum size=0pt, label=left:{when $j=n$, $\Delta_T(v_n)=$}]{};
			\node(ve) at (-1.5,0)[draw, rectangle]{\small $u_1$};
			\node(ve1) at (-1.5,1)[draw,rectangle]{\small $v_n$};
			\node(ve2-1) at(-2.5,2){};
			\node(ve2-2) at(-0.5,2){};
			\draw(ve)--(ve1);
			\draw(ve1)--(ve2-1);
			\draw(ve1)--(ve2-2);
			\draw[dotted](-1.1,1.5)--(-1.9,1.5);
		\end{tikzpicture}
	\end{eqnarray*}
	
	and
	$$\Delta_T(w_n)=\left\{\begin{array}{ll} w_n\ot u_1, & j=1,\\
	u_1\ot w_n, & j=n, \end{array}\right.$$
	which can be pictured as
	\begin{eqnarray*}
		\begin{tikzpicture}[scale=0.8,descr/.style={fill=white}]
			\tikzstyle{every node}=[thick,minimum size=5pt, inner sep=1pt]
			\node(r) at (0,-0.5)[minimum size=0pt,rectangle]{};
			\node(v-1) at(-1.5,0.5)[minimum size=0pt, label=left:{when $j=1$, $\Delta_T(w_n)=$}]{};
			\node(v0) at (0,0)[draw,rectangle]{\small$w_n$};
			\node(v1-1) at (-1.3,1){};
			\node(v1-2) at(0,1)[draw,rectangle]{\small$u_1$};
			\node(v1-3) at(1.3,1){};
			\node(v2-1)at (0,1.8){};
			\draw(v0)--(v1-1);
			\draw(v0)--(v1-3);
			\draw(v1-2)--(v2-1);
			\draw[dotted](-0.5,0.5)--(-0.1,0.5);
			\draw[dotted](0.1,0.5)--(0.5,0.5);
			\path[-,font=\scriptsize]
			(v0) edge node[descr]{{\tiny$i$}} (v1-2);
		\end{tikzpicture}
	\end{eqnarray*}

	\begin{eqnarray*}
		\begin{tikzpicture}[scale=0.8,descr/.style={fill=white}]
			\tikzstyle{every node}=[thick,minimum size=5pt, inner sep=1pt]
			\node(r) at (0,-0.5)[minimum size=0pt,rectangle]{};
			\node(va) at(-2.5,0.5)[minimum size=0pt, label=left:{when $j=n$, $\Delta_T(v_n)=$}]{};
			\node(ve) at (-1.5,0)[draw, rectangle]{\small $u_1$};
			\node(ve1) at (-1.5,1)[draw,rectangle]{\small $v_n$};
			\node(ve2-1) at(-2.5,2){};
			\node(ve2-2) at(-0.5,2){};
			\draw(ve)--(ve1);
			\draw(ve1)--(ve2-1);
			\draw(ve1)--(ve2-2);
			\draw[dotted](-1.1,1.5)--(-1.9,1.5);
		\end{tikzpicture}
	\end{eqnarray*}
	\item[(2)] For a tree $T$ of type $\mathrm{(II)}$ with  $2\leqslant k\leqslant n,   r_1+\cdots +r_k=n, r_1, \dots, r_k\geqslant 1$,  define
	\begin{eqnarray*}
		\begin{tikzpicture}[scale=0.8,descr/.style={fill=white}]
			\tikzstyle{every node}=[thick,minimum size=5pt, inner sep=1pt]
			\node(v-2) at (-3.5,-0.5)[minimum size=0pt, label=left:{$\Delta_T(v_n)=$}]{};
			\node(v-1) at(-3.5,-0.5)[minimum size=0pt,label=right:{$(-1)^{\frac{k(k-1)}{2}}$}]{};
			\node(v0) at (0,-1.5)[rectangle, draw]{\small $u_k$};
			\node(v1) at(-1.2,-0.3)[rectangle,draw]{\small $v_{r_1}$};
			\node(v1-1) at(-2,0.8){};
			\node(v1-2) at (-1,0.8){};
			\node(v3) at (1.2,-0.3)[rectangle, draw]{\small $v_{r_k}$};
			\node(v3-1) at (1,0.8){};
			\node(v3-2) at (2,0.8){};
			\node(v2-1) at (-0.5,-0.3) [rectangle,draw]{\small $v_{r_2}$};
			\node(v2-1-1) at (-0.8,0.8){};
			\node(v2-1-2) at (0.2, 0.8){};
			%\draw [dotted,line width=1pt] (-0.7,-0.5)--(-0.2,-0.5);
			\draw [dotted,line width=1pt] (0.1,-0.5)--(0.5,-0.5);
			\draw [dotted,line width=1pt] (-0.5,0.5)--(-0.1,0.5);
			\draw [dotted,line width=1pt] (-1.6,0.5)--(-1.2, 0.5);
			\draw [dotted, line width=1pt] (1.2,0.5)--(1.6,0.5);
			\draw        (v0)--(v1);
			\draw         (v0)--(v3);
			\draw         (v0)--(v2-1);
			\draw         (v1)--(v1-1);
			\draw          (v1)--(v1-2);
			\draw        (v2-1)--(v2-1-1);
			\draw        (v2-1)--(v2-1-2);
			\draw        (v3)--(v3-1);
			\draw        (v3)--(v3-2);
			%\path[-,font=\scriptsize]
			%(v0) edge node[descr]{{\tiny$t$}} (v2-1);
		\end{tikzpicture}
	\end{eqnarray*}
and
\begin{eqnarray*}
	\begin{tikzpicture}[scale=0.8,descr/.style={fill=white}]
		\tikzstyle{every node}=[thick,minimum size=5pt, inner sep=1pt]
		\node(v-2) at (-3.5,-0.5)[minimum size=0pt, label=left:{$\Delta_T(w_n)=$}]{};
		\node(v-1) at(-3.6,-0.5)[minimum size=0pt,label=right:{$(-1)^{\frac{k(k-1)}{2}}$}]{};
		\node(v0) at (0,-1.5)[rectangle, draw]{\small $u_k$};
		\node(v1) at(-1.3,-0.3)[rectangle,draw]{\small $w_{r_1}$};
		\node(v1-1) at(-2,0.8){};
		\node(v1-2) at (-1,0.8){};
		\node(v3) at (1.2,-0.3)[rectangle, draw]{\small $w_{r_k}$};
		\node(v3-1) at (1,0.8){};
		\node(v3-2) at (2,0.8){};
		\node(v2-1) at (-0.5,-0.3) [rectangle,draw]{\small $w_{r_2}$};
		\node(v2-1-1) at (-0.8,0.8){};
		\node(v2-1-2) at (0.2, 0.8){};
		%\draw [dotted,line width=1pt] (-0.7,-0.5)--(-0.2,-0.5);
		\draw [dotted,line width=1pt] (0.1,-0.5)--(0.5,-0.5);
		\draw [dotted,line width=1pt] (-0.5,0.5)--(-0.1,0.5);
		\draw [dotted,line width=1pt] (-1.6,0.5)--(-1.2, 0.5);
		\draw [dotted, line width=1pt] (1.2,0.5)--(1.6,0.5);
		\draw        (v0)--(v1);
		\draw         (v0)--(v3);
		\draw         (v0)--(v2-1);
		\draw         (v1)--(v1-1);
		\draw          (v1)--(v1-2);
		\draw        (v2-1)--(v2-1-1);
		\draw        (v2-1)--(v2-1-2);
		\draw        (v3)--(v3-1);
		\draw        (v3)--(v3-2);
		%\path[-,font=\scriptsize]
		%(v0) edge node[descr]{{\tiny$t$}} (v2-1);
	\end{tikzpicture}
\end{eqnarray*}
	\item[(3)] For a tree $T$ of type $\mathrm{(III)}$ with $2\leqslant q\leqslant p\leqslant n, 1\leqslant k_1<\dots<k_{q-1}\leqslant p, r_1+\cdots +r_q+p-q=n, 1\leqslant i\leqslant r_1, r_1, \dots, r_q\geqslant 1$, define
	\begin{eqnarray*}
		\begin{tikzpicture}[scale=0.8,descr/.style={fill=white}]
			\tikzstyle{every node}=[minimum size=4pt, inner sep=1pt]
			\node(v-2) at (-3.6,1.2)[minimum size=0pt, label=left:{$\Delta_T(v_n)=$}]{};
			\node(v-1) at(-3.6,1.3)[minimum size=0pt,label=right:{$(-1)^\frac{p(p-1)}{2}$}]{};

		\node(v0) at (0,0)[rectangle, draw]{\small $v_{r_1}$}{};
		\node(v1-1) at (-2,1){};
		\node(v1-2) at(0,1.2)[rectangle, draw]{\small $u_{p}$}{};
		\node(v1-3) at(2,1){};
		\node(v2-1) at(-1.5,2.8)[rectangle, draw]{\small $v_{r_2}$}{};
		\node(v2-2) at (-0.6, 2.8)[rectangle, draw]{\small $v_{r_j}$}{};
		\node(v2-3) at (0,2.9){\tiny $j$}{};
		\node(v2-4) at(1.7,2.8)[rectangle, draw]{\small $w_{r_p}$}{};
		\node(v2-5) at(0.7,2.8)[rectangle, draw]{\small $w_{r_{j+1}}$}{};
		\node(v3-5) at (-2,3.5){};
		\node(v3-6) at (-1.3,3.5){};
		\node(v3-1) at (-1,3.5){};
		\node(v3-2) at (-0.3,3.5){};
		\node(v3-3) at (1.3,3.5){};
		\node(v3-4) at(2,3.5){};
		\node(v3-7) at (1,3.5){};
		\node(v3-8) at (0.3,3.5){};
		\draw(v0)--(v1-1);
		\draw(v0)--(v1-3);
		\path[-,font=\scriptsize]
		(v0) edge node[descr]{{\tiny$i$}} (v1-2);
		\draw(v1-2)--(v2-1);
		\draw(v1-2)--(v2-3);
		\draw(v1-2)--(v2-2);
		\draw(v1-2)--(v2-4);
		\draw(v1-2)--(v2-5);
		\draw(v2-1)--(v3-5);
		\draw(v2-1)--(v3-6);
		\draw(v2-2)--(v3-1);
		\draw(v2-2)--(v3-2);
		\draw(v2-4)--(v3-3);
		\draw(v2-4)--(v3-4);
		\draw(v2-5)--(v3-7);
		\draw(v2-5)--(v3-8);
		\draw[dotted](-1,0.7)--(-0.1,0.7);
		\draw[dotted](0.1,0.7)--(1,0.7);
		\draw[dotted](0.6,2.4)--(1,2.4);
		\draw[dotted](-1,2.4)--(-0.6,2.4);
		\draw[dotted](-1.7,3.2)--(-1.4,3.2);
		\draw[dotted](-0.8,3.2)--(-0.5,3.2);
		\draw[dotted](1.8,3.2)--(1.5,3.2);
				\draw[dotted](0.8,3.2)--(0.5,3.2);
		\end{tikzpicture}
	\end{eqnarray*}
and
	\begin{eqnarray*}
	\begin{tikzpicture}[scale=0.8,descr/.style={fill=white}]
		\tikzstyle{every node}=[minimum size=4pt, inner sep=1pt]
		\node(v-2) at (-3.6,1.2)[minimum size=0pt, label=left:{$\Delta_T(w_n)=$}]{};
		\node(v-1) at(-3.6,1.3)[minimum size=0pt,label=right:{$(-1)^\frac{p(p-1)}{2}$}]{};

		\node(v0) at (0,0)[rectangle, draw]{\small $w_{r_1}$}{};
		\node(v1-1) at (-2,1){};
		\node(v1-2) at(0,1.2)[rectangle, draw]{\small $u_{p}$}{};
		\node(v1-3) at(2,1){};
		\node(v2-1) at(-1.5,2.8)[rectangle, draw]{\small $v_{r_2}$}{};
		\node(v2-2) at (-0.6, 2.8)[rectangle, draw]{\small $v_{r_j}$}{};
		\node(v2-3) at (0,2.9){\tiny $j$}{};
		\node(v2-4) at(1.7,2.8)[rectangle, draw]{\small $w_{r_p}$}{};
		\node(v2-5) at(0.7,2.8)[rectangle, draw]{\small $w_{r_{j+1}}$}{};
		\node(v3-5) at (-2,3.5){};
		\node(v3-6) at (-1.3,3.5){};
		\node(v3-1) at (-1,3.5){};
		\node(v3-2) at (-0.3,3.5){};
		\node(v3-3) at (1.3,3.5){};
		\node(v3-4) at(2,3.5){};
		\node(v3-7) at (1,3.5){};
		\node(v3-8) at (0.3,3.5){};
		\draw(v0)--(v1-1);
		\draw(v0)--(v1-3);
		\path[-,font=\scriptsize]
		(v0) edge node[descr]{{\tiny$i$}} (v1-2);
		\draw(v1-2)--(v2-1);
		\draw(v1-2)--(v2-3);
		\draw(v1-2)--(v2-2);
		\draw(v1-2)--(v2-4);
		\draw(v1-2)--(v2-5);
		\draw(v2-1)--(v3-5);
		\draw(v2-1)--(v3-6);
		\draw(v2-2)--(v3-1);
		\draw(v2-2)--(v3-2);
		\draw(v2-4)--(v3-3);
		\draw(v2-4)--(v3-4);
		\draw(v2-5)--(v3-7);
		\draw(v2-5)--(v3-8);
		\draw[dotted](-1,0.7)--(-0.1,0.7);
		\draw[dotted](0.1,0.7)--(1,0.7);
		\draw[dotted](0.6,2.4)--(1,2.4);
		\draw[dotted](-1,2.4)--(-0.6,2.4);
		\draw[dotted](-1.7,3.2)--(-1.4,3.2);
		\draw[dotted](-0.7,3.2)--(-0.4,3.2);
		\draw[dotted](1.8,3.2)--(1.5,3.2);
		\draw[dotted](0.8,3.2)--(0.5,3.2);
	\end{tikzpicture}
\end{eqnarray*}
	\item[(4)]All other components of $\Delta_T, T\in \frakt$ vanish.
\end{itemize}

\begin{prop}\label{Prop-Koszul-dual-coooperad}
	The graded collection $\mathscr{S}(\RBS^\ac)$ endowed with the  operations $\{\Delta_T\}_{T\in \frakt}$ introduced above forms a coaugmented homotopy cooperad, whose  strict counit is the natural projection $\varepsilon:\mathscr{S}(\RBS^\ac)\twoheadrightarrow \bfk u_1\cong \cali$ and the coaugmentation is just the natural embedding $\eta:\cali\cong \bfk u_1\hookrightarrow \mathscr{S}(\RBS^\ac)$. 	
\end{prop}

We postpone the proof of this result to Appendix~\ref{Appendix: Koszul dual homotopy cooperad} in order not to interrupt the flux of the presentation.

We will justify the following  definition by showing its cobar construction is exactly the minimal model of  $\RBS$ in Section~\ref{Section: minimal model RBSA}, hence the name ``Koszul dual homotopy cooperad".
\begin{defn}
	The homotopy cooperad $\mathscr{S}(\RBS^\ac)\ot_{\mathrm{H}} \cals^{-1}$, denoted by $\RBS^\ac$, is called the Koszul dual homotopy cooperad of $\RBS$.
\end{defn}

Precisely, the underlying graded collection of ${\RBS^\ac}$ is $${\RBS^\ac}(n)=\bfk e_n\oplus \bfk o^{R}_n\oplus \bfk o^S_n, n\geqslant 1$$ with $e_n=u_n\ot \varepsilon_n$, $o^R_n=v_n\ot \varepsilon_n$  and $o^S_n=w_n\ot \varepsilon_n$, thus $|e_n|=n-1$ and $|o_n|=n$.
The   defining operations $\{\Delta_T\}_{T\in \frakt}$   are given by the following formulae:
\begin{itemize}
	\item[(1)] For a tree $T$ of type $\mathrm{(I)}$ with $1\leqslant j\leqslant n, 1\leqslant i\leqslant n-j+1$,
	\begin{eqnarray*}
		\begin{tikzpicture}[scale=0.8,descr/.style={fill=white}]
			\tikzstyle{every node}=[thick,minimum size=4pt, inner sep=1pt]
			\node(r) at (0,-0.5)[minimum size=0pt,rectangle]{};
			\node(v-2) at(-1.4,0.5)[minimum size=0pt, label=left:{$\Delta_T(e_n)=(-1)^{(j-1)(n-i+1)}$}]{};
			\node(v0) at (0,0)[draw,rectangle]{{\small $e_{n-j+1}$}};
			\node(v1-1) at (-1.5,1){};
			\node(v1-2) at(0,1)[draw,rectangle]{\small$e_j$};
			\node(v1-3) at(1.5,1){};
			\node(v2-1)at (-1,2){};
			\node(v2-2) at(1,2){};
			\draw(v0)--(v1-1);
			\draw(v0)--(v1-3);
			\draw(v1-2)--(v2-1);
			\draw(v1-2)--(v2-2);
			\draw[dotted](-0.4,1.5)--(0.4,1.5);
			\draw[dotted](-0.5,0.5)--(-0.1,0.5);
			\draw[dotted](0.1,0.5)--(0.5,0.5);
			\path[-,font=\scriptsize]
			(v0) edge node[descr]{{\tiny$i$}} (v1-2);
		\end{tikzpicture}
	\end{eqnarray*}
 For $\Delta_T(o^R_n)$,  $j=1$,
	\begin{eqnarray*}
		\begin{tikzpicture}[scale=0.8,descr/.style={fill=white}]
			\tikzstyle{every node}=[thick,minimum size=4pt, inner sep=1pt]
			\node(r) at (0,-0.5)[minimum size=0pt,rectangle]{};
			\node(v-1) at(-1.3,0.5)[minimum size=0pt, label=left:{$\Delta_T(o^R_n)=$}]{};
			\node(v0) at (0,0)[draw,rectangle]{\small$o^R_n$};
			\node(v1-1) at (-1.3,1){};
			\node(v1-2) at(0,1.1)[draw,rectangle]{\small$e_1$};
			\node(v1-3) at(1.3,1){};
			\node(v2-1)at (0,1.8){};
			\draw(v0)--(v1-1);
			\draw(v0)--(v1-3);
			\draw(v1-2)--(v2-1);
			\draw[dotted](-0.5,0.5)--(-0.1,0.5);
			\draw[dotted](0.1,0.5)--(0.5,0.5);
			\path[-,font=\scriptsize]
			(v0) edge node[descr]{{\tiny$i$}} (v1-2);
		\end{tikzpicture}
	\end{eqnarray*}
 for  $j=n$,
		\begin{eqnarray*}
		\begin{tikzpicture}[scale=0.75,descr/.style={fill=white}]
			\tikzstyle{every node}=[thick,minimum size=5pt, inner sep=1pt]
			\node(r) at (0,-0.5)[minimum size=0pt,rectangle]{};
			\node(va) at(-2.4,0.5)[minimum size=0pt, label=left:{$\Delta_T(o^R_n)=$}]{};
			\node(ve) at (-1.5,0)[draw, rectangle]{\small $u_1$};
			\node(ve1) at (-1.5,1)[draw,rectangle]{\small $o^R_n$};
			\node(ve2-1) at(-2.5,2){};
			\node(ve2-2) at(-0.5,2){};
			\draw(ve)--(ve1);
			\draw(ve1)--(ve2-1);
			\draw(ve1)--(ve2-2);
			\draw[dotted](-1.1,1.6)--(-1.9,1.6);
		\end{tikzpicture}
	\end{eqnarray*} For $\Delta_T(o^S_n)$,  $j=1$,
\begin{eqnarray*}
\begin{tikzpicture}[scale=0.8,descr/.style={fill=white}]
	\tikzstyle{every node}=[thick,minimum size=5pt, inner sep=1pt]
	\node(r) at (0,-0.5)[minimum size=0pt,rectangle]{};
	\node(v-1) at(-1.5,0.5)[minimum size=0pt, label=left:{$\Delta_T(o^S_n)=$}]{};
	\node(v0) at (0,0)[draw,rectangle]{\small$o^S_n$};
	\node(v1-1) at (-1.3,1){};
	\node(v1-2) at(0,1.1)[draw,rectangle]{\small$e_1$};
	\node(v1-3) at(1.3,1){};
	\node(v2-1)at (0,1.8){};
	\draw(v0)--(v1-1);
	\draw(v0)--(v1-3);
	\draw(v1-2)--(v2-1);
	\draw[dotted](-0.5,0.5)--(-0.1,0.5);
	\draw[dotted](0.1,0.5)--(0.5,0.5);
	\path[-,font=\scriptsize]
	(v0) edge node[descr]{{\tiny$i$}} (v1-2);
\end{tikzpicture}
\end{eqnarray*}
for  $j=n$,
\begin{eqnarray*}
\begin{tikzpicture}[scale=0.8,descr/.style={fill=white}]
	\tikzstyle{every node}=[thick,minimum size=5pt, inner sep=1pt]
	\node(r) at (0,-0.5)[minimum size=0pt,rectangle]{};
	\node(va) at(-2.5,0.5)[minimum size=0pt, label=left:{$\Delta_T(o^S_n)=$}]{};
	\node(ve) at (-1.5,0)[draw, rectangle]{\small $u_1$};
	\node(ve1) at (-1.5,1)[draw,rectangle]{\small $o^S_n$};
	\node(ve2-1) at(-2.5,2){};
	\node(ve2-2) at(-0.5,2){};
	\draw(ve)--(ve1);
	\draw(ve1)--(ve2-1);
	\draw(ve1)--(ve2-2);
	\draw[dotted](-1.1,1.5)--(-1.9,1.5);
\end{tikzpicture}
\end{eqnarray*}
	
	\item[(2)]  For a tree $T$ of type $\mathrm{(II)}$ with  $2\leqslant k\leqslant n,   r_1+\cdots +r_k=n, r_1, \dots, r_k\geqslant 1$,
	\begin{eqnarray*}
		\begin{tikzpicture}[scale=0.8,descr/.style={fill=white}]
			\tikzstyle{every node}=[thick,minimum size=5pt, inner sep=1pt]
			\node(v-2) at (-3.5,-0.5)[minimum size=0pt, label=left:{$\Delta_T(o^R_n)=$}]{};
			\node(v-1) at(-3.5,-0.4)[minimum size=0pt,label=right:{$(-1)^{\frac{k(k-1)}{2}}$}]{};
			\node(v0) at (0,-1.5)[rectangle, draw]{\small $e_k$};
			\node(v1) at(-1.2,-0.3)[rectangle,draw]{\small $o^R_{r_1}$};
			\node(v1-1) at(-2,0.8){};
			\node(v1-2) at (-1,0.8){};
			\node(v3) at (1.2,-0.3)[rectangle, draw]{\small $o^R_{r_k}$};
			\node(v3-1) at (1,0.8){};
			\node(v3-2) at (2,0.8){};
			\node(v2-1) at (-0.4,-0.3) [rectangle,draw]{\small $o^R_{r_2}$};
			\node(v2-1-1) at (-0.8,0.8){};
			\node(v2-1-2) at (0.2, 0.8){};
			%\draw [dotted,line width=1pt] (-0.7,-0.5)--(-0.2,-0.5);
			\draw [dotted,line width=1pt] (0.1,-0.5)--(0.6,-0.5);
			\draw [dotted,line width=1pt] (-0.5,0.5)--(-0.1,0.5);
			\draw [dotted,line width=1pt] (-1.6,0.5)--(-1.2, 0.5);
			\draw [dotted, line width=1pt] (1.2,0.5)--(1.6,0.5);
			\draw        (v0)--(v1);
			\draw         (v0)--(v3);
			\draw         (v0)--(v2-1);
			\draw         (v1)--(v1-1);
			\draw          (v1)--(v1-2);
			\draw        (v2-1)--(v2-1-1);
			\draw        (v2-1)--(v2-1-2);
			\draw        (v3)--(v3-1);
			\draw        (v3)--(v3-2);
			%\path[-,font=\scriptsize]
			%(v0) edge node[descr]{{\tiny$t$}} (v2-1);
		\end{tikzpicture}
	\end{eqnarray*} and	\begin{eqnarray*}
	\begin{tikzpicture}[scale=0.8,descr/.style={fill=white}]
		\tikzstyle{every node}=[thick,minimum size=5pt, inner sep=1pt]
		\node(v-2) at (-3.5,-0.5)[minimum size=0pt, label=left:{$\Delta_T(o^S_n)=$}]{};
		\node(v-1) at(-3.5,-0.4)[minimum size=0pt,label=right:{$(-1)^{\frac{k(k-1)}{2}}$}]{};
		\node(v0) at (0,-1.5)[rectangle, draw]{\small $e_k$};
		\node(v1) at(-1.2,-0.3)[rectangle,draw]{\small $o^S_{r_1}$};
		\node(v1-1) at(-2,0.8){};
		\node(v1-2) at (-1,0.8){};
		\node(v3) at (1.2,-0.3)[rectangle, draw]{\small $o^S_{r_k}$};
		\node(v3-1) at (1,0.8){};
		\node(v3-2) at (2,0.8){};
		\node(v2-1) at (-0.4,-0.3) [rectangle,draw]{\small $o^S_{r_2}$};
		\node(v2-1-1) at (-0.8,0.8){};
		\node(v2-1-2) at (0.2, 0.8){};
		%\draw [dotted,line width=1pt] (-0.7,-0.5)--(-0.2,-0.5);
		\draw [dotted,line width=1pt] (0.1,-0.5)--(0.6,-0.5);
		\draw [dotted,line width=1pt] (-0.5,0.5)--(-0.1,0.5);
		\draw [dotted,line width=1pt] (-1.6,0.5)--(-1.2, 0.5);
		\draw [dotted, line width=1pt] (1.2,0.5)--(1.6,0.5);
		\draw        (v0)--(v1);
		\draw         (v0)--(v3);
		\draw         (v0)--(v2-1);
		\draw         (v1)--(v1-1);
		\draw          (v1)--(v1-2);
		\draw        (v2-1)--(v2-1-1);
		\draw        (v2-1)--(v2-1-2);
		\draw        (v3)--(v3-1);
		\draw        (v3)--(v3-2);
		%\path[-,font=\scriptsize]
		%(v0) edge node[descr]{{\tiny$t$}} (v2-1);
	\end{tikzpicture}
\end{eqnarray*}
	\item[(3)]  For a tree $T$ of type $\mathrm{(III)}$ with $ 1\leqslant j\leqslant p, r_1+\cdots +r_p=n, 1\leqslant i\leqslant r_1, r_1, \dots, r_p\geqslant 1$,
\begin{eqnarray*}
	\begin{tikzpicture}[scale=0.8,descr/.style={fill=white}]
		\tikzstyle{every node}=[minimum size=4pt, inner sep=1pt]
		\node(v-2) at (-3.5,1.2)[minimum size=0pt, label=left:{$\Delta_T(o^R_n)=$}]{};
		\node(v-1) at(-3.5,1.2)[minimum size=0pt,label=right:{$(-1)^{\gamma}$}]{};
		
		\node(v0) at (0,0)[rectangle, draw]{\small $o^R_{r_1}$}{};
		\node(v1-1) at (-2,1){};
		\node(v1-2) at(0,1.2)[rectangle, draw]{\small $e_{p}$}{};
		\node(v1-3) at(2,1){};
		\node(v2-1) at(-1.5,2.8)[rectangle, draw]{\small $o^R_{r_2}$}{};
		\node(v2-2) at (-0.6, 2.8)[rectangle, draw]{\small $o^R_{r_j}$}{};
		\node(v2-3) at (0,2.9){\tiny $j$}{};
		\node(v2-4) at(1.6,2.8)[rectangle, draw]{\small $o^S_{r_p}$}{};
		\node(v2-5) at(0.7,2.8)[rectangle, draw]{\small $o^S_{r_{j+1}}$}{};
		\node(v3-5) at (-2.1,3.6){};
		\node(v3-6) at (-1.3,3.6){};
		\node(v3-1) at (-1,3.6){};
		\node(v3-2) at (-0.3,3.6){};
		\node(v3-3) at (1.3,3.6){};
		\node(v3-4) at(2,3.6){};
		\node(v3-7) at (1,3.6){};
		\node(v3-8) at (0.3,3.6){};
		\draw(v0)--(v1-1);
		\draw(v0)--(v1-3);
		\path[-,font=\scriptsize]
		(v0) edge node[descr]{{\tiny$i$}} (v1-2);
		\draw(v1-2)--(v2-1);
		\draw(v1-2)--(v2-3);
		\draw(v1-2)--(v2-2);
		\draw(v1-2)--(v2-4);
		\draw(v1-2)--(v2-5);
		\draw(v2-1)--(v3-5);
		\draw(v2-1)--(v3-6);
		\draw(v2-2)--(v3-1);
		\draw(v2-2)--(v3-2);
		\draw(v2-4)--(v3-3);
		\draw(v2-4)--(v3-4);
		\draw(v2-5)--(v3-7);
		\draw(v2-5)--(v3-8);
		\draw[dotted,line width=0.9pt](-1,0.7)--(-0.1,0.7);
		\draw[dotted,line width=0.9pt](0.1,0.7)--(1,0.7);
		\draw[dotted,line width=0.8pt](0.5,2.2)--(0.9,2.2);
		\draw[dotted,line width=0.8pt](-0.85,2.2)--(-0.45,2.2);
		\draw[dotted,line width=0.7pt](-1.75,3.3)--(-1.45,3.3);
		\draw[dotted,line width=0.7pt](-0.75,3.3)--(-0.45,3.3);
		\draw[dotted,line width=0.7pt](1.75,3.3)--(1.45,3.3);
		\draw[dotted,line width=0.7pt](0.8,3.3)--(0.5,3.3);
	\end{tikzpicture}
\end{eqnarray*}
and
\begin{eqnarray*}
	\begin{tikzpicture}[scale=0.8,descr/.style={fill=white}]
		\tikzstyle{every node}=[minimum size=4pt, inner sep=1pt]
		\node(v-2) at (-3.5,1.2)[minimum size=0pt, label=left:{$\Delta_T(o^S_n)=$}]{};
		\node(v-1) at(-3.5,1.2)[minimum size=0pt,label=right:{$(-1)^{\gamma}$}]{};
		
		\node(v0) at (0,0)[rectangle, draw]{\small $o^S_{r_1}$}{};
		\node(v1-1) at (-2,1){};
		\node(v1-2) at(0,1.2)[rectangle, draw]{\small $e_{p}$}{};
		\node(v1-3) at(2,1){};
		\node(v2-1) at(-1.5,2.8)[rectangle, draw]{\small $o^R_{r_2}$}{};
		\node(v2-2) at (-0.6, 2.8)[rectangle, draw]{\small $o^R_{r_j}$}{};
		\node(v2-3) at (0,2.9){\tiny $j$}{};
		\node(v2-4) at(1.6,2.8)[rectangle, draw]{\small $o^S_{r_p}$}{};
		\node(v2-5) at(0.7,2.8)[rectangle, draw]{\small $o^S_ {r_{j+1}}$}{};
		\node(v3-5) at (-2,3.6){};
		\node(v3-6) at (-1.3,3.6){};
		\node(v3-1) at (-1,3.6){};
		\node(v3-2) at (-0.3,3.6){};
		\node(v3-3) at (1.3,3.6){};
		\node(v3-4) at(2,3.6){};
		\node(v3-7) at (1,3.6){};
		\node(v3-8) at (0.3,3.6){};
		\draw(v0)--(v1-1);
		\draw(v0)--(v1-3);
		\path[-,font=\scriptsize]
		(v0) edge node[descr]{{\tiny$i$}} (v1-2);
		\draw(v1-2)--(v2-1);
		\draw(v1-2)--(v2-3);
		\draw(v1-2)--(v2-2);
		\draw(v1-2)--(v2-4);
		\draw(v1-2)--(v2-5);
		\draw(v2-1)--(v3-5);
		\draw(v2-1)--(v3-6);
		\draw(v2-2)--(v3-1);
		\draw(v2-2)--(v3-2);
		\draw(v2-4)--(v3-3);
		\draw(v2-4)--(v3-4);
		\draw(v2-5)--(v3-7);
		\draw(v2-5)--(v3-8);
		\draw[dotted,line width=0.9pt](-1,0.7)--(-0.1,0.7);
		\draw[dotted,line width=0.9pt](0.1,0.7)--(1,0.7);
		\draw[dotted,line width=0.8pt](0.5,2.2)--(0.9,2.2);
		\draw[dotted,line width=0.8pt](-0.85,2.2)--(-0.45,2.2);
		\draw[dotted,line width=0.7pt](-1.7,3.3)--(-1.4,3.3);
		\draw[dotted,line width=0.7pt](-0.75,3.3)--(-0.45,3.3);
		\draw[dotted,line width=0.7pt](1.75,3.3)--(1.45,3.3);
		\draw[dotted,line width=0.7pt](0.8,3.3)--(0.5,3.3);
	\end{tikzpicture}
\end{eqnarray*}
	where \begin{eqnarray*}
		\gamma=\sum_{j=1}^{p-1}(q-k)r_k+p-1+(\sum_{k=2}^pr_k)(r_1-i)+\sum_{k=2}^j(r_k-1)+\sum_{k=2}^p(r_k-1)(p-k).	\end{eqnarray*}
	\item[(4)]All other components of $\Delta_T, T\in \frakt$ vanish.
\end{itemize}

\bigskip

\section{The Minimal model for the operad of Rota-Baxter systems} \label{Section: minimal model RBSA}
 In this section,  we will prove that the cobar construction of   the homotopy cooperad $\RBS^\ac$ introduced in Section~\ref{Section: Koszul dual cooperad}  is the minimal model of the operad for Rota-Baxter systems.

%Therefore,  the cohomology theory for Rota-Baxter system    defined before is  the right cohomology theory for Rota-Baxter system   in the sense of operad theory.

%For basic  theory of operads, we refer the reader to the textbooks \cite{LV, BD}.
%As we will only care about symmetric operads in this paper, we will delete the adjective ``symmetric" everywhere.

For a collection $M=\{M(n)\}_{n\geqslant 1} $ of (graded) vector spaces, denote by $ \mathcal{F}(M)$ the free (graded) operad generated by $M$. Recall that a dg operad is called quasi-free if its underlying graded operad is free.

\begin{defn}\cite{DCV13}\label{Def: Minimal model of operads} A minimal model for an operad $\mathcal{P}$  is a quasi-free dg operad $ (\mathcal{F}(M),d)$ together with a surjective quasi-isomorphism of operads $(\mathcal{F}(M), \partial)\overset{\sim}{\twoheadrightarrow} \mathcal{P}$, where the dg operad $(\mathcal{F}(M),  \partial)$  satisfies the following conditions:
	\begin{itemize}
		\item[(i)] the differential $\partial$ is decomposable, i.e.,  $\partial$ takes $M$ to $\mathcal{F}(M)^{\geqslant 2}$, the subspace of $\mathcal{F}(M)$ consisting of elements with weight $\geqslant 2$;
		\item[(ii)] the generating collection $M$ admits a decomposition $M=\bigoplus\limits_{i\geqslant 1}M_{(i)}$  such that $\partial(M_{(k+1)})\subset \mathcal{F}\Big(\bigoplus\limits_{i=1}^kM_{(i)}\Big)$ for any $k\geqslant 1$.  \end{itemize}
\end{defn}
\begin{thm}\cite{DCV13} When an operad $\mathcal{P}$ admits a minimal model, it is unique up to isomorphisms.
\end{thm}

The  operad for Rota-Baxter systems, denoted by $\RBS$,  is generated by two unary operators $R, S$ and a binary operator $\mu$ with the operadic relations generated by
\[ \mu \circ_{1} \mu   - \mu \circ_{2} \mu ,\]

$$ (\mu\circ_1 R)\circ_2 R-(R \circ_1\mu)\circ_1 R-(R \circ_1\mu)\circ_2 S$$
and
$$ (\mu\circ_1 S)\circ_2 S-(S \circ_1\mu)\circ_1 R-(S \circ_1\mu)\circ_2 S.$$

\begin{defn}
The differential graded operad $\Omega (\RBS^\ac)$, denoted by $\RBSinfty$,  is called the operad of homotopy Rota-Baxter   systems.
\end{defn}

Let us describe this dg operad explicitly.
The  dg operad $\RBS_\infty$ is the free operad generated by the graded collection $s^{-1}\overline{{\RBS^\ac}}$ endowed with the differential induced from the homotopy cooperad structure.
More precisely, $$s^{-1}\overline{{\RBS^\ac}}(1)=\bfk s^{-1}o^R_1 \oplus \bfk s^{-1}o^S_1\ \mathrm{and}\ s^{-1}\overline{{\RBS^\ac}}(n)=\bfk s^{-1}e_n\oplus \bfk s^{-1}o^R_n\oplus \bfk s^{-1}o^S_n , n\geqslant2.$$
Denote $m_n=s^{-1}e_n, n\geqslant 2$ and $R_n=s^{-1}o^R_n$, $S_n=s^{-1}o^S_n, n\geqslant 1$  respectively, so  $|m_n|=n-2, |R_n|=n-1, |S_n|=n-1$. The action of the  differential on these generators in $\RBS_\infty$ is given by the following formulae:
	\begin{eqnarray}\label{Eq: defining HRB 1} \forall n\geqslant 2,\quad
		\partial{m_n} = \sum_{j=2}^{n-1}\sum_{i=1}^{n-j+1}(-1)^{i+j(n-i)}m_{n-j+1}\circ_i m_j,
	\end{eqnarray}
	\begin{align} \label{Eq: defining HRB 2}
		\forall n\geq1&\\
	    \notag \partial R_n=& \sum\limits_{k=2}^n\sum_{l_1+\cdots+l_k=n\atop l_1, \dots, l_k\geqslant 1}(-1)^{\alpha}\Big(\cdots\big((m_k\circ_1 R_{l_1})\circ_{l_1+1}R_{l_2}\big)\cdots\Big)\circ_{l_1+\cdots+l_{k-1}+1}T_{l_k}  +\\
		\notag   &  \sum\limits_{{\small\substack{2\leqslant p\leqslant n \\ 1\leqslant j\leqslant p
		}}}\sum\limits_{\small\substack{ r_1+\dots+r_p=n\\r_1, \dots, r_p\geqslant 1\\1\leqslant i\leqslant r_1}}(-1)^{\beta}\Big(R_{r_1}\circ_i \Big(\Big(\big((\cdots( m_p\circ_{1}R_{r_2}))\cdots \big)\circ_{r_2+\dots+r_{j-1}+1}R_{r_j}\Big)\circ_{r_2+\dots+r_{j-1}+2}S_{r_{j+1}}\Big)\circ_{r_2+\dots+r_{p}+2}S_{r_p}\Big) ,
	\end{align}
and
\begin{align} \label{Eq: defining HRB 3}
	\forall n\geq1&\\
   \notag  \partial S_n=& \sum\limits_{k=2}^n\sum_{l_1+\cdots+l_k=n\atop l_1, \dots, l_k\geqslant 1}(-1)^{\alpha}\Big(\cdots\big((m_k\circ_1 S_{l_1})\circ_{l_1+1}T_{l_2}\big)\cdots\Big)\circ_{l_1+\cdots+l_{k-1}+1}S_{l_k} +\\
\notag   &  \sum\limits_{{\small\substack{2\leqslant p\leqslant n \\ 1\leqslant j\leqslant p
}}}\sum\limits_{\small\substack{ r_1+\dots+r_p=n\\r_1, \dots, r_p\geqslant 1\\1\leqslant i\leqslant r_1}}(-1)^{\beta}\Big(S_{r_1}\circ_i \Big(\Big(\big((\cdots( m_p\circ_{1}R_{r_2}))\cdots \big)\circ_{r_2+\dots+r_{j-1}+1}R_{r_j}\Big)\circ_{r_2+\dots+r_{j-1}+2}S_{r_{j+1}}\Big)\circ_{r_2+\dots+r_{p}+2}S_{r_p}\Big),
\end{align}
where the signs $(-1)^{\alpha}$ and $(-1)^{\beta}$ are   respectively
\begin{eqnarray*}\label{Eq: sign   alpha'}
	\alpha&=&1+\frac{k(k-1)}{2}+\sum_{j=1}^k(k-j)l_j=1+\sum_{j=1}^k(k-j)(l_j-1),\\
	\label{Eq: sign   beta'}\beta &=& 1+i+\big(p+\sum\limits_{j=2}^p(r_j-1)\big)\big(r_1-i\big)+\sum_{k=2}^j(r_k-1)+\sum_{k=2}^p(r_k-1)(p-k).	\end{eqnarray*}

Let's display the elements in the dg operad $\RBSinfty$ using labeled planar rooted trees. The generator $m_n$ for $n \geqslant 2$ is represented  by the corolla with $n$ leaves and a black vertex, and the generator $R_n$ (resp. $S_n$), $n \geqslant 1$  is represented by a corolla with $n$ leaves and a white (resp. red) vertex:
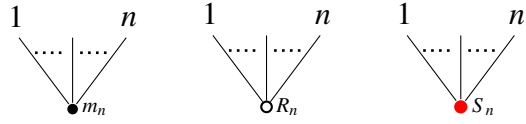
\begin{eqnarray*}
	\begin{tikzpicture}[scale=0.5]
		\tikzstyle{every node}=[thick,minimum size=4pt, inner sep=1pt]
		%\node(a) at (-4,0.5){\begin{huge}$\partial$\end{huge}};
		\node[circle, fill=black, label=right:{\tiny $m_n$}] (b0) at (-2,-0.5)  {};
		%  \node (1a) at (-0.5,1.5);
		\node (b1) at (-3.5,1.5)  [minimum size=0pt,label=above:$1$]{};
		\node (b2) at (-2,1.5)  [minimum size=0pt  ]{};
		\node (b3) at (-0.5,1.5)  [minimum size=0pt,label=above:$n$]{};
		\draw        (b0)--(b1);
		\draw        (b0)--(b2);
		\draw        (b0)--(b3);
		\draw [dotted,line width=1pt] (-3,1)--(-2.2,1);
		\draw [dotted,line width=1pt] (-1.8,1)--(-1,1);
	\end{tikzpicture}
	\hspace{8mm}
	\begin{tikzpicture}[scale=0.5]
		\tikzstyle{every node}=[thick,minimum size=4pt, inner sep=1pt]
		%	\node(a) at (-4,0.5){\begin{huge}$\partial$\end{huge}};
		\node[circle, draw, label=right:{\tiny $R_n$}] (b0) at (-2,-0.5)  {};
		%  \node (1a) at (-0.5,1.5);
		\node (b1) at (-3.5,1.5)  [minimum size=0pt,label=above:$1$]{};
		\node (b2) at (-2,1.5)  [minimum size=0pt,label=above: ]{};
		\node (b3) at (-0.5,1.5)  [minimum size=0pt,label=above:$n$]{};
		\draw        (b0)--(b1);
		\draw        (b0)--(b2);
		\draw        (b0)--(b3);
		\draw [dotted,line width=1pt] (-3,1)--(-2.2,1);
		\draw [dotted,line width=1pt] (-1.8,1)--(-1,1);
	\end{tikzpicture}
\hspace{8mm}
\begin{tikzpicture}[scale=0.5]
	\tikzstyle{every node}=[thick,minimum size=5pt, inner sep=1pt]
	%	\node(a) at (-4,0.5){\begin{huge}$\partial$\end{huge}};
	\node[fill=red, circle,, label=right:{\tiny $S_n$}] (b0) at (-2,-0.5)  {};
	%  \node (1a) at (-0.5,1.5);
	\node (b1) at (-3.5,1.5)  [minimum size=0pt,label=above:$1$]{};
	\node (b2) at (-2,1.5)  [minimum size=0pt,label=above: ]{};
	\node (b3) at (-0.5,1.5)  [minimum size=0pt,label=above:$n$]{};
	\draw        (b0)--(b1);
	\draw        (b0)--(b2);
	\draw        (b0)--(b3);
	\draw [dotted,line width=1pt] (-3,1)--(-2.2,1);
	\draw [dotted,line width=1pt] (-1.8,1)--(-1,1);
\end{tikzpicture}
\end{eqnarray*}
In this means, the action of the  differential operator $\partial$ on generators can be expressed by trees as follows:
\begin{eqnarray*}
	\begin{tikzpicture}[scale=0.6]
		\tikzstyle{every node}=[thick,minimum size=4pt, inner sep=1pt]
		\node(a) at (-3.3,0){\large $\partial$};
		\node[circle, fill=black, label=right:$m_n$] (b0) at (-2,-0.5)  {};
		%  \node (1a) at (-0.5,1.5);
		\node (b1) at (-3.5,1.5)  [minimum size=0pt,label=above:$1$]{};
		\node (b2) at (-2,1.5)  [minimum size=0pt  ]{};
		\node (b3) at (-0.5,1.5)  [minimum size=0pt,label=above:$n$]{};
		\draw        (b0)--(b1);
		\draw        (b0)--(b2);
		\draw        (b0)--(b3);
		\draw [dotted,line width=1pt] (-3,1)--(-2.2,1);
		\draw [dotted,line width=1pt] (-1.8,1)--(-1,1);
	\end{tikzpicture}
	&
	\begin{tikzpicture}
		\node(0){{$=\sum\limits_{j=2}^{n-1}\sum\limits_{i=1}^{n-j+1}(-1)^{i+j(n-i)}$}};
	\end{tikzpicture}
	&
	\begin{tikzpicture}[scale=0.7]
		\tikzstyle{every node}=[thick,minimum size=4pt, inner sep=1pt]
		\node(e0) at (0,-1.5)[circle, fill=black,label=right:$m_{n-j+1}$]{};
		\node(e1) at(-1.5,0){{\tiny$1$}};
		\node(e2-0) at (0,-0.5){{\tiny$i$}};
		\node(e3) at (1.5,0){{\tiny{$n-j+1$}}};
		\node(e2-1) at (0,0.5) [circle,fill=black,label=right: $m_j$]{};
		\node(e2-1-1) at (-1,1.5){{\tiny$1$}};
		\node(e2-1-2) at (1, 1.5){{\tiny $j$}};
		\draw [dotted,line width=1pt] (-0.7,-0.5)--(-0.2,-0.5);
		\draw [dotted,line width=1pt] (0.3,-0.5)--(0.8,-0.5);
		\draw [dotted,line width=1pt] (-0.4,1)--(0.4,1);
		\draw        (e0)--(e1);
		\draw         (e0)--(e3);
		\draw         (e0)--(e2-0);
		\draw         (e2-0)--(e2-1);
		\draw        (e2-1)--(e2-1-1);
		\draw        (e2-1)--(e2-1-2);
	\end{tikzpicture}	
\end{eqnarray*}

\begin{eqnarray*}
	\begin{tikzpicture}[scale=0.5]
		\tikzstyle{every node}=[thick,minimum size=4pt, inner sep=1pt]
		\node(a) at (-3.3,0){\large $\partial$};
		\node[circle, draw, label=right:{\tiny $R_n$}] (b0) at (-2,-0.5)  {};
		%  \node (1a) at (-0.5,1.5);
		\node (b1) at (-3.5,1.5)  [minimum size=0pt,label=above:$1$]{};
		\node (b2) at (-2,1.5)  [minimum size=0pt,label=above: ]{};
		\node (b3) at (-0.5,1.5)  [minimum size=0pt,label=above:$n$]{};
		\draw        (b0)--(b1);
		\draw        (b0)--(b2);
		\draw        (b0)--(b3);
		\draw [dotted,line width=1pt] (-3,1)--(-2.2,1);
		\draw [dotted,line width=1pt] (-1.8,1)--(-1,1);
	\end{tikzpicture}
	&
	\begin{tikzpicture}
		\node(0){{$= \sum\limits_{k=2}^n\sum\limits_{l_1+\cdots+l_k=n \atop l_1, \dots, l_k\geqslant 1}(-1)^{\alpha}$}};
	\end{tikzpicture}
	&
	\begin{tikzpicture}[scale=0.8]
		\tikzstyle{every node}=[thick,minimum size=4pt, inner sep=1pt]
		\node(e0) at (0,-1.5)[circle, fill=black,label=right:{\tiny $m_{k}$}]{};
		\node(e1) at(-1.2,-0.3)[circle, draw, label=left:{\tiny $R_{l_1}$}]{};
		\node(e1-1) at(-2,0.8){};
		\node(e1-2) at (-1,0.8){};
		% \node(e2-0) at (0,-0.7){{\tiny$2$}};
		\node(e3) at (1.2,-0.3)[draw, circle, label=right: {\tiny $R_{l_k}$}]{};
		\node(e3-1) at (1,0.8){};
		\node(e3-2) at (2,0.8){};
		\node(e2-1) at (-0.3,-0.3) [draw,circle,label=right: {\tiny $R_{l_2}$}]{};
		\node(e2-1-1) at (-0.5,0.8){};
		\node(e2-1-2) at (0.5, 0.8){};
		%\draw [dotted,line width=1pt] (-0.7,-0.5)--(-0.2,-0.5);
		\draw [dotted,line width=0.8pt] (0.2,-0.6)--(0.7,-0.6);
		\draw [dotted,line width=0.8pt] (-0.3,0.5)--(0.1,0.5);
		\draw [dotted,line width=0.8pt] (-1.6,0.5)--(-1.2, 0.5);
		\draw [dotted, line width=0.8pt] (1.2,0.5)--(1.6,0.5);
		\draw        (e0)--(e1);
		\draw         (e0)--(e3);
		\draw         (e1)--(e1-1);
		\draw          (e1)--(e1-2);
		\draw         (e0)--(e2-1);
		\draw        (e2-1)--(e2-1-1);
		\draw        (e2-1)--(e2-1-2);
		\draw        (e3)--(e3-1);
		\draw        (e3)--(e3-2);
		
	\end{tikzpicture}	\\
	&\begin{tikzpicture}
		\node(0){{$+\sum\limits_{{\small\substack{2\leqslant p\leqslant n \\ 1\leqslant j\leqslant p
				}}}\sum\limits_{\small\substack{ r_1+\dots+r_p=n\\r_1, \dots, r_q\geqslant 1\\1\leqslant i\leqslant r_1 }}(-1)^{\beta}$}};
	\end{tikzpicture}&
	\begin{tikzpicture}[scale=0.8,descr/.style={fill=white}]
		\tikzstyle{every node}=[minimum size=4pt, inner sep=1pt]
		\node(v0) at (0,0)[draw, circle,label=left:{\tiny $R_{r_1}$}]{};
		\node(v1-1) at (-2,1){};
		\node(v1-2) at(0,1.2)[fill=black,draw,circle,label=right: {\scriptsize $\ \ m_p$}]{};
		\node(v1-3) at(2,1){};
		\node(v2-1) at(-1.5,2.8)[draw, circle,label=left:{\tiny $R_{r_2}$}]{};
		\node(v2-2) at (-0.5, 2.8)[draw, circle,label=left:{\tiny $R_{r_j}$}]{};
		\node(v2-3) at (0,2.9){\tiny $j$}{};
		\node(v2-4) at(1.6,2.8)[fill=red, circle,label=right:{\tiny{$S_{r_p}$}}]{};
		\node(v2-5) at(0.6,2.8)[fill=red, circle,label=right:{\tiny $S_{r_{j+1}}$}]{};
		\node(v3-5) at (-2,3.5){};
		\node(v3-6) at (-1.3,3.5){};
		\node(v3-1) at (-1,3.5){};
		\node(v3-2) at (-0.3,3.5){};
		\node(v3-3) at (1.3,3.5){};
		\node(v3-4) at(2,3.5){};
		\node(v3-7) at (1,3.5){};
		\node(v3-8) at (0.3,3.5){};
		\draw(v0)--(v1-1);
		\draw(v0)--(v1-3);
		\path[-,font=\scriptsize]
		(v0) edge node[descr]{{\tiny$i$}} (v1-2);
		\draw(v1-2)--(v2-1);
		\draw(v1-2)--(v2-3);
		\draw(v1-2)--(v2-2);
		\draw(v1-2)--(v2-4);
		\draw(v1-2)--(v2-5);
		\draw(v2-1)--(v3-5);
		\draw(v2-1)--(v3-6);
		\draw(v2-2)--(v3-1);
		\draw(v2-2)--(v3-2);
		\draw(v2-4)--(v3-3);
		\draw(v2-4)--(v3-4);
		\draw(v2-5)--(v3-7);
		\draw(v2-5)--(v3-8);
		\draw[dotted,line width=0.8pt](-1,0.6)--(-0.1,0.6);
		\draw[dotted,line width=0.8pt](0.1,0.6)--(1,0.6);
		\draw[dotted,line width=0.8pt](0.6,2.4)--(1,2.4);
		\draw[dotted,line width=0.8pt](-0.9,2.4)--(-0.5,2.4);
		\draw[dotted,line width=0.8pt](-1.7,3.2)--(-1.4,3.2);
		\draw[dotted,line width=0.8pt](-0.7,3.2)--(-0.4,3.2);
		\draw[dotted,line width=0.8pt](1.7,3.2)--(1.4,3.2);
		\draw[dotted,line width=0.8pt](0.8,3.2)--(0.5,3.2);
	\end{tikzpicture}
\end{eqnarray*}

\begin{eqnarray*}
	\begin{tikzpicture}[scale=0.5]
		\tikzstyle{every node}=[thick,minimum size=4pt, inner sep=1pt]
		\node(a) at (-3.3,0){\large$\partial$};
		\node[fill=red, circle, label=right:{\tiny $S_n$}] (b0) at (-2,-0.5)  {};
		%  \node (1a) at (-0.5,1.5);
		\node (b1) at (-3.5,1.5)  [minimum size=0pt,label=above:$1$]{};
		\node (b2) at (-2,1.5)  [minimum size=0pt,label=above: ]{};
		\node (b3) at (-0.5,1.5)  [minimum size=0pt,label=above:$n$]{};
		\draw        (b0)--(b1);
		\draw        (b0)--(b2);
		\draw        (b0)--(b3);
		\draw [dotted,line width=1pt] (-3,1)--(-2.2,1);
		\draw [dotted,line width=1pt] (-1.8,1)--(-1,1);
	\end{tikzpicture}
	&
	\begin{tikzpicture}
		\node(0){{$=  \sum\limits_{k=2}^n\sum\limits_{l_1+\cdots+l_k=n \atop l_1, \dots, l_k\geqslant 1}(-1)^{\alpha}$}};
	\end{tikzpicture}
	&
	\begin{tikzpicture}[scale=0.8]
		\tikzstyle{every node}=[thick,minimum size=4pt, inner sep=1pt]
		\node(e0) at (0,-1.5)[circle, fill=black,label=right:{\tiny $m_{k}$}]{};
		\node(e1) at(-1.2,-0.3)[fill=red, circle, label=left:{\tiny $S_{l_1}$}]{};
		\node(e1-1) at(-2,0.8){};
		\node(e1-2) at (-1,0.8){};
		% \node(e2-0) at (0,-0.7){{\tiny$2$}};
		\node(e3) at (1.2,-0.3)[fill=red, circle, label=right: {\tiny $S_{l_k}$}]{};
		\node(e3-1) at (1,0.8){};
		\node(e3-2) at (2,0.8){};
		\node(e2-1) at (-0.3,-0.3) [fill=red, circle,label=right:{\tiny  $S_{l_2}$}]{};
		\node(e2-1-1) at (-0.5,0.8){};
		\node(e2-1-2) at (0.5, 0.8){};
		%\draw [dotted,line width=1pt] (-0.7,-0.5)--(-0.2,-0.5);
		\draw [dotted,line width=1pt] (0.2,-0.6)--(0.7,-0.6);
		\draw [dotted,line width=1pt] (-0.3,0.5)--(0.1,0.5);
		\draw [dotted,line width=1pt] (-1.6,0.5)--(-1.2, 0.5);
		\draw [dotted, line width=1pt] (1.2,0.5)--(1.6,0.5);
		\draw        (e0)--(e1);
		\draw         (e0)--(e3);
		\draw         (e1)--(e1-1);
		\draw          (e1)--(e1-2);
		\draw         (e0)--(e2-1);
		\draw        (e2-1)--(e2-1-1);
		\draw        (e2-1)--(e2-1-2);
		\draw        (e3)--(e3-1);
		\draw        (e3)--(e3-2);
		
	\end{tikzpicture}	\\
	&\begin{tikzpicture}
		\node(0){{$+
				\sum\limits_{{\small\substack{2\leqslant p\leqslant n \\ 1\leqslant j\leqslant p
				}}}\sum\limits_{\small\substack{ r_1+\dots+r_p=n\\r_1, \dots, r_q\geqslant 1\\1\leqslant i\leqslant r_1 }}(-1)^{\beta}$}};
	\end{tikzpicture}&
	\begin{tikzpicture}[scale=0.8,descr/.style={fill=white}]
		\tikzstyle{every node}=[minimum size=4pt, inner sep=1pt]
		
		\node(v0) at (0,0)[fill=red, circle,,label=left:{\tiny $S_{r_1}$}]{};
		\node(v1-1) at (-2,1){};
		\node(v1-2) at(0,1.2)[fill=black,draw,circle,label=right: {\scriptsize $\ \ m_p$}]{};
		\node(v1-3) at(2,1){};
		\node(v2-1) at(-1.4,2.8)[draw, circle,label=left:{\tiny $R_{r_2}$}]{};
		\node(v2-2) at (-0.5, 2.8)[draw, circle,label=left:{\tiny $R_{r_j}$}]{};
		\node(v2-3) at (0,2.9){\tiny $j$}{};
		\node(v2-4) at(1.5,2.8)[fill=red, circle,label=right:{\tiny $S_{r_p}$}]{};
		\node(v2-5) at(0.6,2.8)[fill=red, circle,label=right:{\tiny $S_{r_{j+1}}$}]{};
		\node(v3-5) at (-2,3.5){};
		\node(v3-6) at (-1.3,3.5){};
		\node(v3-1) at (-1,3.5){};
		\node(v3-2) at (-0.3,3.5){};
		\node(v3-3) at (1.3,3.5){};
		\node(v3-4) at(2,3.5){};
		\node(v3-7) at (1,3.5){};
		\node(v3-8) at (0.3,3.5){};
		\draw(v0)--(v1-1);
		\draw(v0)--(v1-3);
		\path[-,font=\scriptsize]
		(v0) edge node[descr]{{\tiny$i$}} (v1-2);
		\draw(v1-2)--(v2-1);
		\draw(v1-2)--(v2-3);
		\draw(v1-2)--(v2-2);
		\draw(v1-2)--(v2-4);
		\draw(v1-2)--(v2-5);
		\draw(v2-1)--(v3-5);
		\draw(v2-1)--(v3-6);
		\draw(v2-2)--(v3-1);
		\draw(v2-2)--(v3-2);
		\draw(v2-4)--(v3-3);
		\draw(v2-4)--(v3-4);
		\draw(v2-5)--(v3-7);
		\draw(v2-5)--(v3-8);
		\draw[dotted,line width=0.8pt](-1,0.6)--(-0.1,0.6);
		\draw[dotted,line width=0.8pt](0.1,0.6)--(1,0.6);
		\draw[dotted,line width=0.8pt](0.6,2.4)--(1,2.4);
		\draw[dotted,line width=0.8pt](-0.9,2.4)--(-0.5,2.4);
		\draw[dotted,line width=0.8pt](-1.7,3.2)--(-1.4,3.2);
		\draw[dotted,line width=0.8pt](-0.7,3.2)--(-0.4,3.2);
		\draw[dotted,line width=0.8pt](1.7,3.2)--(1.4,3.2);
		\draw[dotted,line width=0.8pt](0.8,3.2)--(0.5,3.2);
	\end{tikzpicture}
\end{eqnarray*}

The following theorem is the main result of this section, with its proof comprising the remainder of the section.
\begin{thm}\label{Thm: Minimal model}
The dg operad $\RBSinfty$ is the minimal model of the operad $\RBS$.
\end{thm}

\bigskip

In order to prove Theorem~\ref{Thm: Minimal model},  we are going to construct a quasi-isomorphism of dg operads $\RBSinfty\xrightarrow{~}\RBS$, where $\RBS$ is considered as a dg operad concentrated in degree 0.

Recall that the  operad   $\RBS$  is generated by two unary operators $R$, $S$ and a binary operator $\mu$ with  operadic relations:\[ \mu \circ_{1} \mu  - \mu \circ_{2} \mu ,\]
$$ (\mu\circ_1 R)\circ_2 R-(R \circ_1\mu)\circ_1 R-(R \circ_1\mu)\circ_2 S ,$$
and
$$ (\mu\circ_1 S)\circ_2 S-(S \circ_1\mu)\circ_1 R-(S \circ_1\mu)\circ_2 S .$$
\begin{lem} There exists a  natural surjective  map $\phi: \RBSinfty\twoheadrightarrow \RBS$ of dg operads sending $m_2$ (resp. $R_1$, $S_1$, all other generators) to $\mu$ (resp. $R$, $S$, zero), which induces an isomorphism
$\rmH_0(\RBSinfty)  \cong \RBS$.

\end{lem}
\begin{proof}
The degree zero part of $\RBSinfty$ is the free graded operad  generated by $\{m_2, R_1,S_1\}$. The image of $\partial$  in this  degree zero part is the operadic ideal   generated by $\partial R_2,\partial S_2$ and $\partial m_3$. By definition, we have
\begin{eqnarray*}
	\partial(m_3)&=&  m_{2} \circ_{1} m_{2}-  m_{2} \circ_{2} m_{2},\\
	\partial (R_2)&=& R_1\circ_1(m_2\circ_1 R_1)+R_1\circ_1(m_2\circ_2 S_1)-(m_2\circ_1R_1)\circ_2 R_1,\\
	\partial (S_2)&=& S_1\circ_1(m_2\circ_1 R_1)+S_1\circ_1(m_2\circ_2 S_1)-(m_2\circ_1S_1)\circ_2 S_1.\\
\end{eqnarray*}
Thus  the map  $\phi: \RBSinfty\twoheadrightarrow \RBS$ induces the isomorphism  $\rmH_0(\RBSinfty)  \cong \RBS$.
\end{proof}

\medskip
To show that  $\phi: \RBSinfty\twoheadrightarrow \RBS$ is a quasi-isomorphism, we just need to prove that $\rmH_i(\RBSinfty)=0$ for all $i\geqslant 1$. To achieve this, we will  construct a filtration on the dg operad $\RBS_\infty$, and we will prove that the homology of $0$-the page of the spectral sequence associated to the filtration  concentrates in degree 0 (see Theorem~\ref{Thm: inclusion-exclusion model for the monomials of RBS}). Then using convergence argument of spectral sequence, we have that the homology of $\RBS_\infty$ concentrates in degree $0$.

% {\red To achieve this, we will first construct a differential graded operad $_m\RBS_\infty$ whose homology is constructed in degree $0$. Then we will  construct a filtration on the dg operad $\RBS_\infty$, and we will see that the $0$-the page of the spectral sequence associated the filtration is exactly the dg operad $_m\RBS_\infty$. The using convergence argument of spectral sequence, we can prove that the homology $\RBS_\infty$ concentrates in degreee $0$.}
%
%
%\bigskip
%
%{\red \bf Introduce the dg operad $_m\RBS_\infty$ here.}

%Let $\mathfrak{Tr}^{(2, 1)}$ be the free operad generated by two unary operations $R$, $S$, and a binary operation $\mu.$ Then $\RBS\cong \mathfrak{Tr}^{(2, 1)}/<\tilde{G}>$ where $\tilde{G}$ is the defining relations for Rota-Baxter system of Lie  algebras.

\bigskip

\textbf{Proof of Theorem~\ref{Thm: Minimal model}:}
%
%We have proved that the natural map $\phi: \RBSinfty\twoheadrightarrow \RBS$ is a surjective quasi-isomorphism, and  it can be easily seen that the differential $\partial$ on $\RBSinfty$ satisfies the conditions $(1)(2)$ in Definition~\ref{Def: Minimal model of operads}.  We are done!

%\begin{remark}\label{rem: dotsenko}We are grateful to  Dotsenko who kindly pointed out an alternative proof of Theorem~\ref{Thm: Minimal model}.\end{remark}
\begin{proof}
It is obvious that the dg operad $\RBS_\infty$ satisfies conditions $\rm{(i) (ii)}$ of Definition~\ref{Def: Minimal model of operads}. So we just need to prove that $\phi: \RBS_\infty\rightarrow \RBS$ is a quasi-isomorphism.
Let $\mathfrak{Tr}^{(2, 1)}$ be the free operad generated by two unary operations $R$, $S$, and a binary operation $\mu$.
Now, let's introduce a new ordering $\prec$ on the the set $\mathfrak{Tr}^{(2, 1)}(n)$ as follows: For two tree monomials $\mathcal{T, T'}$ in $\mathfrak{Tr}^{(2, 1)}(n)$,
\begin{itemize}
	\item[(1)]If $\mbox{weight}(\mathcal{T})<\mbox{weight}( \mathcal{T'})$, then $\mathcal{T}\prec\mathcal{T'}$.
	\item[(2)] If $\mbox{weight}(\mathcal{T})=\mbox{weight}( \mathcal{T'})$, compare $\mathcal{T}$ with $\mathcal{T'}$ via the natural  path-extension-lexicographic induced by setting $R\prec S\prec\mu.$
\end{itemize}
With respect to the ordering $\prec$, the set $\mathfrak{Tr}^{(2, 1)}(n)$ becomes a totally ordered set: $\mathfrak{Tr}^{(2, 1)}(n)=\{x_1\prec x_2\prec x_3\prec\dots\}\cong \mathbb{N}^+$. Now, we define a map $\omega:\RBSinfty\rightarrow \mathfrak{Tr}^{(2, 1)}$ by replacing vertices $m_n (n\geqslant 2), R_n(n\geqslant 1), S_n(n\geqslant 1)$ by $\underbrace{\mu\circ_1\mu\circ_1\mu\dots\circ_1\mu}_{n-1}$, $ \underbrace{R\circ_1\mu\circ_1R\circ_1\mu\circ\dots\circ_1R\circ_1\mu\circ_1R}_{2n-1},$ and $ \underbrace{S\circ_1\mu\circ_1R\circ_1\mu\circ\dots\circ_1R\circ_1\mu\circ_1R}_{2n-1}$ in $\mathfrak{Tr}^{(2, 1)}$ respectively. Notice that different tree monomials in $\RBSinfty$ may have same image in $\mathfrak{Tr}^{(2, 1)}$ under the action $\omega$.
Define $\mathcal{F}^n_i $  to be the subspace of $\RBSinfty(n)$ spanned by the tree monomials $\mathcal{R}$ with $\omega(\mathcal{R})$ smaller than or equal to $x_i$ with respect to  the ordering $\prec$ on $\mathfrak{Tr}^{(2, 1)}(n)$.
Then we get a bounded below and exhaustive filtration for $\RBSinfty(n)$ \[0=\mathcal{F}^n_0\subset \mathcal{F}^n_1\subset\mathcal{F}^n_2\subset \cdots \subset \RBSinfty(n) .\]
It  can be easily seen that the filtration is compatible with the differential $\partial$ on $\RBSinfty(n)$.
Thus the spectral sequence associated to the filtration will converge to the homology of $\RBSinfty(n)$.

 Moreover, one can prove that there is an isomorphism of complexes
$$\bigoplus_{i\geq0}\mathcal{F}^n_{i+1}/\mathcal{F}^n_i\cong {}_m\RBSinfty(n),$$
where ${}_m\RBS_\infty$ is a quasi-free dg operad generated by operations $\{m_n\}_{n\geq 2}$, $\{R_n\}_{n\geq1}$ and $\{S_n\}_{n\geq1}$. The action of the differential $\bar\partial$ on these generators in ${}_m\RBS_\infty$ is given by the following :
$$\begin{array}{rcl}\label{Eq: defining mHRB 1} &&\forall n\geqslant 2,\quad
\bar{\partial}{m_n} = \sum_{j=2}^{n-1}(-1)^{1+j(n-1)}m_{n-j+1}\circ_1 m_j,
	\\&& \forall n\geqslant 1,   \quad   \bar{\partial} R_n=   \sum\limits_{\small\substack{ r_1+ r_2=n\\r_1 , r_2\geqslant 1}}(-1)^{r_1(r_2-1)}\Big(R_{r_1}\circ_1 m_2\circ_1 R_{r_2}\Big) ,\\
	& &\forall n\geqslant 1,   \quad   \bar{\partial} S_n=   \sum\limits_{\small\substack{ r_1+ r_2=n\\r_1 , r_2\geqslant 1}}(-1)^{r_1(r_2-1)}\Big(S_{r_1}\circ_1 m_2\circ_1 R_{r_2}\Big) .
\end{array}$$
According to Theorem~\ref*{Thm: inclusion-exclusion model for the monomials of RBS} in Appendix~\ref{Appendix: the minimal model for monomimal operad}, we have  that  ${}_m\RBSinfty(n)$ is the minimal model of a certain operad.  Thus, all positive homologies of ${}_m\RBSinfty(n)$ vanish. By classical spectral sequence argument, we have that all positive homologies of $\RBSinfty(n)$ are trivial.

\end{proof}

%This provides another proof for   Theorem \ref{Thm: Minimal model}.

By the general theory of minimal models of operads, the generators of the minimal model for an operad is exactly the desuspension of its Quillen homology, i.e., the homology of the bar construction of the operad. So for the operad $\RBS$,  denote by $\mathrm{B}(\RBS)$  the bar construction of $\RBS$ and   we have the following result.

\begin{cor} There exists   a quasi-isomorphism of homotopy cooperads between ${\RBS^\ac}$ and $ \mathrm{B}(\RBS),$  which induces  an isomorphism of
graded collections  $${\mathrm{H}}_\bullet(\mathrm{B}(\RBS))\cong {\RBS^\ac}.$$

\end{cor}

\bigskip 

\section{Homotopy Rota-Baxter  systems and infinity-Yang-Baxter pairs}\label{Section: Homotopy Rota-Baxter systems and infinity-Yang-Baxter pairs}

In Section~\ref{Section: minimal model RBSA}, we constructed the minimal model for Rota-Baxter systems operad. Building on this foundation, we now introduce the concept of homotopy Rota-Baxter systems, which serves as higher homotopy version of the concept of Rota-Baxter systems. Additionally, we define the notion of associative infinity-Yang-Baxter pairs and investigate the relationship between these two structures, thereby generalizing the classical correspondence between Rota-Baxter systems and associative Yang-Baxter pairs to a higher homotopical setting.

\subsection{Homotopy Rota-Baxter systems}\

Firstly, let's introduce the notion of homotopy Rota-Baxter systems.
\begin{defn}
	Let $(V,d_V)$ be a complex. A homotopy Rota-Baxter   system on $V$ is defined to be a morphism of dg operads from $\RBSinfty$ to the endomorphism operad $\End_V$.
\end{defn}

Let $(V, d_V)$ be an algebra over the operad $\RBSinfty$.
Still denote by
$  m_n: V^{\ot n}\rightarrow V , n\geqslant 2$ (resp. $ R_n, S_n: V^{\ot n}\rightarrow V, n\geqslant 1$)  the image of $m_n\in \RBSinfty$ (resp. $R_n,S_n\in \RBSinfty$).
We also rewrite $m_1=d_V$. Then Equations~\eqref{Eq: defining HRB 1} and \eqref{Eq: defining HRB 2} give
\begin{eqnarray}\label{Eq: stasheff-id}
	\sum_{    i+j+k= n,\atop
		i, k\geqslant 0, j\geqslant 1 } (-1)^{i+jk}m_{i+1+k}\circ\Big(\id^{\ot i}\ot m_j\ot \id^{\ot k}\Big)=0,
\end{eqnarray}
\begin{eqnarray} \label{Eq: homotopy RB-operator-version-1}
	\sum\limits_{ l_1+\dots+l_k=n,\atop
		l_1, \dots, l_k\geqslant 1 } (-1)^{\delta}m_k\circ\Big(R_{l_1}\ot \cdots \ot R_{l_k}\Big)=\sum\limits_{1\leqslant j\leqslant p}\sum\limits_{  r_1+\dots+r_p=n,\atop
		r_1, \dots, r_p\geqslant 1 }
\end{eqnarray}
$$ \quad  \quad\quad\quad\quad\quad  (-1)^\eta
R_{r_1}\circ\Big(\id^{\ot i}\ot m_p\circ( R_{r_2}\ot   \cdots\ot R_{r_j}\ot\id \ot S_{r_{j+1}}\ot \cdots\ot S_{r_p})\ot \id^{\ot k}\Big),
$$ and
\begin{eqnarray} \label{Eq: homotopy RB-operator-version-2}
	\sum\limits_{ l_1+\dots+l_k=n,\atop
		l_1, \dots, l_k\geqslant 1 } (-1)^{\delta}m_k\circ\Big(S_{l_1}\ot \cdots \ot S_{l_k}\Big)=\sum\limits_{1\leqslant j\leqslant p}\sum\limits_{  r_1+\dots+r_p=n,\atop
		r_1, \dots, r_p\geqslant 1 }
\end{eqnarray}
$$ \quad  \quad\quad\quad\quad\quad  (-1)^\eta
S_{r_1}\circ\Big(\id^{\ot i}\ot m_p\circ( R_{r_2}\ot   \cdots\ot R_{r_j}\ot\id \ot S_{r_{j+1}}\ot \cdots\ot S_{r_p})\ot \id^{\ot k}\Big),
$$
where
\begin{align*}\delta&=\frac{k(k-1)}{2}+\sum_{j=1}^k(k-j)l_j,\\
	\eta&=i+(p+\sum\limits_{j=2}^p(r_j-1))k+\sum_{t=2}^j(r_t-1)+\sum_{t=2}^p(r_t-1)(p-t).
\end{align*}
{Equation~(\ref{Eq: stasheff-id}) is exactly the Stasheff identity   defining  $A_\infty$-algebras \cite{Sta63}. In particular,  the operator $m_1$ is a differential on $V$  and the operator $m_2$ induces an associative algebra structure on the homology   $\rmH_\bullet(V, m_1)$.

We obtain thus the following equivalent definition of  homotopy Rota-Baxter  system.
%The operators $(\{R_n\}_{n\geq1}, \{S_n\}_{n\geq1})$ are called homotopy Rota-Baxter     system operator on $A_\infty$-algebra $(V,\{m_n\}_{n\geq1})$.

\begin{defn}
	Let $V$ be a graded space. A homotopy Rota-Baxter  system on $V$ consists of three families of graded maps
	$ \{m_n: V^{\ot n}\rightarrow V \}_{n\geqslant 1}$  and $\{R_n,S_n: V^{\ot n}\rightarrow V\}_{n\geqslant 1}$  with $|m_n|=n-2,   |R_n|=n-1, |S_n|=n-1$  that subject to
	Equations~\eqref{Eq: stasheff-id}-\eqref{Eq: homotopy RB-operator-version-2}. 
	
	The graded operations $\big(\{R_n\}_{n\geq1}, \{S_n\}_{n\geq 1}\big)$ are called a homotopy coupled Rota-Baxter operator on the $A_\infty$-algebra $(V,\{m_n\}_{n\geq1})$.
\end{defn}
}

%\begin{defn}
%	Let $V$ be a graded space. A homotopy Rota-Baxter   system   on $V$ is  two family of graded maps
%	$  m_n: V^{\ot n}\rightarrow V , n\geqslant 1$  and $ R_n,S_n: V^{\ot n}\rightarrow V, n\geqslant 1$  with $|m_n|=n-2, |R_n|=n-1, |S_n|=n-1$ are two $A_\infty$-morphisms from $(V,\{\tilde{m}_n\})$ to $(V,\{m_n\})$, where
%	$$ \tilde{m}_n=\sum_{r_1+\cdots+r_k=n,\atop  0\leqslant j\leqslant k}m_k\circ( R_{r_1}\ot   \cdots \ot R_{r_j}\ot\id \ot S_{r_{j+1}}\ot \cdots\ot S_{r_k})  .$$
%\end{defn}

\begin{exam}
	Expanding  Equation~\eqref{Eq: homotopy RB-operator-version-2}  for small  $n$'s gives the following:
	\begin{itemize}
		\item[(i)]
		When $n=1$, $|R_1|=|S_1|=0$, $$  m_1\circ R_1=R_1\circ m_1,$$ and $$  m_1\circ S_1=S_1\circ m_1,$$
		which implies that $R_1,S_1: (V, m_1)\to (V, m_1)$ are  chain maps;
		
		\item[(ii)]
		when $n=2$,  $|R_2|=|S_2|=1$,
		$$\begin{array}{ll}  &m_2\circ(R_1\ot R_1)-R_1\circ m_2\circ (\id\ot S_1)-R_1\circ m_2\circ (R_1\ot \id) \\
			=&-\partial(R_2)= -\big(m_1\circ R_2+R_2\circ (\id\ot m_1)+R_2\circ(m_1\ot \id)\big),
		\end{array}$$
		and
		
		$$\begin{array}{ll}  &m_2\circ(S_1\ot S_1)-S_1\circ m_2\circ (\id\ot S_1)-S_1\circ m_2\circ (R_1\ot \id) \\
			=&-\partial(S_2)= -\big(m_1\circ S_2+S_2\circ (\id\ot m_1)+S_2\circ(m_1\ot \id)\big),
		\end{array}$$
	{	which 	show that $(R_1, S_1)$ is  coupled Rota-Baxter operator with respect to $m_2$, but up to homotopy provided by the pair of operators $(R_2, S_2)$.}

	\end{itemize}
	Observe that for a homotopy Rota-Baxter   algebra $(V;  \{m_n\}_{n\geqslant 1}, \{R_n\}_{n\geqslant 1}, \{S_n\}_{n\geqslant 1})$,  its homology $\rmH_\bullet(V, m_1)$ endowed with the operators induced by $m_2$, $ R_1$ and  $ S_1$ is  a usual  Rota-Baxter system.
\end{exam}

In particular, if the underlying $A_\infty$-algebra $(V,\{m_n\}_{n\geqslant 1})$ in the definition of homotopy Rota-Baxter system is just a differential graded algebra, i.e., $m_n=0$ for all $n\geqslant 3$,  Equation~\eqref{Eq: homotopy RB-operator-version-1} and Equation~\eqref{Eq: homotopy RB-operator-version-2} will degenerate into the following forms:
\begin{align}\label{Eq: homotopy RB-operator-version-3}&\\
	\nonumber&m_1\circ R_n+\sum_{i+j=n}(-1)^{i+1}m_2(R_{i}\ot R_j)\\
	\nonumber =&\sum_{0\leqslant i\leqslant p-1}\sum_{p+q=n}(-1)^{i+(q-1)(p-i-1)}R_p\circ\big(\id^{\ot i}\ot m_2\circ (R_q\ot\id)\ot \id^{\ot p-i-1}\big)\\
	&\nonumber +(-1)^{n-1}\sum_{i=0}^{n-1}R_n(\id^{\ot i}\ot m_1\ot \id^{\ot n-i-1})+\sum_{0\leqslant i\leqslant p-1}\sum_{p+q=n}(-1)^{i+(q-1)(p-i-1)}R_p\circ\big(\id^{\ot i}\ot m_2\circ (\id\ot S_q)\ot \id^{\ot p-i-1}\big)
	\end{align}
	\begin{align}\label{Eq: homotopy RB-operator-version-4}
		&\\
		\nonumber &m_1\circ S_n+\sum_{i+j=n}(-1)^{i+1}m_2(S_{i}\ot S_j)\\
		\nonumber =&\sum_{0\leqslant i\leqslant p-1}\sum_{p+q=n}(-1)^{i+(q-1)(p-i)}S_p\circ\big(\id^{\ot i}\ot m_2\circ (R_q\ot\id)\ot \id^{\ot p-i-1}\big)\\
		&\nonumber +(-1)^{n-1}\sum_{i=0}^{n-1}S_n(\id^{\ot i}\ot m_1\ot \id^{\ot n-i-1})+\sum_{0\leqslant i\leqslant p-1}\sum_{p+q=n}(-1)^{i+(q-1)(p-i)}S_p\circ\big(\id^{\ot i}\ot m_2\circ (\id\ot S_q)\ot \id^{\ot p-i-1}\big).
	\end{align}

%Then we can introduce the following notion.
%\begin{defn}Let $(A,d,m_2)$ be a differential graded algebra. A  homotopy Rota-Baxter pair on $A$ consists of
%	\end{defn}

\subsection{Associative infinity-Yang-Baxter pairs}\

In this subsection, we firstly introduce the notion of  associative  infinity-Yang-Baxter pairs, which can be considered as the higher version of associative Yang-Baxter  pairs (Definition~\ref{definition-AYBEP}). Then, we will give an infinity version of Theorem~\ref{YBP and RBS}.

\begin{defn}\label{Def:associative Yang-Baxter infinty pairs}
	Let $A$ be a unitary graded algebra.   An associative infinity-Yang-Baxter pair  consists of two families of elements $(\{r_n\}_{n\geq1},\{s_n\}_{n\geq1})$ with $r_1=s_1$, $|r_n|=|s_n|=n-2$ for $n\geq1$
	 $$r_n=\sum r^{[1]}\otimes\cdots\otimes r^{[n]},s_n=\sum s^{[1]}\otimes\cdots\otimes s^{[n]}\in A^{\otimes n}, $$
	that satisfy	the following equations: for $n\geq 0$
	\begin{eqnarray}  \label{Eq: homotopy AYBE-version-1}
		 &&\sum_{i+j=n,\atop i,j\geq 1}(-1)^{1+i}r_{i+1}^{1,\ldots,i+1}\cdot r_{j+1}^{i+1,\ldots,n} \\
 \nonumber&=&\sum_{i+j=n,\atop i,j\geq 0}\sum_{s=1}^{i+1} (-1)^{s-1+(j-1)(n-s-j-1)}\left( r_{i+1}^{1,\ldots,s,s+j+1,\ldots,n+1}\cdot r_{j+1}^{s,\ldots,s+j} +  (-1)^{(1-j)i} s_{j+1}^{s+1,\ldots,s+j+1}\cdot r_{i+1}^{1,\ldots,s,s+j+1,\ldots,n+1}\right),
\end{eqnarray}
		\begin{eqnarray}
 \label{Eq: homotopy AYBE-version-2}
&&\sum_{i+j=n,\atop i,j\geq 1}(-1)^{1+i}s_{i+1}^{1,\ldots,i+1}\cdot s_{j+1}^{i+1,\ldots,n}\\
\nonumber &=&\sum_{i+j=n,\atop i,j\geq 0}\sum_{s=1}^{i+1} (-1)^{s-1+(j-1)(n-s-j-1)}\left( s_{i+1}^{1,\ldots,s,s+j+1,\ldots,n+1}\cdot r_{j+1}^{s,\ldots,s+j} +  (-1)^{(1-j)i}s_{j+1}^{s+1,\ldots,s+j+1}  \cdot  s_{i+1}^{1,\ldots,s,s+j+1,\ldots,n+1} \right),
 \end{eqnarray}
	where  $$r_{m}^{k_1+1,k_1+k_2+2,\ldots ,\sum_{s=1}^m k_s+m}=\sum 1^{\otimes k_1}\otimes r^{[1]}\otimes1^{\otimes k_2}\otimes\cdots\otimes1^{\otimes k_m}\otimes r^{[m]}\otimes1^{\otimes m+1}$$ and $$s_{m}^{k_1+1,k_1+k_2+2,\ldots ,\sum_{s=1}^m k_s+m}=\sum 1^{\otimes k_1}\otimes s^{[1]}\otimes1^{\otimes k_2}\otimes \cdots\otimes1^{\otimes k_m}\otimes s^{[m]}\otimes1^{\otimes m+1}, $$ for $\sum_{s=1}^m k_s+m=n+1$, $m\geq1$  and $k_1\ldots,k_{m+1}\geq0$ .
\end{defn}

	 For small integers  $n$ , the associative  infinity-Yang-Baxter pair has the following forms:
	 \begin{itemize}
	 	\item [(i)]  When $n=0$, $r_1=s_1$ and $|r_1|=|s_1|=-1$. Denote $d=r_1=s_1$. Equation~\eqref{Eq: homotopy AYBE-version-1} and Equation~\eqref{Eq: homotopy AYBE-version-2} both read as $d\cdot d=0$, which implies that the operator $ \partial=[-,d]:A\rightarrow A$ is a square-zero derivation on $A$. Thus $(A,\partial)$ is a dg algebra.
	 	\item [(ii)]    when $n=1$, $|r_1|=|s_1|=-1$, $|r_2|=|s_2|=0$,
	 	\[  \partial_{A\otimes A} (r_2)=-\left( (r_1^1+r_1^2) r_2^{12}-r_2^{12}(r_1^1+r_1^2)\right) =0,\]
	 	and
	 	\[\partial_{A\otimes A} (s_2)= -\left( (s_1^1+s_1^2) s_2^{12}-s_2^{12}(s_1^1+s_1^2)\right) =0,\]
	 	which shows that $r_2,s_2\in A\otimes A$ are closed cycles.
	 	\item  [(iii)] when $n=2$, 	$|r_1|=|s_1|=-1$, $|r_2|=|s_2|=0$,	$|r_3|=|s_3|=1$,
	 	
	 		$$\begin{array}{ll}  & r_2^{13} r_2^{12}-r_2^{12} r_2^{23}+s_2^{23} r_2^{13} \\
	 		=& -\big((s_1^1+s_2^2+s_2^3)r_3^{123}+r_3^{123}(r_1^1+r_2^2+r_2^3)\big)=\partial_{A^{\otimes 3}}(r_3),
	 	\end{array}$$
	 	and
	 		$$\begin{array}{ll}  & s_2^{13} r_2^{12}-s_2^{12} s_2^{23}+s_2^{23} s_2^{13} \\
	 		=& -\big((s_1^1+s_2^2+s_2^3)s_3^{123}+s_3^{123}(r_1^1+r_2^2+r_2^3)\big)=\partial_{A^{\otimes 3}}(s_3),
	 	\end{array}$$
 	which 	shows that $(r_2,s_2)$ is a  usual associative Yang-Baxter pair (Definition~\ref{definition-AYBEP}), but up to homotopy given by the elements $r_3$ and $s_3$.
	 \end{itemize}

\medskip
Next, we will give  infinity analogies for Proposition~\ref{prop-from YBEP to RBS} and Theorem~\ref{YBP and RBS}.

\begin{prop}\label{prop-from HYBEP to HRBS}
	 Let $(\{r_n\}_{n\geq1},\{s_n\}_{n\geq1})$ be an associative infinity-Yang-Baxter pair  in   unitary graded algebra $(A,m_2=\cdot)$. Let  $r_1=s_1=d\in A$. We  define a differential $m_1$ on $A$ and   operators  $ \mathscr{F}(   r_{n+1} )  ,  \mathscr{F} (s_{n+1} )   : A^{\otimes n}\rightarrow A$, for $n\geq1$, as: for $x,x_1,\ldots,x_n\in A$
	 \[m_1(x)=-[d,x]=-d\cdot x +(-1)^{|x|}x\cdot d,\]
	 \[  \mathscr{F}( r_{n+1})  (x_1,\ldots,x_n)=\sum(-1)^{ \sum\limits_{k=1}^{n}\sum\limits_{j=k+1}^{n+1}|x_k||r^{[j]}|}r^{[1]}\cdot x_1\cdot r^{[2]}\cdots r^{[n]}\cdot x_n\cdot r^{[n+1]},\] 	
	 \[   \mathscr{F}(s_{n+1} ) (x_1,\ldots,x_n)=\sum(-1)^{ \sum\limits_{k=1}^{n}\sum\limits_{j=k+1}^{n+1}|x_k||s^{[j]}|}s^{[1]}\cdot x_1\cdot s^{[2]}\cdots s^{[n]}\cdot x_n\cdot s^{[n+1]}.\]
	 Then, $(A,m_1,m_2)$ is a dg algebra, and  $(\{  \mathscr{F}(r_{n+1})  \}_{n\geq1}, \{  \mathscr{F}(s_{n+1}) \}_{n\geq 1}$) is homotopy   coupled Rota-Baxter operator on $(A,m_1,m_2)$.
\end{prop}

\begin{proof}
	We have already seen that  $d\cdot d=0$, the graded algebra $(A,m_1,m_2)$ is a dg algebra. For $n\geq 1$ and $ x_1,\ldots,x_n\in A$,
	\begin{eqnarray*}
		&&\left( 	m_1\circ   \mathscr{F}(r_{n+1} ) -\sum_{s+k+1=n} (-1)^{n-1}  \mathscr{F}(r_{n+1} ) \circ\left( \id^{\ot s}\otimes  m_1\otimes\id^{\ot k}\right) \right)(x_1,\ldots,x_n)\\
		= &&-\sum\limits_{[1],\ldots,[n+1]}(-1)^{ \sum\limits_{k=1}^{n}\sum\limits_{j=k+1}^{n+1}|x_k||r^{[j]}|}  d\cdot r^{[1]}\cdot x_1 \cdots   x_n\cdot r^{[n+1]}\\
		&&+ \sum\limits_{[1],\ldots,[n+1]}(-1)^{ \sum\limits_{k=1}^{n}\sum\limits_{j=k+1}^{n+1}|x_k|(|r^{[j]}|+1)+(n-1)} r^{[1]}\cdot x_1  \cdots   x_n\cdot r^{[n+1]}\cdot d \\
		&& +\sum\limits_{[1],\ldots,[n+1]}\sum\limits_{s=1}^{n}(-1)^{n-1+\sum\limits_{k=1}^{n}\sum\limits_{j=k+1}^{n+1}|x_k||r^{[j]}|+\sum\limits_{k=1}^s|r^{[j]}|+\sum\limits_{k=1}^{s-1}|x_j| }r^{[1]}\cdot x_1  \cdots  r^{[s]}\cdot d\cdot  x_s\cdot r^{[s+1]}\cdots x_n\cdot r^{[n+1]}\\
		&&-\sum\limits_{[1],\ldots,[n+1]}\sum\limits_{s=1}^{n}(-1)^{n-1+\sum\limits_{k=1}^{n}\sum\limits_{j=k+1}^{n+1}|x_k||r^{[j]}|+\sum\limits_{k=1}^s|r^{[j]}||x_j| }r^{[1]}\cdot x_1  \cdots  r^{[s]}\cdot  x_s\cdot d\cdot r^{[s+1]}\cdots x_n\cdot r^{[n+1]}\\
		= &&-\sum\limits_{[1],\ldots,[n+1]}(-1)^{ \sum\limits_{k=1}^{n}\sum\limits_{j=k+1}^{n+1}|x_k||r^{[j]}|}  s_1\cdot r^{[1]}\cdot x_1 \cdots   x_n\cdot r^{[n+1]}\\
		&&+ \sum\limits_{[1],\ldots,[n+1]}(-1)^{ \sum\limits\limits_{k=1}^{n}\sum\limits_{j=k+1}^{n+1}|x_k|(|r^{[j]}|+1)+  (n-1)} r^{[1]}\cdot x_1  \cdots   x_n\cdot r^{[n+1]}\cdot r_1 \\
		 && +\sum\limits_{[1],\ldots,[n+1]}\sum\limits_{s=1}^{n}(-1)^{n-1+\sum\limits_{k=1}^{n}\sum\limits_{j=k+1}^{n+1}|x_k||r^{[j]}|+\sum\limits_{k=1}^s|r^{[j]}|+\sum\limits_{k=1}^{s-1}|x_j| }r^{[1]}\cdot x_1  \cdots  r^{[s]}\cdot r_1\cdot  x_s\cdot r^{[s+1]}\cdots x_n\cdot r^{[n+1]}\\
		 &&-\sum\limits_{[1],\ldots,[n+1]}\sum\limits_{s=1}^{n}(-1)^{n-1+\sum\limits_{k=1}^{n}\sum\limits_{j=k+1}^{n+1}|x_k||r^{[j]}|+\sum\limits_{k=1}^s(  |r^{[j]}||x_j|)  }  r^{[1]} \cdot x_1  \cdots  r^{[s]}\cdot  x_s\cdot s_1\cdot r^{[s+1]}\cdots x_n\cdot r^{[n+1]}\\
		 =&& -\sum\limits_{k=1}^{n+1}\mathscr{F}\left(     s_1^{k} \cdot r_{n+1}^{1,\ldots,n+1}     - (-1)^{n-1}  r_{n+1}^{1,\ldots,n+1}\cdot r_1^{k}    \right) (x_1,\ldots,x_n).
	\end{eqnarray*}
	Thus we have that the following identity holds:
		\begin{eqnarray}&&\label{Equation: Equivalence-YB-RBS-1}\\
				\nonumber &&m_1\circ   \mathscr{F}(r_{n+1} ) -\sum_{s+k+1=n} (-1)^{n-1} \mathscr{F}( r_{n+1} ) \circ\left( \id^{\ot s}\otimes  m_1\otimes\id^{\ot k}\right)=-\sum\limits_{k=1}^{n+1}\mathscr{F}\left(    s_1^{k} \cdot r_{n+1}^{1,\ldots,n+1}     - (-1)^{n-1}  r_{n+1}^{1,\ldots,n+1}\cdot r_1^{k}   \right).
			\end{eqnarray}
Similarly, one can also obtain the following identities:
\begin{eqnarray}\label{Equation: Equivalence-YB-RBS-2}\sum_{i+j=n,\atop i,j\geq 1}(-1)^{1+i} m_2\circ\Big(\mathscr{F} (r_{i+1}) \ot   \mathscr{F}(r_{j+1}) \Big)&=&\sum_{i+j=n,\atop i,j\geq 1}(-1)^{1+i} \mathscr{F}(r_{i+1}^{1,\ldots,i+1}\cdot r_{j+1}^{i+1,\ldots,n}) ,
	\end{eqnarray}
\begin{eqnarray}\label{Equation: Equivalence-YB-RBS-3}
	&& \sum_{s+k+j+1=n} (-1)^{s+(j-1)(k+1)}   \mathscr{F}(r_{i+1}) \left( \id^{\ot s} \ot m_2\circ(  \mathscr{F}(r_{i+1}) \otimes \id)\ot\id^{\ot k}\right)\\
\nonumber	&=& \sum_{i+j=n,\atop i,j\geq 1}\sum_{s=1}^{i+1} (-1)^{s-1+(j-1)(n-s-j-1)}   \mathscr{F}(r_{i+1}^{1,\ldots,s,s+j+1,\ldots,n+1}\cdot r_{j+1}^{s,\ldots,s+j}) ,
\end{eqnarray}
\begin{eqnarray}\label{Equation: Equivalence-YB-RBS-4}
	 &&\sum_{s+k+j+1=n} (-1)^{s+(j-1)k}  \mathscr{F}(r_{i+1}) \left( \id^{\ot s} \ot m_2\circ( \id\otimes  \mathscr{F}(s_{j+1}) )\ot\id^{\ot k}\right)\\
	 \nonumber &=&\sum_{i+j=n,\atop i,j\geq 1}\sum_{s=1}^{i+1} (-1)^{s-1+(j-1)(s+1)}   \mathscr{F}(s_{j+1}^{s+1,\ldots,s+j+1} \cdot  r_{i+1}^{1,\ldots,s,s+j+1,\ldots,n+1}) .
\end{eqnarray}
By summing Equations~\eqref{Equation: Equivalence-YB-RBS-1}-\eqref{Equation: Equivalence-YB-RBS-4}, one can find that the right-hand side equals $0$ since $(\{r_n\}_{n\geq1},\{s_n\}_{n\geq1})$ satisfies Equation~\eqref{Eq: homotopy AYBE-version-1}. Thus we have that $\{  \mathscr{F}(r_{n+1})  \}_{n\geq1}$ and $\{ \mathscr{F} (s_{n+1})  \}_{n\geq 1}$ satisfy Equation~\eqref{Eq: homotopy RB-operator-version-3}.  Similarly, it can be checked that $\{ \mathscr{F}(r_{n+1})  \}_{n\geq1}$ and $\{ \mathscr{F} (s_{n+1})  \}_{n\geq 1}$ satisfy the Equation~\eqref{Eq: homotopy RB-operator-version-4}.  In conclusion,   $(\{ \mathscr{F} (r_{n+1} ) \}_{n\geq1}, \{  \mathscr{F}(s_{n+1})  \}_{n\geq 1})$ is a  homotopy coupled Rota-Baxter operator  on $(A,m_1,m_2)$.
\end{proof}

\begin{thm}\label{HYBP and HRBS}
	Let $V$ be a finite dimensional graded space. Let $(\End(V),m_2=\cdot)$ be the endomorphism graded algebra of $V$ with composition as the multiplication.
	The map $$\chi:(r_1=s_1=d_V,\{r_n\}_{n\geq2},\{s_n\}_{n\geq2})\rightarrow (m_1=-[d_V,-],\{  \mathscr{F}(r_{n})  \}_{n\geq2}, \{ \mathscr{F}( s_{n})  \}_{n\geq 2})$$ is a bijection between  the set of associative infinity-Yang-Baxter pairs in graded algebra $(\End(V),m_2)$  and the set of triples $(m_1,\{R_n\}_{n\geqslant 1}, \{S_n\}_{n\geqslant 1})$ in which $m_1$ is a differential on the graded algebra $(\End(V),m_2)$, and   $(\{R_n\}_{n\geqslant 1},\{S_n\}_{n\geqslant 1})$ is a homotopy   coupled Rota-Baxter operator on the dg  $(\End(V),m_1,m_2)$.
	
\end{thm}
\begin{proof}
	For $n=0$, the equations of associative Yang-Baxter pair \eqref{Eq: homotopy RB-operator-version-3} \eqref{Eq: homotopy RB-operator-version-4}) are equal to $d_V^2=0$. Thus, $(V,d_V)$ is a complex and its endomorphism graded algebra $(\End(V),m_2)$ is a finite dimensional dg algebra with a differential $m_1$ defined by
	$$ m_1(x)=-d_V\cdot x+(-1)^{|x|}x\cdot d_V,\text{ for }x\in \End(V).$$
	Thus, $(\End(V),m_1=-[d_V,-],m_2)$ is a dg algebra if and only if the defining equations of associative Yang-Baxter pairs \eqref{Eq: homotopy RB-operator-version-3}\eqref{Eq: homotopy RB-operator-version-4}  hold at $n=0$.

%	Let $\{v_i\}_{i\in I}$ be a homogeneous basis of $V$, the set of elements $\{e_{p} ^{~q}\}_{p,q}$ in  $\End(V)$ denoted by
%	\[ e_{p} ^{~q}v_q=v_p, \text{ for all }p,q\in I. \]
%	Then $\{e_{p} ^{~q}\}_{p,q\in I}$ is a homogeneous basis of $\End(V)$,
%	and  the differential $m_1(e_{p} ^{~q})=-\sum_{j}(d_V)_p^je_{j} ^{~q}+(-1)^{|e_{p} ^{~q}|}\sum_j(d_V)^q_je_{p} ^{~j} ,$
%	where $d_V(v_p)=\sum_j(d_V)^j_pv_j$.
%	With the notations,
	
The operators  $R_n,S_n:(\End(V))^{\otimes n}\rightarrow \End(V), n\geq1,$ can be uniquely written as:

	$$ R_n(e_{p_1} ^{~q_1}\ot\cdots\ot e_{p_n} ^{~q_n})=\sum_{1\leq q_0,p_{n+1}\leq n}(-1)^{ \sum_{k=1}^{n}\sum_{j=k}^{n}|e_{p_k} ^{~q_k}||e_{q_j} ^{~p_{j+1}}|} r_{p_1,\ldots,p_n,p_{n+1}}^{q_0,q_1,\ldots,q_n}e_{q_0} ^{~p_{n+1}},$$   and   $$S_n(e_{p_1} ^{~q_1}\ot\cdots\ot e_{p_n} ^{~q_n})=\sum_{1\leq q_0,p_{n+1}\leq n}(-1)^{ \sum_{k=1}^{n}\sum_{j=k}^{n}|e_{p_k} ^{~q_k}||e_{q_j} ^{~p_{j+1}}|} s_{p_1,\ldots,p_n,p_{n+1}}^{q_0,q_1,\ldots,q_n}e_{q_0} ^{~p_{n+1}}$$
		with coefficients $r_{p_1,\ldots,p_n,l}^{q_0,q_1,\ldots,q_n},s_{p_1,\ldots,p_n,l}^{q_0,q_1,\ldots,q_n}\in \textbf{k}$. Denote $r_{n+1}=\sum r_{p_1,\ldots,p_n,l}^{q_0,q_1,\ldots,q_n}e_{q_0} ^{~p_1}\ot\cdots\ot e_{q_n} ^{~p_{n+1}}$ and $s_{n+1}=\sum s_{p_1,\ldots,p_n,l}^{q_0,q_1,\ldots,q_n}e_{q_0} ^{~p_1}\ot\cdots\ot e_{q_n} ^{~p_{n+1}}$, for all $n\geq1$. We have that $R_n=\mathscr{F} (r_{n+1}) $ and $S_n= \mathscr{F}(s_{n+1}) $, for all $n\geq1$.

By Lemma~\ref{[] iso} and Proposition~\ref{prop-from HYBEP to HRBS},   $(\{r_n\}_{n\geq1},\{s_n\}_{n\geq1})$ is an associative infinity-Yang-Baxter pair if and only if  operators $\{ \mathscr{F}(r_{n+1}) \}_{n\geq1}$ and $\{ \mathscr{F}(s_{n+1}) \}_{n\geq 1}$  define a  homotopy   Rota-Baxter system  on $(\End(V),m_1,m_2)$ with $m_1=-[d_V,-]$.

%Moreover,   $R_n=[r_{n+1}]$ and $S_n=[s_{n+1}]$ for all $n\geq1$, it says that $$\chi:(r_1=s_1=d_V,\{r_n\}_{n\geq2},\{s_n\}_{n\geq2})\rightarrow (m_1=-[d_V,-],\{ [r_{n+1}] \}_{n\geq1}, \{ [s_{n+1}] \}_{n\geq 1})$$ is a between  the set of associative infinity-Yang-Baxter pairs in graded algebra $(\End(V),m_2)$  and  the set of  differentials and  homotopy   Rota-Baxter system    operators on   graded algebra $(\End(V),m_2)$.
\end{proof}

\begin{remark}
In \cite{Sch09}, Schedler introduces the concept of associative infinity Yang-Baxter equations, which can be viewed as the infinity version of the skew-symmetric associative Yang-Baxter equations. In Definition~\ref{Def:associative Yang-Baxter infinty pairs}, we establish the notion of infinity versions of associative Yang-Baxter pairs, known as associative infinity-Yang-Baxter pairs. Moreover, if $\{r_n\}_{n\geq1}$ is a skew-symmetric solution of the associative infinity Yang-Baxter equations as defined by Schedler, then $(\{r_n\}_{n\geq1},\{r_n\}_{n\geq1})$ will give an associative infinity-Yang-Baxter pair in our sense.
\end{remark}
	
\bigskip 

\section{From   the minimal model to the   $L_\infty$-algebra governing deformations}\label{Section: Linfinty algebras}

%\subsection{The $L_\infty$-algebra on the deformation complex of Rota-Baxter   algebras}

%In the last section, we construct the Koszul dual homotopy cooperad of the operad $\RB$ and we prove that its cobar construction is the minimal model of $\RB$.
In this section, we will use the minimal model $\RBSinfty$, or more precisely the Koszul dual  homotopy cooperad ${\RBS^{\ac}}$,  to determine the deformation complex as well as the $L_\infty$-algebra structure  on it for Rota-Baxter systems.

For the terminology and basic facts about $L_\infty$-algebras, we refer the reader to \cite[Section 2]{CGWZ24}.

\begin{defn}Let $V$ be a graded space. Introduce an   $L_\infty$-algebra $\frakC_{\RBA}(V)$ associated to $V$ as $\frakC_{\RBA}(V):=\mathbf{Hom}({\RBS^\ac}, \End_V)^{\prod}$.
\end{defn}

Now, let's determine the $L_\infty$-algebra $\frakC_{\RBA}(V)$ explicitly. The sign rules in the homotopy cooperad ${\RBS^\ac}$ are complicated, so we need some transformations.
Explicitly, any $f\in \End_V(n)$ corresponds to an element $\tilde{f}\in\mathbf{Hom}(\cals, \End_{sV})(n)$ which is defined as $\big(\tilde{f}(\delta_n)\big)(sv_1\ot\cdots\ot sv_n)=(-1)^{\sum_{k=1}^{n-1}\sum_{j=1}^k|v_j|}(-1)^{(n-1)|f|}sf(v_1\ot \cdots \ot v_n)$ for any $v_1,\dots, v_n\in V$.

Thus we have the following  isomorphisms of homotopy operads:
\begin{eqnarray*}\mathbf{Hom}\big({\RBS^\ac}, \End_V\big)&\cong& \mathbf{Hom}\big({\RBS^\ac}, \mathbf{Hom}(\cals, \End_{sV})\big)\\
&\cong&\mathbf{Hom}\big({\RBS^\ac}\ot_{\mathrm{H}}\cals,\End_{sV}\big)\\
&=&\mathbf{Hom}\big({\mathscr{S}({\RBS^\ac})}, \End_{sV}\big).
\end{eqnarray*}
We obtain $$\frakC_{\RBA}(V)\cong \mathbf{Hom}\big({\mathscr{S}({\RBS^\ac})}, \End_{sV}\big)^{\prod}.$$ Recall that ${\mathscr{S}({\RBS^\ac})}(n)=\bfk u_n\oplus \bfk v_n\oplus \bfk w_n$ with $|u_n|=0$ and $|v_n|=|w_n|=1$. By definition $$\mathbf{Hom}\big({\mathscr{S}({\RBS^\ac})}, \mathrm{End}_{sV}\big)(n)=\Hom\big(\bfk u_n\oplus \bfk v_n\oplus \bfk w_n, \Hom((sV)^{\ot n},sV)\big).$$ Each $f\in \Hom((sV)^{\ot n},sV)$ determines bijectively a map $\tilde{f}$ in $\Hom(\bfk u_n, \Hom((sV)^{\ot n},sV))$ by imposing  $\tilde{f}(u_n)=f$,  each $g\in \Hom((sV)^{\ot n},V)$ is in bijection with a map $\hat{g}$ in $\Hom(\bfk v_n, \Hom((sV)^{\ot n},sV))$ as $\hat{g}(v_n)=(-1)^{|g|}sg$  and  each $h\in \Hom((sV)^{\ot n},V)$ is in bijection with a map $\hat{h}$ in $\Hom(\bfk w_n, \Hom((sV)^{\ot n},sV))$ as $\hat{h}(v_n)=(-1)^{|h|}sh$.

Denote  $$\frakC_{\Alg}(V)=\prod\limits_{n\geqslant 1}\Hom((sV)^{\ot n},sV),$$
$$ \frakC_{\RBO}(V)= \frakC_{\RBO_R}(V)\oplus\frakC_{\RBO_S}(V) =\prod\limits_{n\geqslant 1}\Hom((sV)^{\ot n},V)\oplus\prod\limits_{n\geqslant 1}\Hom((sV)^{\ot n},V) .$$

In this way, we identify $\frakC_{\RBA}(V)$ with $\frakC_{\Alg}(V)\oplus \frakC_{\RBO}(V)$.
By the general theory of homotopy (co)operads (see \cite[Section 3.1]{CGWZ24}), a direct computation gives the $L_\infty$-algebra structure on $\frakC_{\RBA}(V)$:
\begin{itemize}

\item[(I)] For homogeneous elements $sf, sh\in \mathfrak{C}_{\Alg}(V)$, define $$l_2(sf\ot sh):= [sf, sh]_{G}\in\mathfrak{C}_{\Alg}(V).$$

\item[(II)]	 Let $n\geqslant 1$.  For homogeneous elements $sf\in \Hom((sV)^{\ot n},sV)\subset \mathfrak{C}_{\Alg}(V)$ and $g_1,\dots, g_n\in \frakC_{\RBO_R}(V)$ and $h_1,\dots,h_n\in \frakC_{\RBO_S}(V)$,  $1\leqslant i\leqslant n$, $1< j< n$	define $$l_{n+1}(sf\ot g_1\ot \cdots \ot g_n)\in \frakC_{\RBO_R}(V),\  l_{n+1}(sf\ot h_1\ot \cdots \ot h_n)\in \frakC_{\RBO_S}(V)$$
and
\[l_{n+1}(sf\otimes g_1\otimes \cdots\otimes g_j\otimes h_{j+1}\otimes\cdots\otimes h_n)\in \frakC_{\RBO}(V)\]
 as :
	\begin{eqnarray}&&l_{n+1}(sf\ot g_1\ot \cdots \ot g_n)\\
		&=& \sum_{\sigma\in S_n}(-1)^{\eta_1} (f\circ\{sg_{\sigma(1)}, \ldots , sg_{\sigma(n)}\}-(-1)^{(|g_{\sigma(1)}+1)(|f|+1)}s^{-1}(sg_{\sigma(1)})\{sf\{sg_{\sigma(2)},\ldots,sg_{\sigma(n)},\id\}\})\nonumber\\
		\nonumber &&\\
	&&l_{n+1}(sf\ot h_1\ot \cdots \ot h_n)\\
	&=& \sum_{\sigma\in S_n}(-1)^{\eta_2} (f\circ\{sh_{\sigma(1)}, \ldots , sh_{\sigma(n)}\}-(-1)^{(|h_{\sigma(1)}+1)(|f|+1)}s^{-1}(sh_{\sigma(1)})\{sf\{\id,sh_{\sigma(2)},\ldots,sh_{\sigma(n)}\}\})\nonumber	\\
	\nonumber &&\\
	&& l_{n+1}(sf\otimes g_1\otimes \cdots\otimes g_j\otimes h_{j+1}\otimes\cdots\otimes h_n)\label{L infinity O}\\ &=&\underbrace{\sum_{\sigma'\in S_j, \sigma''\in S_{n-j}} (-1)^{\eta_3}s^{-1}(sg_{\sigma'(1)})\big\{sf\big\{sg_{\sigma'(2)},\dots,sg_{\sigma'(j)},\id,sh_{j+\sigma''(1)},\dots,sh_{j+\sigma''(n-j)}\big\}\big\}}_{\mathrm{(A)}}\nonumber \\
		&&\nonumber\ \   +\underbrace{\sum_{\sigma'\in S_j, \sigma''\in S_{n-j}} (-1)^{\eta_4}s^{-1}(sh_{j+\sigma''(1)})\big\{sf\big\{sg_{\sigma'(1)},\dots,sg_{\sigma'(j)},\id,sh_{j+\sigma''(2)},\dots,sh_{j+\sigma''(n-j)}\big\}\big\}}_{ \mathrm{(B)}},	
\end{eqnarray}
where
$$(-1)^{\eta_1}=\chi(\sigma; g_1,\dots,g_n)(-1)^{n(|f|+1)+\sum\limits_{k=1}^{n-1}\sum\limits_{j=1}^k|g_{\sigma(j)}|},$$
$$(-1)^{\eta_1}=\chi(\sigma; h_1,\dots,h_n)(-1)^{n(|f|+1)+\sum\limits_{k=1}^{n-1}\sum\limits_{j=1}^k|h_{\sigma(j)}|},$$
\[\begin{aligned}
	 (-1)^{\eta_3}=&\chi(\sigma'; g_1,\dots,g_j)\chi(\sigma''; h_{j+1},\dots,h_n)(-1)^{1+n(|f|+1)+\sum\limits_{k=1}^{n-1}\sum\limits_{s=1}^k|h_{j+\sigma''(s)}|}&\\
	 &(-1)^{(\sum_{k=1}^j|g_{\sigma'(k)}|)(n-j)+\sum\limits_{k=1}^{j-1}\sum\limits_{s=1}^k|g_{\sigma'(s)}|+(|g_{\sigma'(1)}|+1)(|f|+1)},&
\end{aligned}\]

and
\begin{align*}
	 (-1)^{\eta_4}=&\chi(\sigma'; g_1,\dots,g_j)\chi(\sigma''; h_{j+1},\dots,h_n)(-1)^{1+n(|f|+1)+\sum\limits_{k=1}^{j-1}\sum\limits_{s=1}^k|g_{\sigma'(s)}|}\\
	 &(-1)^{(|h_{j+\sigma''(1)}|+1)(|f|+\sum_{k=1}^j|g_{\sigma'(k)}|+j+1) +\sum\limits_{k=1}^{n-1}\sum\limits_{s=1}^k|h_{j+\sigma''(s)}|+(\sum_{k=1}^j|g_{\sigma'(k)}|)(n-j)}.
\end{align*}
In Equation~\eqref{L infinity O}, Part $\rm (A)$  belongs to  $\frakC_{\RBO_R}(V)$ and Part $\rm (B)$ belongs to  $\frakC_{\RBO_S}(V)$.

\smallskip

\item[(III)]  Let $m\geqslant 1$.  For homogeneous elements $sh\in \Hom((sV)^{\ot n},sV)\subset \mathfrak{C}_{\Alg}(V), g_1,\dots,g_m\in \Hom(T^c(sV),V)\subset \mathfrak{C}_{\RBO}(V)$ with $1\leqslant m\leqslant n$, for $1\leqslant k\leqslant m$,  define $$l_{m+1}(g_1\ot \cdots\ot g_k\ot sh \ot g_{k+1}\ot \cdots\ot g_m)\in \mathfrak{C}_{\RBA}(V)$$ to be
$$l_{m+1}(g_1\ot \cdots\ot g_k\ot sh \ot g_{k+1}\ot \cdots\ot g_m)=(-1)^{(|h|+1)(\sum\limits_{j=1}^k|g_j|)+k}l_{m+1}(sh\ot g_1\ot \cdots \ot g_m),$$
where the RHS has been introduced in (II).

\item[(IV)] All other  components of operators $\{l_n\}_{n\geqslant 1}$ vanish.
\end{itemize}

\cite[Proposition 3.9]{CGWZ24} gives immediately an alternative  definition of homotopy Rota-Baxter systems.
\begin{prop}
	A homotopy Rota-Baxter  system  structure  on a graded space $V$ is equivalent to a Maurer-Cartan element in the $L_\infty$-algebra $\frakC_{\RBA}(V)$. In particular, when $V$ is concentrated in degree $0$, a Maurer-Cartan element in $\frakC_{\RBA}(V)$ gives a Rota-Baxter   system  structure  on $V$.
\end{prop}

By interpreting Rota-Baxter systems as Maurer-Cartan elements as described above, we can study their deformation theory using the twisting procedure in this \( L_\infty \)-algebra.

\begin{prop}
Let $(A,\mu,R,S)$ be  a Rota-Baxter system. Twist the $L_\infty$-algebra $\frakC_{\RBA}(V)$ by the Maurer-Cartan element corresponding to the Rota-Baxter system structure $(A,\mu,R,S)$, then its underlying complex controls the deformations of Rota-Baxter system.

\end{prop}
}

\begin{remark} In \cite{LQWY22p}, Y. Liu, K. Wang, and L. Yin construct a cochain complex for Rota-Baxter systems, and they prove that the cohomology of this complex governs the deformations of Rota-Baxter systems. Notably, the underlying complex of the twisting \(L_\infty\) algebra introduced above is exactly the suspension of the cochain complex introduced by Liu, Wang, and Yin. In this paper, we present an operadic approach to the cochain complex described in \cite{LQWY22p}, showing that it is endowed with an \(L_\infty\)-algebra structure.
	\end{remark}

\bigskip

	\appendix

\section{The Koszul dual homotopy cooperad: Proof of Proposition~\ref{Prop-Koszul-dual-coooperad}}\label{Appendix: Koszul dual homotopy cooperad}

In this appendix, we present a proof of Proposition~\ref{Prop-Koszul-dual-coooperad}.

 One needs to show that the induced derivation $\partial$ on the cobar construction of $\mathscr{S}(\RBS^\ac)$, i.e., the free operad generated by $s^{-1}\overline{\mathscr{S}(\RBS^\ac)}$ is a differential, that is, $\partial^2=0$.
	
	Denote $s^{-1}u_n, n\geqslant 2$ (resp. $s^{-1}v_n$, $ s^{-1}w_n$ $n\geqslant 1$) by $x_n$  (resp. $y_n$, $z_n$) which are  the basis elements of $s^{-1}\overline{\mathscr{S}(\RBS^\ac)}$ and generators of $\Omega\big(\mathscr{S}(\RBS^\ac)\big)$. Notice that $|x_n|=-1$ and $|y_n|=|z_n|=0$. By the definition of cobar construction of coaugmented homotopy cooperads,   the action of differential $\partial$ on generators $x_n, y_n, z_n$ is given by the following formulae:
	\begin{eqnarray}\label{Eq: partial xn}
		\partial(x_n)&=&-\sum\limits_{j=2}^{n-1}x_{n-j+1}\{x_j\}, \ n\geqslant 2;\end{eqnarray}
	\begin{eqnarray}\label{Eq: partial yn} \ \ \partial(y_n)&=&-\sum_{k=2}^{n}\sum_{r_1+\dots+r_k=n,\atop r_1,\dots,r_k\geqslant 1}x_k\{y_{r_1},\dots,y_{r_k}\}\\
	\nonumber	&&+\sum_{2\leqslant p\leqslant n,\atop 1\leqslant j\leqslant p}\sum_{r_1+\dots+r_p=n,\atop r_1,\dots,r_p\geqslant1}y_{r_1}\big\{x_p\{y_{r_2},\dots,y_{r_j},\id,z_{r_{j+1}},\dots,z_{r_{p}}\}\big\}, \  n\geqslant 1;
	\end{eqnarray}	

\begin{eqnarray}\label{Eq: partial yn} \ \ \partial(z_n)&=&-\sum_{k=2}^{n}\sum_{r_1+\dots+r_k=n,\atop r_1,\dots,r_k\geqslant 1}x_k\{z_{r_1},\dots,z_{r_k}\}\\
	\nonumber	&&+\sum_{2\leqslant p\leqslant n,\atop 1\leqslant j\leqslant p}\sum_{r_1+\dots+r_p=n,\atop r_1,\dots,r_p\geqslant1}z_{r_1}\big\{x_p\{y_{r_2},\dots,y_{r_j},\id,z_{r_{j+1}},\dots,z_{r_{p}}\}\big\}, \  n\geqslant 1.
\end{eqnarray}
For the definition of the brace operation in operads, we refer to  \cite{Ger63, GV95, Get93, CGWZ24}.

	Note that \begin{eqnarray}\label{Eq: partial y1} \partial(y_1)=\partial(w_1)=0.\end{eqnarray}
	
	We just need to show that $\partial^2=0$ holds on generators $x_n, n\geqslant 2$ and $y_n, z_n, n\geqslant 1$, which can be checked by direct computations.
	In fact, we have
	\begin{eqnarray*}
		\partial^2(x_n)&=&\partial\Big(-\sum\limits_{j=2}^{n-1}x_{n-j+1}\{x_j\}\Big)\\
		&\stackrel{\eqref{Eq: partial xn}}{=}&-\sum\limits_{j=2}^{n-1}\partial(x_{n-j+1})\{x_j\}+\sum_{j=2}^{n-1}x_{n-j+1}\{\partial(x_j)\}\\
		&\stackrel{\eqref{Eq: partial xn}}{=}&\sum_{i+j+k-2=n \atop
			2\leqslant i,j,k\leqslant n-2}(x_i\{x_j\})\{x_k\}-\sum_{i+j+k-2=n \atop
			2\leqslant i,j,k\leqslant n-2}x_i\big\{x_j\{x_k\}\big\}\\
		&{=}&\sum_{i+j+k-2=n \atop
			2\leqslant i,j,k\leqslant n-2}x_i\big\{x_j\{x_k\}\big\}+\sum_{i+j+k-2=n \atop
			2\leqslant i,j,k\leqslant n-2}x_i\{x_j,x_k\}-\sum_{i+j+k-2=n \atop
			2\leqslant i,j,k\leqslant n-2}x_i\{x_k,x_j\}-\sum_{i+j+k-2=n \atop
			2\leqslant i,j,k\leqslant n-2}x_i\big\{x_j\{x_k\}\big\}\\
		&=&0.\end{eqnarray*}
	We also  have
	\begin{eqnarray*}
		& &\\
		\partial^2(y_n)&\stackrel{\eqref{Eq: partial yn}}{=}&\partial\Big(-\sum_{k=2}^{n}\sum_{r_1+\dots+r_k=n,\atop r_1,\dots,r_k\geqslant 1}x_k\{y_{r_1},\dots,y_{r_k}\}\\
		\nonumber	&&+\sum_{2\leqslant p\leqslant n,\atop 1\leqslant j\leqslant p}\sum_{r_1+\dots+r_p=n,\atop r_1,\dots,r_p\geqslant1}y_{r_1}\big\{x_p\{y_{r_2},\dots,y_{r_j},\id,z_{r_{j+1}},\dots,z_{r_{p}}\}\big\}\Big).\end{eqnarray*}
	As $\partial$ is a derivation and by \eqref{Eq: partial y1}, \begin{eqnarray*}
		\partial^2(y_n)&=&-\sum_{k=2}^{n}\sum_{r_1+\dots+r_k=n,\atop r_1,\dots,r_k\geqslant 1}\partial(x_k)\{y_{r_1},\dots,y_{r_k}\}
		+\sum_{k=2}^{n-1}\sum_{r_1+\dots+r_k=n,\atop r_1,\dots,r_k\geqslant 1}\sum_{j=1}^k x_k\{y_{r_1},\dots,\partial(y_{r_j}),\dots,y_{r_k}\}\\
		&&+\sum_{2\leqslant p\leqslant n,\atop 1\leqslant j\leqslant p}\sum_{r_1+\dots+r_p=n,\atop r_1,\dots,r_p\geqslant1}\partial (y_{r_1})\big\{x_p\{y_{r_2},\dots,y_{r_j},\id,z_{r_{j+1}},\dots,z_{r_{p}}\}\big\}\\
		&&+\sum_{2\leqslant p\leqslant n,\atop 1\leqslant j\leqslant p}\sum_{r_1+\dots+r_p=n,\atop r_1,\dots,r_p\geqslant1}y_{r_1}\big\{\partial(x_p)\{y_{r_2},\dots,y_{r_j},\id,z_{r_{j+1}},\dots,z_{r_{p}}\}\big\}\\
		&&-\sum_{2\leqslant p\leqslant n,\atop 1\leqslant j\leqslant p}\sum_{r_1+\dots+r_p=n,\atop r_1,\dots,r_p\geqslant1}\sum_{i=2}^jy_{r_1}\big\{x_p\{y_{r_2},\dots,\partial(y_{r_i}),\dots,y_{r_j},\id,z_{r_{j+1}},\dots,z_{r_{p}}\}\big\}\\
		&&-\sum_{2\leqslant p\leqslant n,\atop 1\leqslant j\leqslant p}\sum_{r_1+\dots+r_p=n,\atop r_1,\dots,r_p\geqslant1}\sum_{i=j+1}^py_{r_1}\big\{x_p\{y_{r_2},\dots,y_{r_j},\id,z_{r_{j+1}},\dots,\partial(z_{r_i}),\dots,z_{r_{p}}\}\big\}\\
		&\stackrel{\eqref{Eq: partial xn}\eqref{Eq: partial yn}}{=}&\underbrace{\sum_{k=2}^{n}\sum_{r_1+\dots+r_k=n \atop r_1,\dots,r_k\geqslant 1}\sum_{i=2}^{k-1}\big(x_{k-i+1}\{x_i\}\big)\{y_{r_1},\dots,y_{r_k}\}}_{\mathrm{(I)}}\\
		&&- \underbrace{\sum_{k=2}^{n-1}\sum_{r_1+\dots+r_k=n \atop r_1,\dots,r_k\geqslant 1} \sum_{j=1}^k \sum_{u=2}^{r_j} \sum_{l_1+\dots+l_u=r_j,\atop l_1,\dots,l_u\geqslant 1}x_k\{y_{r_1},\dots,   x_u\{y_{l_1}, \dots, y_{l_u}\},\dots,y_{r_k}\}}_{\mathrm{(II)}}\\
		&&+\underbrace{\sum_{k=2}^{n-1}\sum_{r_1+\dots+r_k=n \atop r_1,\dots,r_k\geqslant 1} \sum_{j=1}^k\sum_{2\leqslant p\leqslant r_j,\atop 1\leqslant j\leqslant p}\sum_{l_1+\dots+l_p=r_j,\atop l_1,\dots,l_p\geqslant 1}     x_k\{y_{r_1},\dots,   y_{l_1}\big\{x_p\{y_{l_2},\dots,y_{l_j},\id,z_{l_{j+1}},\dots,z_{l_{p}}\}\big\},\dots,y_{r_k}\}}_{\mathrm{(III)}}\\
		&&-\underbrace{\sum_{2\leqslant p\leqslant n-1,\atop 1\leqslant j\leqslant p}\sum_{r_1+\dots+r_p=n,\atop r_1,\dots,r_q\geqslant1} \sum_{k=2}^{r_1}\sum_{l_1+\dots+l_k=r_1 \atop  l_1, \dots, l_k\geqslant 1 }\Big(x_k\big\{y_{l_1},\dots,y_{l_k}\big\}\Big)\big\{x_p\{y_{r_2},\dots,y_{r_j},\id,z_{r_{j+1}},\dots,z_{r_{p}}\}\big\}}_{\mathrm{(IV)}}\\
		&&+\underbrace{\sum_{2\leqslant p\leqslant n-1,\atop 1\leqslant j\leqslant p}\sum_{r_1+\dots+r_p=n,\atop r_1,\dots,r_p\geqslant1} \sum_{2\leqslant k\leqslant r_1,\atop 1\leqslant i\leqslant k}\sum_{l_1+\dots+l_k=r_1,\atop l_1,\dots,l_k\geqslant 1}\Big(y_{l_1}\big\{x_{k}\{y_{l_2},\dots,y_{l_i},\id,\dots z_{l_{j}}\}\big\}\Big)\Big\{x_p\{y_{r_2},\dots,y_{r_j},\id,\dots,z_{r_{q}}\}\Big\}}_{\mathrm{(V)}}\\
		&&-\underbrace{\sum_{2\leqslant p\leqslant n,\atop 1\leqslant j\leqslant p}\sum_{r_1+\dots+r_p=n,\atop r_1,\dots,r_p\geqslant1}\sum\limits_{j=2}^{p-1}y_{r_1}\big\{(x_{p-j+1}\{x_j\})\{y_{r_2},\dots,y_{r_j},\id,z_{r_{j+1}},\dots,z_{r_{p}}\}\big\}}_{\mathrm{(VI)}}\\
		&&+\underbrace{\sum_{2\leqslant p\leqslant n-1\atop 1\leqslant i\leqslant p} \sum_{r_1+\dots+r_p=n,\atop r_1,\dots,r_p\geqslant1}\sum_{j=2}^i\sum_{k=2}^{r_j} \sum_{l_1+\dots+l_k=r_j\atop l_1,\dots,l_k\geqslant1}y_{r_1}\big\{x_p\{y_{r_2},\dots,x_{k}\{y_{l_1},\dots,y_{l_{k}}\},\dots,y_{r_{i}},\id,\cdots z_{r_{p}}\}\big\}}_{\mathrm{(VII)}}\\
		&&-\underbrace{\sum_{ 2\leqslant p\leqslant n-1, \atop 1\leqslant i\leqslant p, 1\leqslant t\leqslant s.} \sum_{r_1+\dots+r_p=n \atop r_1,\dots,r_p\geqslant 1}\sum_{j=2}^i  \sum_{l_1+\dots+l_k=r_j\atop l_1,\dots,l_s\geqslant1}  y_{r_1}\big\{x_p\{y_{r_2},\dots,y_{l_1}\{x_k\{y_{l_2},\dots,y_{l_t},\id,\dots,z_{l_s}\}\},\dots,y_{r_i},\id,\dots, z_{r_{p}}\}\big\}}_{\mathrm{(VIII)}}
		\\	
		&&+\underbrace{\sum_{2\leqslant p\leqslant n-1\atop 1\leqslant i\leqslant p} \sum_{r_1+\dots+r_p=n,\atop r_1,\dots,r_p\geqslant1}\sum_{j=i+1}^p\sum_{k=2}^{r_j} \sum_{l_1+\dots+l_k=r_j\atop l_1,\dots,l_k\geqslant1}y_{r_1}\big\{x_p\{y_{r_2},\dots,y_{r_{i}},\id,\dots,x_{k}\{z_{l_1},\dots,z_{l_{k}}\},\cdots z_{r_{p}}\}\big\}}_{\mathrm{(IX)}}\\
		&&-{\small{\underbrace{\sum_{ 2\leqslant p\leqslant n-1, \atop 1\leqslant i\leqslant p, 1\leqslant t\leqslant s.} \sum_{r_1+\dots+r_p=n \atop r_1,\dots,r_p\geqslant 1}\sum_{j=i+1}^p  \sum_{l_1+\dots+l_k=r_j\atop l_1,\dots,l_s\geqslant1}  y_{r_1}\big\{x_p\{y_{r_2},\dots,y_{r_i},\id,z_{r_{i+1}},\cdots, z_{l_1}\{x_k\{y_{l_2},\dots,y_{l_t},\id,z_{l_{t+1}}\dots,z_{l_s}\}\},\dots, z_{r_{p}}\}\big\}}_{\mathrm{(X)}}}}\\
	\end{eqnarray*}

	According to the pre-Jacobi identity (see \cite{Ger63, GV95, Get93}or \cite[Proposition 3.7]{CGWZ24}), we have
	$$
	\mathrm{(I)}=\mathrm{(II)},
	\mathrm{(III)}=\mathrm{(IV)}\  \mathrm{and}\
	\mathrm{(V)+(VII)+(IX)}=\mathrm{(VI)+(VIII)+(X)},
	$$
	so we obtain   $\partial^2(y_n)=0$.
	
	Similarly,  one can check that $\partial^2(z_n)=0$.

\bigskip

	\section{The minimal model for operad ${}_m\RBS$}\label{Appendix: the minimal model for monomimal operad}
	
In this section,	we will show that the homology of the differential graded operad ${}_m\RBS_\infty$ introduced in the proof of Theorem~\ref{Thm: Minimal model} concentrates in degree $0$. More explicitly, we will prove that ${}_m\RBS_\infty$  is the minimal model of a certain monomial operad.

\begin{defn}The operad $m\RBS$ is defined by the quotient $\mathcal{F}(E)/<G>$, where $E$ is the collection defined by $E(1)=\mathbf{k}R\oplus \mathbf{k}S, E(2)=\mathbf{k}\mu, E(n)=0$ for $n\geqslant 3$, and $G$ is the set of relations:
\[\mu\circ_{1} \mu\ \mathrm{,  } \ \ R\circ_1(\mu\circ_1 R) \ \ \mathrm{and}\ \  S\circ_1(\mu\circ_1 R).\]
	\end{defn}

Recall that ${}_m\RBS_\infty$ is the  quasi-free dg operad generated by graded collections $\{m_n\}_{n\geq 2}\cup\{R_n\}_{n\geq1}\cup\{S_n\}_{n\geq1}$ endowed with the differential $\overline{\partial}$.  The action  of $\bar \partial$ on generators is given by the following formulae:
\begin{eqnarray}\label{Eq: defining mHRB 1} &&\forall n\geqslant 2,\quad
	\bar{\partial}{m_n} = \sum_{j=2}^{n-1}(-1)^{1+j(n-1)}m_{n-j+1}\circ_1 m_j,
	\\&& \forall n\geqslant 1,   \quad   \bar{\partial} R_n=   \sum\limits_{\small\substack{ r_1+ r_2=n\\r_1 , r_2\geqslant 1}}(-1)^{r_1(r_2-1)}\Big(R_{r_1}\circ_1 m_2\circ_1 R_{r_2}\Big) ,\\
	& &\forall n\geqslant 1,   \quad   \bar{\partial} S_n=   \sum\limits_{\small\substack{ r_1+ r_2=n\\r_1 , r_2\geqslant 1}}(-1)^{r_1(r_2-1)}\Big(S_{r_1}\circ_1 m_2\circ_1 R_{r_2}\Big) .
\end{eqnarray}

Now, we  are going to prove that ${}_m\RBS_\infty$ is the minimal model of the operad ${}_m\RBS$. To achieve this, we will construct  a explicit homotopy. The construction of this homotopy map involves a  graded path-lexicographic ordering on the operad $\RBS_\infty$.
%	Each tree monomial gives rise to a path sequence; for details,  \cite[Chapter 3]{BD16}.

	To any tree monomial $\mathcal{T}$ with  $n$ leaves (written as $\mbox{arity}(\mathcal{T})=n$), we can associate   with a sequence   $(x_1, \dots, x_n)$  where  $x_i$ is the word formed by   generators of ${}_m\RBS_\infty$ corresponding to the vertices along the unique path from the root of $\mathcal{T}$  to its $i$-th leaf.
	
	For two graded tree monomials $\mathcal{T},\mathcal{T}'$, we compare $\mathcal{T},\mathcal{T}'$ in the following way:
	\begin{itemize}
		\item[(i)] If $\mbox{arity}(\mathcal{T})>\mbox{arity}(\mathcal{T}')$, then  $\mathcal{T}>\mathcal{T}'$;
		\item[(ii)] if $\mbox{arity}(\mathcal{T})=\mbox{arity}(\mathcal{T}')$, and $\deg(\mathcal{T})>\deg(\mathcal{T}')$, then $\mathcal{T}>\mathcal{T}'$, where $\deg(\mathcal{T})$ is the sum of the degrees of all generators of ${}_m\RBS_\infty$ appearing in    $\mathcal{T}$;
		\item[(iii)] if $\mbox{arity}(\mathcal{T})=\mbox{arity}(\mathcal{T}')(=n), \deg(\mathcal{T})=\deg(\mathcal{T}')$, then $\mathcal{T}>\mathcal{T}'$ if the path sequences   $(x_1,\dots,x_n), (x'_1,\dots,x_n')$ associated to $\mathcal{T}, \mathcal{T}'$ satisfies $(x_1,\dots,x_n)>(x_1',\dots,x_n')$ with respect to the length-lexicographic order of words induced by $$R_1<S_1<m_2<R_2<S_2<m_3<\cdots<R_n<S_n<m_{n+1}<R_{n+1}<S_{n+1}<\cdots.$$
		
	\end{itemize}
	It is ready to see that this is a well-order.
	Under this order, the leading terms in the expansion of $\bar{\partial}(m_n), \bar{\partial}(R_n)$ and $\bar{\partial}(S_n)$ are the following tree monomials respectively:
	
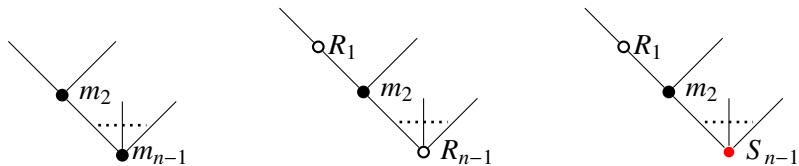
\begin{figure}[h]
		\begin{tikzpicture}[scale=0.8]
			\tikzstyle{every node}=[thick,minimum size=4pt, inner sep=1pt]
			\node(1) at (0,0) [draw, circle, fill=black, label=right:{\small $m_{n-1}$}]{};
			\node(2-1) at (-1,1) [draw, circle, fill=black, label=right: {\small $\ m_2$}]{};
			\node(3-1) at (-2,2){};
			\node(3-2) at (0,2){};
			\node(2-2) at (0,1){};
			\node(2-3) at (1,1){};
			\draw (1)--(2-1);
			\draw  (1)--(2-2);
			\draw  (1)--(2-3);
			\draw (2-1)--(3-1);
			\draw (2-1)--(3-2);
			\draw [dotted,line width=1pt](-0.4,0.5)--(0.4,0.5);
		\end{tikzpicture}
		\hspace{8mm}
		\begin{tikzpicture}[scale=0.8]
			\tikzstyle{every node}=[thick,minimum size=4pt, inner sep=1pt]
			\node(1) at (0,0) [draw, circle, label=right:{\small$\ R_{n-1}$}]{};
			\node(2-1) at (-1,1) [draw, circle, fill=black, label=right: {\small$\ m_2$}]{};
			\node(3-1) at (-1.75,1.75)[draw, circle, label=right:{\small $R_1$}]{};
			\node(3-2) at (0,2){};
			\node(2-2) at (0,1){};
			\node(2-3) at (1,1){};
			\node(4) at (-2.5,2.5){};
			\draw (1)--(2-1);
			\draw  (1)--(2-2);
			\draw  (1)--(2-3);
			\draw (2-1)--(3-1);
			\draw (2-1)--(3-2);
			\draw (3-1)--(4);
			\draw [dotted,line width=1pt](-0.4,0.5)--(0.4,0.5);
		\end{tikzpicture}
		\hspace{8mm}
		\begin{tikzpicture}[scale=0.8]
			\tikzstyle{every node}=[thick,minimum size=4pt, inner sep=1pt]
			\node(1) at (0,0) [fill=red, circle, label=right:{\small $\ S_{n-1}$}]{};
			\node(2-1) at (-1,1) [draw, circle, fill=black, label=right:{\small $\ m_2$}]{};
			\node(3-1) at (-1.75,1.75)[draw, circle, label=right: {\small$R_1$}]{};
			\node(3-2) at (0,2){};
			\node(2-2) at (0,1){};
			\node(2-3) at (1,1){};
			\node(4) at (-2.5,2.5){};
			\draw (1)--(2-1);
			\draw  (1)--(2-2);
			\draw  (1)--(2-3);
			\draw (2-1)--(3-1);
			\draw (2-1)--(3-2);
			\draw (3-1)--(4);
			\draw [dotted,line width=1pt](-0.4,0.5)--(0.4,0.5);
		\end{tikzpicture}
	\end{figure}

	Let $\mathcal{S}$ be a generator of degree $\geqslant 1$ in $ {}_m\RBS_\infty$. Denote the leading monomial of $\bar{\partial} \mathcal{S}$ by $\widehat{\mathcal{S}}$ and the  coefficient of $\widehat{\mathcal{S}}$ in $\bar{\partial}$ is written  as $l_\mathcal{S}$.  A tree monomial of the form $\widehat{\mathcal{S}}$ is called typical, so all typical tree monomials are  of the form $$ m_{n-1}\circ_1 m_2 \ \mathrm{,  } \ \  (R_{n-1}\circ_1 m_2)\circ_1 R_1 \ \ \mathrm{and}\ \ (S_{n-1}\circ_1 m_2)\circ_1 R_1, $$
	which are illustrated above.
	It is  easily seen that the coefficients $l_\mathcal{S}$ are always $\pm 1$.

	\begin{defn}\label{Def: effective tree monomials}  A tree monomial $\mathcal{T}$ in ${}_m\RBS_\infty$ is called effective if $\mathcal{T}$ satisfies the following conditions:
		\begin{itemize}
			\item[(i)] There exists a typical divisor $\mathcal{T}'=\widehat{\mathcal{S}}$ in $\mathcal{T}$ such that  on the path from the root of $\mathcal{T}'$ to the leftmost leaf $l$ of $\mathcal{T}$ above the root of $\mathcal{T}'$, there are no other typical divisors, and there are no vertex of positive degree on this path except the root of $\mathcal{T}'$ possibly.
			\item[(ii)] For any leaf $l'$ of $\mathcal{T}$ which lies on the left of $l$, there are no vertices of positive degree and no typical divisors  on the path from the root of $\mathcal{T}$ to $l'$.
		\end{itemize}
		The typical divisor $\mathcal{T}'$ is called the effective divisor of $\mathcal{T}$ and  the leaf $l$ is called the typical leaf of $\mathcal{T}$.
	\end{defn}

	Morally, the effective divisor of a tree monomial $\mathcal{T}$ is the left-upper-most typical divisor of $\mathcal{T}$.
	It can be easily seen that for the effective divisor $\mathcal{T}'$ in $\mathcal{T}$ with effective leaf $l$, any vertex in $\mathcal{T}'$ doesn't belong to the path from root of $\mathcal{T}$ to any leaf $l'$ located on the left of $l$.
	\begin{exam}Consider three tree monomials as follows:
	
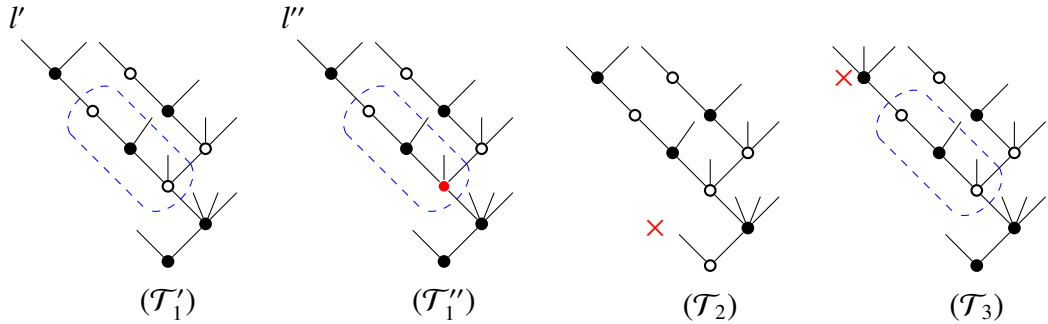
\begin{eqnarray*}
		\begin{tikzpicture}
			\tikzstyle{every node}=[thick,minimum size=4pt, inner sep=1pt]
			\node(1)at (0,0)[circle, draw, fill=black]{};
			\node(2-1) at (-0.5,0.5){};
			\node(2-2) at (0.5,0.5)[circle,draw, fill=black]{};
			\node(3-1) at (0,1)[circle, draw]{};
			\node(3-2) at (0.3,1){};
			\node(3-3) at (0.7,1){};
			\node(3-4) at (1,1){};
			\node(4-1) at (-0.5,1.5)[circle, draw,fill=black]{};
			\node (4-2) at(0,1.5){};
			\node(5-1) at (-1,2)[circle, draw]{};
			\node(5-2) at(0.5,1.5)[circle, draw] {};
	     	\node(8-2) at(-0.15,2){};
			\node (6-1) at (-1.5, 2.5)[circle,draw, fill=black]{};
			\node(6-2) at (0, 2)[circle, draw, fill=black]{};
			\node(6-3) at(0.5,2){};
			\node(6-4) at (1,2){};
			\node(7-1) at (-2,3)[minimum size=0pt, label=above:$l'$]{};
			\node(7-2) at (-1,3){};
			\node(7-3) at(-1,3){};
			\node(7-4) at (0.5,2.5){};
			\node(8-1) at (-0.5, 2.5) [circle, draw]{};
			\draw (1)--(2-1);
			\draw (1)--(2-2);
			\draw (2-2)--(3-1);
			\draw (2-2)--(3-2);
			\draw (2-2)--(3-3);
			\draw (2-2)--(3-4);
			\draw (3-1)--(4-1);
			\draw (3-1)--(4-2);
			\draw (4-1)--(8-2);
			\draw (3-1)--(5-2);
			\draw (4-1)--(5-1);
			\draw (5-1)--(6-1);
			\draw (6-1)--(7-1);
			\draw (6-1)--(7-2);
			\draw (5-2)--(6-2);
			\draw (5-2)--(6-3);
			\draw (5-2)--(6-4);
			\draw(6-2)--(7-4);
			\draw (6-2)--(8-1);
			\draw (8-1)--(7-3);
			\draw[dashed,blue](-1.3,1.7) to [in=150, out=120] (-0.7,2.3) ;
			\draw[dashed,blue](-1.3,1.7)--(-0.3,0.7);
			\draw[dashed,blue] (0.3,1.3)to [in=-30, out=-60] (-0.3,0.7);
			\node[minimum size=0pt,inner sep=0pt,label=below:$(\mathcal{T}_1')$] (name) at (0,-0.3){};
			\draw[dashed, blue](0.3,1.3)--(-0.7,2.3);
		\end{tikzpicture}
		\hspace{4mm}
		\begin{tikzpicture}
		\tikzstyle{every node}=[thick,minimum size=4pt, inner sep=1pt]
		\node(1)at (0,0)[circle, draw, fill=black]{};
		\node(2-1) at (-0.5,0.5){};
		\node(2-2) at (0.5,0.5)[circle,draw, fill=black]{};
		\node(3-1) at (0,1)[circle,fill=red]{};
		\node(3-2) at (0.3,1){};
		\node(3-3) at (0.7,1){};
		\node(3-4) at (1,1){};
		\node(4-1) at (-0.5,1.5)[circle, draw,fill=black]{};
		\node (4-2) at(0,1.5){};
		\node(5-1) at (-1,2)[circle, draw]{};
		\node(5-2) at(0.5,1.5)[circle, draw] {};
		\node(8-2) at(-0.15,2){};
		\node (6-1) at (-1.5, 2.5)[circle,draw, fill=black]{};
		\node(6-2) at (0, 2)[circle, draw, fill=black]{};
		\node(6-3) at(0.5,2){};
		\node(6-4) at (1,2){};
		\node(7-1) at (-2,3)[minimum size=0pt, label=above:$l''$]{};
		\node(7-2) at (-1,3){};
		\node(7-3) at(-1,3){};
		\node(7-4) at (0.5,2.5){};
		\node(8-1) at (-0.5, 2.5) [circle, draw]{};
		\draw (1)--(2-1);
		\draw (1)--(2-2);
		\draw (2-2)--(3-1);
		\draw (2-2)--(3-2);
		\draw (2-2)--(3-3);
		\draw (2-2)--(3-4);
		\draw (3-1)--(4-1);
		\draw (3-1)--(4-2);
		\draw (4-1)--(8-2);
		\draw (3-1)--(5-2);
		\draw (4-1)--(5-1);
		\draw (5-1)--(6-1);
		\draw (6-1)--(7-1);
		\draw (6-1)--(7-2);
		\draw (5-2)--(6-2);
		\draw (5-2)--(6-3);
		\draw (5-2)--(6-4);
		\draw(6-2)--(7-4);
		\draw (6-2)--(8-1);
		\draw (8-1)--(7-3);
		\draw[dashed,blue](-1.3,1.7) to [in=150, out=120] (-0.7,2.3) ;
		\draw[dashed,blue](-1.3,1.7)--(-0.3,0.7);
		\draw[dashed,blue] (0.3,1.3)to [in=-30, out=-60] (-0.3,0.7);
		\node[minimum size=0pt,inner sep=0pt,label=below:$(\mathcal{T}_1'')$] (name) at (0,-0.3){};
		\draw[dashed, blue](0.3,1.3)--(-0.7,2.3);
	\end{tikzpicture}	\hspace{4mm}
		\begin{tikzpicture}
		\tikzstyle{every node}=[thick,minimum size=4pt, inner sep=1pt]
		\node(1)at (0,0)[circle, draw]{};
		\node(2-1) at (-0.5,0.5)[minimum size=0pt, label=left:$\red \times$]{};
		\node(2-1) at (-0.5,0.5){};
		\node(2-2) at (0.5,0.5)[circle,draw, fill=black]{};
		\node(3-1) at (0,1)[circle, draw]{};
		\node(3-2) at (0.3,1){};
		\node(3-3) at (0.7,1){};
		\node(3-4) at (1,1){};
		\node(4-1) at (-0.5,1.5)[circle, draw,fill=black]{};
		\node (4-2) at(0,1.5){};
		\node(5-1) at (-1,2)[circle, draw]{};
		\node(5-2) at(0.5,1.5)[circle, draw] {};
		\node(8-2) at(-0.15,2){};
		\node (6-1) at (-1.5, 2.5)[circle,draw, fill=black]{};
		\node(6-2) at (0, 2)[circle, draw, fill=black]{};
		\node(6-3) at(0.5,2){};
		\node(6-4) at (1,2){};
		\node(7-1) at (-2,3){};
		\node(7-2) at (-1,3){};
		\node(7-3) at(-1,3){};
		\node(7-4) at (0.5,2.5){};
		\node(8-1) at (-0.5, 2.5) [circle, draw]{};
		\draw (1)--(2-1);
		\draw (1)--(2-2);
		\draw (2-2)--(3-1);
		\draw (2-2)--(3-2);
		\draw (2-2)--(3-3);
		\draw (2-2)--(3-4);
		\draw (3-1)--(4-1);
		\draw (3-1)--(4-2);
		\draw (4-1)--(8-2);
		\draw (3-1)--(5-2);
		\draw (4-1)--(5-1);
		\draw (5-1)--(6-1);
		\draw (6-1)--(7-1);
		\draw (6-1)--(7-2);
		\draw (5-2)--(6-2);
		\draw (5-2)--(6-3);
		\draw (5-2)--(6-4);
		\draw(6-2)--(7-4);
		\draw (6-2)--(8-1);
		\draw (8-1)--(7-3);
		\node[minimum size=0pt,inner sep=0pt,label=below:$(\mathcal{T}_2)$] (name) at (0,-0.3){};
	\end{tikzpicture}
		\hspace{4mm}	\begin{tikzpicture}
			\tikzstyle{every node}=[thick,minimum size=4pt, inner sep=1pt]
			\node(1)at (0,0)[circle, draw, fill=black]{};
			\node(2-1) at (-0.5,0.5){};
			\node(2-2) at (0.5,0.5)[circle,draw, fill=black]{};
			\node(3-1) at (0,1)[circle, draw]{};
			\node(3-2) at (0.3,1){};
			\node(3-3) at (0.7,1){};
			\node(3-4) at (1,1){};
			\node(4-1) at (-0.5,1.5)[circle, draw,fill=black]{};
			\node (4-2) at(0,1.5){};
			\node(5-1) at (-1,2)[circle, draw]{};
			\node(5-2) at(0.5,1.5)[circle, draw]{};	
			\node (6-1) at (-1.5, 2.5)[circle,draw, fill=black,label=left:$\red \times$]{};
			\node(8-2) at(-0.15,2){};
			\node (6-1) at (-1.5, 2.5)[circle,draw, fill=black]{};
			\node(6-2) at (0, 2)[circle, draw, fill=black]{};
			\node(6-3) at(0.5,2){};
			\node(6-4) at (1,2){};
			\node(7-1) at (-2,3){};
			\node(7-2) at (-1,3){};
			\node(7-3) at(-1,3){};
			\node(7-4) at (0.5,2.5){};
			\node(7-5) at (-1.5,3){};
			\node(8-1) at (-0.5, 2.5) [circle, draw]{};
			\draw (1)--(2-1);
			\draw (1)--(2-2);
			\draw (2-2)--(3-1);
			\draw (2-2)--(3-2);
			\draw (2-2)--(3-3);
			\draw (2-2)--(3-4);
			\draw (3-1)--(4-1);
			\draw (3-1)--(4-2);
			\draw (4-1)--(8-2);
			\draw (3-1)--(5-2);
			\draw (4-1)--(5-1);
			\draw (5-1)--(6-1);
			\draw (6-1)--(7-1);
			\draw (6-1)--(7-2);
			\draw (5-2)--(6-2);
			\draw (5-2)--(6-3);
			\draw (5-2)--(6-4);
			\draw(6-2)--(7-4);
			\draw (6-2)--(8-1);
			\draw (8-1)--(7-3);
		 \draw (6-1)--(7-5);
			\draw[dashed,blue](-1.3,1.7) to [in=150, out=120] (-0.7,2.3) ;
			\draw[dashed,blue](-1.3,1.7)--(-0.3,0.7);
			\draw[dashed,blue] (0.3,1.3)to [in=-30, out=-60] (-0.3,0.7);
			\node[minimum size=0pt,inner sep=0pt,label=below:$(\mathcal{T}_3)$] (name) at (0,-0.3){};
			\draw[dashed, blue](0.3,1.3)--(-0.7,2.3);
		\end{tikzpicture}
	\end{eqnarray*}

	For the three trees displayed above, each has two typical divisors.
	\begin{itemize}
		\item $\mathcal{T}_1'$ and $\mathcal{T}_1''$ are effective and the divisors in the blue  dashed circle are their effective divisor, $l'$ and $l''$ are respectively their effective leafs.
		\item $\mathcal{T}_2$ is not effective, since the first leaf is incident  to a vertex of degree 1, say the root of $\mathcal{T}_2$, which violates Condition (ii) in Definition \ref{Def: effective tree monomials}.
		\item $\mathcal{T}_3$ is not effective since there is a vertex of degree 1 on the path from the root of the typical divisor in the blue  dashed circle to the leftmost leaf above it, which violates Condition (i) in Definition \ref{Def: effective tree monomials}.
	\end{itemize}
\end{exam}

	Now we are going to construct a homotopy map $\overline{\mathfrak{H}}: _m\RBS_\infty\rightarrow _m\RBS_\infty$,  i.e.,  a degree $1$ map that satisfies  $\bar{\partial} \overline{\mathfrak{H}}+\overline{\mathfrak{H}}\bar{\partial} =\mathrm{Id}$ in positive degrees.

	\begin{defn}
		Let $\mathcal{T}$ be an effective tree monomial in $ _m\RBS_\infty$ and $\mathcal{T}'$ be its effective divisor. Assume that $\mathcal{T}'=\widehat{\mathcal{S}}$, where $\mathcal{S}$ is a generator of positive degree. Then define $$\overline{\mathfrak{H}}(\mathcal{T})=(-1)^\omega \frac{1}{l_\mathcal{S}}m_{\mathcal{T}', \mathcal{S}}(\mathcal{T}),$$ where $m_{\mathcal{T}',\mathcal{S}}(\mathcal{T})$ is the tree monomial obtained from $\mathcal{T}$ by replacing the effective divisor $\mathcal{T}'$ by $\mathcal{S}$, $\omega$ is the sum of degrees of all the vertices on the path from root of $\mathcal{T}'$ to the root of $\mathcal{T}$  (excluding the root vertex of $\mathcal{T}'$) and all the vertices located on the left of this path .
	\end{defn}
	
	\medskip

\begin{prop}
		The degree $1$ map $\overline{\mathfrak{H}}: {_m\RBS_\infty}\rightarrow {_m\RBS_\infty}$  is a homotopy map satisfying  $\bar{\partial} \overline{\mathfrak{H}}+\overline{\mathfrak{H}}\bar{\partial} =\mathrm{Id}$ in all positive degrees.
	
\end{prop}

\begin{proof}We divide this proof into two cases: whether $\mathcal{T}$ is an effective tree or not.
	
	 If $\mathcal{T}$ is a tree monomial of positive degree that is not effective, then by the definition of $\overline{\mathfrak{H}}$, we have $\overline{\mathfrak{H}}(\mathcal{T})=0$. Thus, $\bar{\partial} \overline{\mathfrak{H}}(\mathcal{T})+\overline{\mathfrak{H}}\bar{\partial}(\mathcal{T}) =\overline{\mathfrak{H}}\bar{\partial}(\mathcal{T}).$  We firstly just need to check that $\overline{\mathfrak{H}}\bar{\partial}(\mathcal{T})=\mathcal{T}$.	Since $\mathcal{T}$ has   positive degree, there must exist at least one vertex of positive degree. Let's pick a special vertex $\mathcal{S}$ satisfying  the following conditions:
	\begin{itemize}
		\item[(i)] on the path from $\mathcal{S}$ to the leftmost leaf $l$ of $\mathcal{T}$ above $\mathcal{S}$, there are no other vertices of positive degree;
		\item[(ii)] for any leaf $l'$ of $\mathcal{T}$ located on the left of $l$, the vertices  on the path from the root of $\mathcal{T}$ to $l'$ are all of degree 0.
	\end{itemize}	
	It is evident that such a vertex always exists in $\mathcal{T}$. Conceptually, this vertex is the ``left-upper-most" vertex of positive degree. Denote by $\mathcal{T}_{\mathcal{S}\rightarrow \widehat{\mathcal{S}}}$ the tree monomial obtained from $\mathcal{T}$ by replacing the vertex $\mathcal{S}$ with $\widehat{\mathcal{S}}$.
	Then there will be a tree monomial $(-1)^\epsilon  l_\mathcal{S}   \mathcal{T}_{\mathcal{S}\rightarrow \widehat{\mathcal{S}}}$ appearing in $\bar\partial\mathcal{T}$, where $\epsilon$ is the sum of degrees of vertices located preceding $\mathcal{S}$ in the planar order. Due to the special position of the vertex $\mathcal{S}$, we have that the tree monomial $(-1)^\epsilon  l_\mathcal{S}  \mathcal{T}_{\mathcal{S}\rightarrow \widehat{\mathcal{S}}}$ is an effective tree monomial. Actually, this term is the only effective tree monomial appearing in $\bar\partial(\mathcal{T})$, i.e., there is no effective tree monomial in $\bar{\partial}(\mathcal{T})-(-1)^\epsilon  l_\mathcal{S} \mathcal{T}_{\mathcal{S}\rightarrow \widehat{\mathcal{S}}}$.
	So we have \begin{eqnarray}\overline{\mathfrak{H}}\bar{\partial}(\mathcal{T})&=&\overline{\mathfrak{H}}\big((-1)^\epsilon  l_\mathcal{S}  \mathcal{T}_{\mathcal{S}\rightarrow \widehat{\mathcal{S}}}\big)+\overline{\mathfrak{H}}\big(\bar{\partial}\mathcal{T}-(-1)^\epsilon  l_\mathcal{S}  \mathcal{T}_{\mathcal{S}\rightarrow \widehat{\mathcal{S}}}\big)\\
		&=&\mathcal{T}.\nonumber
		\end{eqnarray}

	 	If $\mathcal{T}$ is  an effective tree monomial with positive degree, we write $\mathcal{T}$ as a compositions  in the following way:
	 $$(\cdots (((((\cdots(X_1\circ_{i_1}X_2)\circ\cdots )\circ_{i_{p-1}}X_p)\circ_{i_p} \widehat{\mathcal{S}})\circ_{j_1}Y_1)\circ_{j_2}Y_2)\cdots)\circ_{j_q}Y_q,$$
	 where $\widehat{\mathcal{S}}$ is the effective divisor of $\mathcal{T}$ and $X_1,\dots, X_p$ are generators of $_m\RBS_\infty$ corresponding to the vertices which live on the path from root of $\mathcal{T}$ and root of $\widehat{\mathcal{S}}$ (except the root of $\widehat{\mathcal{S}}$) and on the left of this path in the underlying  tree of  $\mathcal{T}$.
	 { It is not difficult to see the following identity holds for effective divisors: }
	  $$\bar{\mathfrak{H}}\bar\partial(\widehat{\cals})=\widehat{\cals}-\frac{1}{l_\mathcal{S}}\bar\partial(\cals).$$
	  Then by definition,
	 { \small{
	 \begin{eqnarray*}
	 	&&\bar{\partial} \overline{\mathfrak{H}}(\mathcal{T})\\
	 	&=&
	 	\frac{1}{l_\mathcal{S}}(-1)^{\sum\limits_{j=1}^p|X_j|}\bar{\partial}
	 	\Big((\cdots ((((\cdots ((X_1\circ_{i_1}X_2)\circ_{i_2}\cdots )\circ_{i_{p-1}}X_p)\circ_{i_p} \mathcal{S})\circ_{j_1}Y_1)\circ_{j_2} \cdots)\circ_{j_q}Y_q\Big)\\
	 	&=&\frac{1}{l_\mathcal{S}}\Big(\sum_{k=1}^p(-1)^{\sum\limits_{j=1}^p|X_j|+\sum\limits_{j=1}^{k-1}|X_j|}\\
	 	&&(\cdots ((((\cdots ((\cdots(X_1\circ_{i_1} X_2)\circ_{i_2} \cdots ) \circ_{i_{k-1}}\bar{\partial} X_k )\circ_{i_k} \cdots)\circ_{i_{p-1}}X_p)\circ_{i_p} \mathcal{S})\circ_{j_1}Y_1)\circ_{j_2} \cdots)\circ_{j_q}Y_q\Big)\\
	 	&&+\frac{1}{l_\mathcal{S}}\Big((\cdots ((( (\cdots\cdots(X_1\circ_{i_1}X_2)\circ_{i_2}\cdots )\circ_{i_{p-1}}X_p)\circ_{i_p} \bar{\partial} \mathcal{S})\circ_{j_1}Y_1)\circ_{j_2} \cdots)\circ_{j_q}Y_q\Big)\\
	 	&&+\frac{1}{l_\mathcal{S}}\Big(\sum_{k=1}^q(-1)^{|\mathcal{S}|+\sum\limits_{j=1}^{k-1}|Y_j|}\\
	 	&&(\cdots ((\cdots((((\cdots(X_1\circ_{i_1}X_2)\circ\cdots )\circ_{i_{p-1}}X_p)\circ_{i_p} \mathcal{S})\circ_{j_1}Y_1)\circ_{j_2} \cdots ) \circ_{j_k}\bar{\partial} Y_{k})\circ_{j_{k+1}}\cdots )\circ_{j_q}Y_q\Big)
	 \end{eqnarray*}
	}}
 and
 	{\small{\begin{eqnarray*}\small
 	&&\overline{\mathfrak{H}}\bar{\partial} (\mathcal{T} )\\
 	&=& \overline{\mathfrak{H}}\bar{\partial} \Big((\cdots((((\cdots(X_1\circ_{i_1}X_2)\circ_{i_2}\cdots )\circ_{i_{p-1}}X_p)\circ_{i_p} \widehat{\mathcal{S}})\circ_{j_1}Y_1)\circ_{j_2} \cdots)\circ_{j_q}Y_q\Big)\\
 	&=& \overline{\mathfrak{H}}\Big(\sum\limits_{k=1}^p(-1)^{\sum\limits_{j=1}^{k-1}|X_j|}
 	(\cdots((((\cdots  ((\cdots(X_1\circ_{i_1} \cdots )\circ_{i_{k-1}}\bar{\partial} X_k)\circ_{i_k}\cdots)\circ_{i_{p-1}}X_p)\circ_{i_p} \widehat{\mathcal{S}})\circ_{j_1}Y_1)\circ_{j_2} \cdots)\circ_{j_q}Y_q\Big)\\
 	&&+  \sum_{k=1}^q(-1)^{\sum\limits_{j=1}^p|X_j| }    \overline{\mathfrak{H}} \Big((\cdots((((\cdots(X_1\circ_{i_1}X_2)\circ_{i_2}\cdots )\circ_{i_{p-1}}X_p)\circ_{i_p} \bar{\partial}\widehat{\mathcal{S}})\circ_{j_1}Y_1)\circ_{j_2} \cdots)\circ_{j_q}Y_q\Big) \\
 	&&+ \overline{\mathfrak{H}}\Big(\sum_{k=1}^q(-1)^{\sum\limits_{j=1}^p|X_j|+|\mathcal{S}|-1+\sum_{j=1}^{k-1}|Y_j|}\\
 	&&(\cdots((\cdots ((((\cdots(X_1\circ_{i_1}X_2)\circ_{i_2}\cdots )\circ_{i_{p-1}}X_p)\circ_{i_p} \widehat{\mathcal{S}})\circ_{j_1}Y_1)\circ_{j_2} \cdots)\circ_{j_k}\bar{\partial} Y_k)\circ_{j_{k+1}}\cdots\circ_{j_q}Y_q\Big)\\
 		&=&\frac{1}{l_\mathcal{S}}\Big(\sum_{k=1}^p(-1)^{\sum\limits_{j=1}^p|X_j|+\sum\limits_{j=1}^{k-1}|X_j|+1}\\
 	&&(\cdots ((((\cdots ((\cdots(X_1\circ_{i_1} X_2)\circ_{i_2} \cdots ) \circ_{i_{k-1}}\bar{\partial} X_k )\circ_{i_k} \cdots)\circ_{i_{p-1}}X_p)\circ_{i_p} \mathcal{S})\circ_{j_1}Y_1)\circ_{j_2} \cdots)\circ_{j_q}Y_q\Big)\\
 		&&+  \sum_{k=1}^q(-1)^{\sum\limits_{j=1}^p|X_j| }    \overline{\mathfrak{H}} \Big((\cdots((((\cdots(X_1\circ_{i_1}X_2)\circ_{i_2}\cdots )\circ_{i_{p-1}}X_p)\circ_{i_p} \bar{\partial}(\widehat{\mathcal{S}}-\bar{\partial} \mathcal{S}))\circ_{j_1}Y_1)\circ_{j_2} \cdots)\circ_{j_q}Y_q\Big) \\
 		&&+ \frac{1}{l_\mathcal{S}}\Big(\sum_{k=1}^q(-1)^{\sum\limits_{j=1}^p|X_j|+|\mathcal{S}|-1+\sum_{j=1}^{k-1}|Y_j|}\\
 	&&(\cdots((\cdots ((((\cdots(X_1\circ_{i_1}X_2)\circ_{i_2}\cdots )\circ_{i_{p-1}}X_p)\circ_{i_p}  \mathcal{S} )\circ_{j_1}Y_1)\circ_{j_2} \cdots)\circ_{j_k}\bar{\partial} Y_k)\circ_{j_{k+1}}\cdots\circ_{j_q}Y_q\Big)\\	&=&\frac{1}{l_\mathcal{S}}\Big(\sum_{k=1}^p(-1)^{\sum\limits_{j=1}^p|X_j|+\sum\limits_{j=1}^{k-1}|X_j|+1}\\
 	&&(\cdots ((((\cdots ((\cdots(X_1\circ_{i_1} X_2)\circ_{i_2} \cdots ) \circ_{i_{k-1}}\bar{\partial} X_k )\circ_{i_k} \cdots)\circ_{i_{p-1}}X_p)\circ_{i_p} \mathcal{S})\circ_{j_1}Y_1)\circ_{j_2} \cdots)\circ_{j_q}Y_q\Big)\\
 	&&+  \sum_{k=1}^q(-1)^{\sum\limits_{j=1}^p|X_j| }      \frac{1}{l_\mathcal{S}}\Big((\cdots((((\cdots(X_1\circ_{i_1}X_2)\circ_{i_2}\cdots )\circ_{i_{p-1}}X_p)\circ_{i_p}  \widehat{\mathcal{S}}-\bar{\partial} \mathcal{S}) \circ_{j_1}Y_1)\circ_{j_2} \cdots)\circ_{j_q}Y_q\Big) \\
 	&&+ \frac{1}{l_\mathcal{S}}\Big(\sum_{k=1}^q(-1)^{\sum\limits_{j=1}^p|X_j|+|\mathcal{S}|-1+\sum_{j=1}^{k-1}|Y_j|}\\
 	&&(\cdots((\cdots ((((\cdots(X_1\circ_{i_1}X_2)\circ_{i_2}\cdots )\circ_{i_{p-1}}X_p)\circ_{i_p}  \mathcal{S} )\circ_{j_1}Y_1)\circ_{j_2} \cdots)\circ_{j_k}\bar{\partial} Y_k)\circ_{j_{k+1}}\cdots\circ_{j_q}Y_q\Big),\\
 \end{eqnarray*}
}}
thus, we have $\bar{\partial} \overline{\mathfrak{H}}(\mathcal{T})+\overline{\mathfrak{H}}\bar{\partial}(\mathcal{T} )=\mathcal{T} $.
\end{proof}
	
	\begin{thm}\label{Thm: inclusion-exclusion model for the monomials of RBS}
		${}_m\RBS_\infty$ is the minimal model of $_m\RBS$.
	\end{thm}

	\begin{proof}
		By direct calculation, we have $\rmH_0({}_m\RBS_\infty)  \cong {}_m\RBS $. And the homology of ${}_m\RBS_\infty$ concentrates in degree 0.
	\end{proof}
	
	\begin{remark}
		In   \cite{DK13}, the authors constructed an inclusion-exclusion model for the operad  with  monomial relations.     The  minimal model ${}_m\RBS_\infty$   is  essentially the inclusion-exclusion model for the  monomial operad $_m\RBS$.
	\end{remark}

\bigskip

\textbf{Acknowledgements}

This work was supported by  the National Key R  $\&$ D Program of China (No. 2024YFA1013803),  by Key Laboratory of Mathematics and Engineering Applications  (Ministry of Education), and by Shanghai Key Laboratory of PMMP (No.
22DZ2229014).

\textbf{Conflict of Interest}

None of the authors has any conflict of interest in the conceptualization or publication of this
work.

\textbf{Data availability}

Data sharing is not applicable to this article as no new data were created or analyzed in this study.

\bigskip

\end{document}